\documentclass[12pt]{amsart}

\usepackage{frontmatter}

\begin{document}

\title[Micropterons in Diatomic FPUT Lattices]{Micropteron Traveling Waves in Diatomic Fermi-Pasta-Ulam-Tsingou Lattices under the Equal Mass Limit}

\author{Timothy E. Faver}
\address{Mathematical Institute, University of Leiden, P.O. Box 9512, 2300 RA Leiden, The Netherlands, t.e.faver@math.leidenuniv.nl}

\author{Hermen Jan Hupkes}
\address{Mathematical Institute, University of Leiden, P.O. Box 9512, 2300 RA Leiden, The Netherlands, hhupkes@math.leidenuniv.nl}


\thanks{We thank Doug Wright for originally suggesting this problem to us and for his helpful comments and perspectives.}


\maketitle

\begin{abstract}
The diatomic Fermi-Pasta-Ulam-Tsingou (FPUT) lattice is an infinite chain of alternating particles connected by identical nonlinear springs.  
We prove the existence of micropteron traveling waves in the diatomic FPUT lattice in the limit as the ratio of the two alternating masses approaches 1, at which point the diatomic lattice reduces to the well-understood monatomic FPUT lattice.  
These are traveling waves whose profiles asymptote to a small periodic oscillation at infinity, instead of vanishing like the classical solitary wave.  
We produce these micropteron waves using a functional analytic method, originally due to Beale, that was successfully deployed in the related long wave and small mass diatomic problems.
Unlike the long wave and small mass problems, this equal mass problem is not singularly perturbed, and so the amplitude of the micropteron's oscillation is not necessarily small beyond all orders (i.e., the traveling wave that we find is not necessarily a nanopteron).  
The central challenge of this equal mass problem hinges on a hidden solvability condition in the traveling wave equations, which manifests itself in the existence and fine properties of asymptotically sinusoidal solutions (Jost solutions) to an auxiliary advance-delay differential equation.
\end{abstract}

\section{Introduction}

\subsection{The diatomic FPUT lattice}
A diatomic Fermi-Pasta-Ulam-Tsingou (FPUT) lattice is an infinite one-dimensional chain of particles of alternating masses connected by identical springs.
These lattices, also called mass dimers, are a material generalization of the finite lattice of identical particles studied numerically by Fermi, Pasta, and Ulam \cite{fput-original} and Tsingou \cite{dauxois}; such lattices are valued in applications as models of wave propagation in discrete and granular materials \cite{brillouin, kevrekidis}.  

We index the particles and their masses by $j \in \Z$ and let $u_j$ denote the position of the $j$th particle.
After a routine nondimensionalization, we may assume that the $j$th particle has mass 
\begin{equation}\label{masses}
m_j 
= \begin{cases}
1, &j \text{ is odd} \\
m, &j \text{ is even}
\end{cases}
\end{equation}
and that each spring exerts the force $F(r) = r+r^2$ when stretched a distance $r$ from its equilibrium length.
Newton's law then implies that the position functions $u_j$ satisfy the system
\begin{equation}\label{original equations of motion}
\begin{cases}
\ddot{u}_j = F(u_{j+1}-u_j)-F(u_j-u_{j-1}), &j \text{ is odd} \\
\\
m\ddot{u}_j = F(u_{j+1}-u_j)-F(u_j-u_{j-1}), &j \text{ is even}.
\end{cases}
\end{equation}
We sketch a diatomic FPUT lattice in Figure \ref{fig-mass dimer}.

\begin{figure}
\[
\input{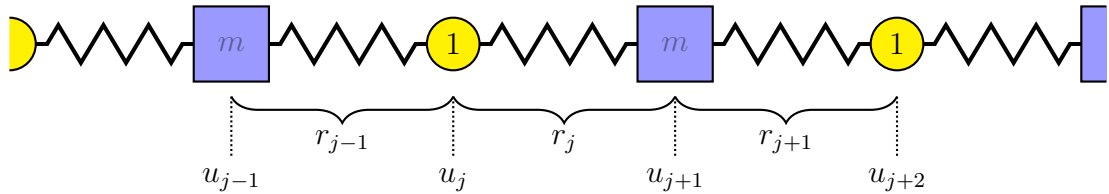}
\]
\caption{The mass dimer ($r_j := u_{j+1}-u_j$)}
\label{fig-mass dimer}
\end{figure}

We are interested in the existence and properties of traveling waves in diatomic lattices in the limit as the ratio of the two alternating masses approaches 1. 
Specifically, we define the relative displacement between the $j$th and the $(j+1)$st particle to be
\[
r_j := u_{j+1}-u_j,
\]
and then we make the traveling wave ansatz
\begin{equation}\label{tw ansatz}
r_j(t) = \begin{cases}
p_1(j-ct), &j \text{ is even} \\
p_2(j-ct), &j \text{ is odd}.
\end{cases}
\end{equation}
Here $p_1$ and $p_2$ are the traveling wave profiles and $c \in \R$ is the wave speed.

When the masses are identical, the lattice reduces to a {\it{monatomic}} lattice, which, due to the work of Friesecke and Wattis \cite{friesecke-wattis} and Friesecke and Pego \cite{friesecke-pego1,friesecke-pego2,friesecke-pego3,friesecke-pego4} is known to bear solitary traveling waves, i.e., waves whose profiles vanish exponentially fast at spatial infinity; see also Pankov \cite{pankov} for a comprehensive overview of monatomic traveling waves.
Our interest, then, is to determine how the monatomic solitary traveling wave changes when the mass ratio is close to 1.

\subsection{Parameter regimes}
We derive our motivation for this equal mass situation from two recent papers studying traveling waves in diatomic lattices under different limits.
Faver and Wright \cite{faver-wright} fix the mass ratio\footnote{They work with $w := 1/m > 1$; after rescaling and relabeling the lattice, we may equivalently think of their results for $m \in (0,1)$.} and consider the long wave limit, i.e., they look for traveling waves where the wave speed is close to a special $m$-dependent threshold called the ``speed of sound'' and where the traveling wave profile is close to a certain KdV $\sech^2$-type soliton.  
Hoffman and Wright \cite{hoffman-wright} fix the wave speed and consider the small mass limit, in which the ratio of the alternating masses approaches zero, thereby reducing the lattice from diatomic to monatomic\footnote{This is not the same monatomic lattice that results from the equal mass limit; the springs in this small mass limiting lattice are double the length of the original springs in the diatomic lattice, and so they exert twice the original force.
This factor of 2 ends up affecting the wave speed of the traveling waves that Hoffman and Wright construct: theirs have speed close to $\sqrt{2}$.}.
Figure \ref{fig-all the results} sketches the bands of long wave (in yellow) and small mass (in red) traveling waves and indicates, roughly, how they depend on wave speeds and mass ratios.

In both problems, the solitary wave that exists in the limiting case perturbs into a traveling wave whose profile asymptotes to a small amplitude periodic oscillation or ``ripple.'' 
That is, the wave is not ``localized'' in the ``core'' of the classical solitary wave, and so, per Boyd \cite{boyd}, it is a {\it{nonlocal}} solitary wave.
Moreover, the long wave and small mass problems are singularly perturbed, which causes the amplitude of their periodic ripples to be small beyond all orders of the long wave/small mass parameter.
So, these nonlocal traveling waves are, in Boyd's parlance, {\it{nanopterons}}; see \cite{boyd} for an overview of the nanopteron's many incarnations in applied mathematics and nature. 

The question of the equal mass limit then follows naturally from the success of these two studies.
It was raised in the conclusion of \cite{faver-wright}, where the authors wondered if the diatomic long wave profiles would converge to those found by Friesecke and Pego in the monatomic long wave limit \cite{friesecke-pego1}, and appears as far back as Brillouin's book \cite{brillouin}, which examines both the small and equal mass limits for lattices with linear spring forces.

The articles \cite{faver-wright} and \cite{hoffman-wright} both derive their nanopteron traveling waves via a method due to Beale \cite{beale2} for a capillary-gravity water wave problem.
Beale's method was later adapted by Amick and Toland \cite{amick-toland} for a model  singularly perturbed KdV-type fourth order equation.
More recently, Faver \cite{faver-spring-dimer, faver-dissertation} used Beale's method to study the long wave problem in spring dimer lattices (FPUT lattices with alternating spring forces but constant masses), and Johnson and Wright \cite{johnson-wright} adapted it for a singularly perturbed Whitham equation.

Although the equal mass problem is not, ultimately, singularly perturbed, its structure has enough in common with these predecessors that we are able to adapt Beale's method to this situation, too.
We discover that the monatomic traveling wave perturbs into a ``micropteron'' traveling wave as the mass ratio hovers around 1.
Boyd uses this term to refer to a nonlocal solitary wave whose ripples are only algebraically small in the relevant small parameter, not small beyond {\it{all}} algebraic orders.
We do not find such small beyond all orders estimates in our equal mass problem, and so we consciously avoid using the term ``nanopteron'' for our profile.

We sketch this micropteron wave in Figure \ref{fig-nonlocal sol wave} and provide in Figure \ref{fig-all the results} an informal, but evocative, cartoon comparing the three families of nonlocal solitary waves (long wave, small mass, equal mass) that now exist for the diatomic FPUT lattice.
We state our main result below in Theorem \ref{main theorem - informal}.

Beyond these nanopteron problems, wave propagation in diatomic and more generally heterogeneous lattices has received considerable recent attention; we mention, among others, the papers \cite{chirilus-bruckner-etal, betti-pelinovsky, qin,vswp,sv,wattis,lustri} for theoretical and numerical examples of waves in FPUT lattices under different material regimes.
See \cite{gmwz} for a discussion of long wave KdV approximations in polyatomic FPUT lattices and \cite{co-ops} for a further discussion of the metastability of these waves in diatomic lattices.
The paper \cite{hmsz} studies a regular perturbation problem in monatomic FPUT lattices in which the spring force is perturbed from a known piecewise quadratic potential; the resulting solutions are asymptotic, like ours, to a sinusoid whose amplitude is algebraically small. 
Lattice differential equations abound in contexts beyond the FPUT model that we study here; see, for example, \cite{hmssvv} for a survey of traveling wave results for the Nagumo lattice equation and related models.

\begin{figure}
\[
%
%
%
\begin{tikzpicture}
%
%
\draw[thick,<->] (-7.5,0)--(7.5,0)node[right]{$x$};
\draw[thick,<->] (0,-.5)--(0,3);

\draw[thick] plot[domain=-7.5:7.5,smooth,samples=500] (\x,{2/(1+(\x)^2)+cos(deg(30*\x))/8});
\draw[densely dotted,very thick,blue] plot[domain=-7.5:7.5,smooth,samples=500] (\x,{2.35/(1+2*(\x)^2)});

\draw[blue,<-,thick] (-.25,2.25)--(-1.25,2.5)node[left]{$\varsigma_c$};

\draw[<-,thick] (1.5,.85)--(2.5,1.5);
\node[anchor=north west,outer sep=0pt, inner sep = 0pt] at (2.55,1.65)
{
\begin{minipage}{1.5in}
\setlength{\baselineskip}{0pt}
the nonlocal solitary \\
wave profile
\end{minipage}
};

\end{tikzpicture}
%
\]
\caption{One of the micropteron profiles ($p_1$ or $p_2$) sketched close to the monatomic solitary wave $\varsigma_c$.  
In a nanopteron, the periodic ripple would be so small as to be invisible relative to the monatomic profile; this is not the case in the equal mass limit.}
\label{fig-nonlocal sol wave}
\end{figure}
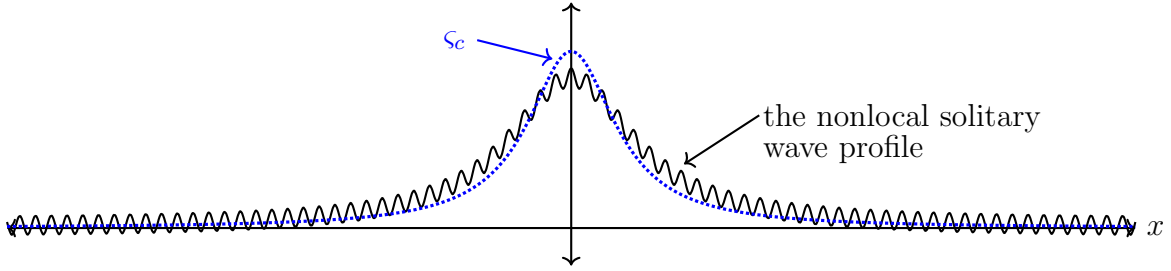

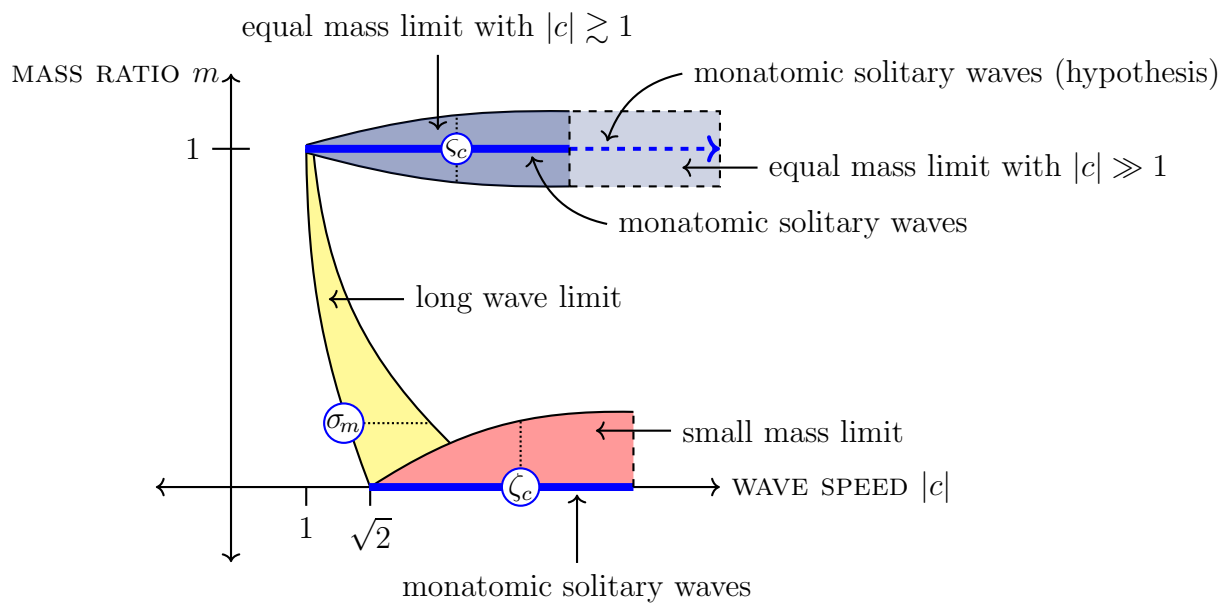
\begin{figure}
\[
\begin{tikzpicture}[thick]

\draw[<->] (0,-1)--(0,5.5)node[left]{\sc mass ratio $m$};
\draw[<->] (-1,0)--(6.5,0)node[right]{\sc wave speed $|c|$};

\draw (-.25,4.5)node[left]{1}--(.25,4.5);
\draw (1,0)--(1,-.25)node[below]{1};

\draw (1.85,0)--(1.85,-.25)node[below]{$\sqrt{2}$};


\begin{scope}

\clip (0,4.5)--(0,0)--(1.85,0) to[bend left=17] (4.5+.85,1)--(4.5+.85,4.5)--cycle;

\fill[yellow,opacity=.4] (1,4.5) to[bend right = 10] (1.85,0)--(3,0)--(3,.5) to[bend left = 20] (1.1,4.4);

\draw (1,4.5) to[bend right=10] (1.85,0);
\draw (1.1,4.415) to[bend right=20] (3,.5);

\end{scope}

\begin{scope}

\clip (1,4.5) to[bend right = 10] (1.85,0)--(3,.5) to[bend left = 20] (1,4.5);
\draw[densely dotted] (1,.85)--(4,.85);
\end{scope}

\node[circle,draw=blue,outer sep=0pt,inner sep = 0pt,fill=white] at (1.5,.85){$\sigma_{\!m}$};

\draw[<-] (1.3,2.5)--(2.3,2.5)node[right]{long wave limit};


\fill[red,opacity=.4] (1.85,0) to[bend left=17] (4.5+.85,1)--(4.5+.85,0)--cycle;

\draw (1.85,0) to[bend left=17] (4.5+.85,1);

\draw[blue,line width = 3pt] (1.85-.015,0)--(4.5+.85,0);
\draw[<-] (4+.85,.75)--(5+.85,.75)node[right]{small mass limit};
\draw[<-] (3.75+.85,-.1)--(3.75+.85,-1.01)node[below]{monatomic solitary waves};

\begin{scope}
\clip (1+.85,0) to[bend left=17] (4.5+.85,1)--(4.5+.85,0)--cycle;

\draw[densely dotted] (3+.85,1)--(3+.85,0);
\end{scope}

\node[circle,draw=blue,outer sep=0pt,inner sep = 0pt,fill=white] at (3+.85,0){$\zeta_c$};

\draw[dashed] (4.5+.85,0)--(4.5+.85,1);


\draw[thick] (1,4.55) to[bend left=8] (4.5,5);
\draw[thick] (1,4.45) to[bend right=8] (4.5,4);

\fill[leidenblue,opacity=.4] (1,4.55) to[bend left=8] (4.5,5)--(4.5,4) to[bend left=8] (1,4.45);
\fill[leidenblue,opacity=.2] (4.5,5) rectangle (6.5,4);

\draw[blue,line width = 3pt] (.985,4.5)--(4.5,4.5);
\draw[blue,dashed,->, line width = 1.5pt] (4.5,4.5)--(6.5,4.5);

\draw[dashed] (4.5,5) rectangle (6.5,4);
\draw[->] (2.75,5.75)node[above]{equal mass limit with $|c| \gtrsim 1$}--(2.75,4.75);
\draw[<-] (5,4.6)to[bend left] (6,5.5)node[right]{monatomic solitary waves (hypothesis)};
\draw[<-] (4,4.4) to[bend right] (5,3.5)node[right]{monatomic solitary waves};
\draw[<-] (6,4.25)--(7,4.25)node[right]{equal mass limit with $|c| \gg 1$};

\begin{scope}
\clip (1,4.55) to[bend left=8] (4.5,5)--(4.5,4) to[bend left=8] (1,4.45);
\draw[densely dotted] (3,4)--(3,5);
\end{scope}

\node[circle,draw=blue,outer sep=.25pt,inner sep = .25pt,fill=white] at (3,4.5){$\varsigma_c$};

\end{tikzpicture}
\]
\caption{A comparison of nonlocal solitary waves across the different diatomic FPUT problems.
In this paper, for $|c| \gtrsim 1$ fixed and $m \approx 1$, we find micropterons close to a Friesecke-Pego solitary wave profile $\varsigma_c$.
Under suitable hypotheses on the existence of a monatomic solitary wave for $|c| \gg 1$, we may still find micropterons close to that profile.
For $|c| \gtrsim \sqrt{2}$ fixed and $m \approx 0$, Hoffman and Wright found nanopterons close to a (different) Friesecke-Pego solitary wave $\zeta_c$.
For $m \in (0,1)$ fixed and $|c|$ close to an $m$-dependent threshold called the ``speed of sound,'' Faver and Wright found nanopterons close to a KdV $\sech^2$-type profile $\sigma_m$.
That the bands collapse as $|c| \to 1^+$ or $|c| \to \sqrt{2}^+$ is intentional; see (the proofs of) Lemma C.7 in \cite{faver-wright}, Lemma 7.1 in \cite{hoffman-wright}, and Lemma \ref{critical frequency technical lemma} in this paper.
This graphic elides the interesting, but difficult, question of how the different families of waves interact; for example, for $|c| \gtrsim 1$ and $m \approx 1$, how do the long wave nanopterons and equal mass micropterons compare?
See Section 7 in \cite{faver-wright} for a further discussion of the challenges involved in addressing these questions.}
\label{fig-all the results}
\end{figure}

\subsection{The traveling wave problem}
After making the traveling wave ansatz \eqref{tw ansatz} for the original equations of motion \eqref{original equations of motion}, we write $1/m = 1+\mu$ for $\mu \in (-1,1)$ and make the linear change of variables
\[
\rho_1 = \frac{p_1+p_2}{2}
\quadword{and}
\rho_2 = \frac{p_1-p_2}{2}
\]
to obtain the equivalent system
\begin{equation}\label{cov simplified}
\bunderbrace{c^2\rhob'' + \D_{\mu}\rhob + \D_{\mu}\nl(\rhob,\rhob)}{\G_c(\rhob,\mu)}
=0.
\end{equation}
The details of this change of variables are discussed in Appendix \ref{tw derivation}.
For now, we focus on the definitions and properties of the operators $\D_{\mu}$ and $\nl$.

First, let $S^d$ be the ``shift-by-$d$'' operator defined by $(S^df)(x) := f(x+d)$ and set
\[
A := S^1 + S^{-1},
\qquad 
\delta := S^1-S^{-1}.
\] 
Then we have
\begin{equation}\label{D-mu defn}
\D_{\mu}
:= \frac{1}{2}\begin{bmatrix*}
(2+\mu)(2-A) &\mu\delta \\
-\mu\delta &(2+\mu)(2+A)
\end{bmatrix*}.
\end{equation}
Next, we define
\begin{equation}\label{nl}
\nl(\rhob,\grave{\rhob}) 
= \begin{pmatrix*}
\nl_1(\rhob,\grave{\rhob}) \\
\nl_2(\rhob,\grave{\rhob})
\end{pmatrix*}
:= \begin{pmatrix*} 
\rho_1\grave{\rho}_1 + \rho_2\grave{\rho}_2 \\
\rho_1\grave{\rho}_2 + \grave{\rho}_1{\rho}_2
\end{pmatrix*}.
\end{equation}

The version \eqref{cov simplified} is particularly useful because it preserves a number of symmetries.
Namely, $\G_c$ maps
\begin{equation}\label{all the pretty symmetries}
\{\text{even functions}\} \times \{\text{odd functions}\} \times \R
\to \{\text{even mean-zero functions}\} \times \{\text{odd functions}\}.
\end{equation}
We prove these symmetries in Appendix \ref{tw derivation}.
We say that a function $f$ is ``mean-zero'' if $\hat{f}(0) = 0$; here $\hat{f}$ is the Fourier transform of $f$, and our conventions for the Fourier transform are outlined in Appendix \ref{fourier analysis appendix}.

When $\mu = 0$, the diatomic lattice reverts to a monatomic lattice, and the traveling wave problem \eqref{cov simplified} reduces to
\[
\begin{cases}
c^2\rho_1''+(2-A)(\rho_1+\rho_1^2+\rho_2^2) =0 \\
c^2\rho_2'' + (2+A)(2\rho_1\rho_2) =0.
\end{cases}
\]
If we take $\rho_2 = 0$, then the second equation is satisfied, and the first becomes
\[
c^2\rho_1'' + (2-A)(\rho_1+\rho_1^2) 
= 0.
\]
This is the equation for the traveling wave profile of a monatomic FPUT lattice.
For $|c| \gtrsim 1$, it has an even exponentially decaying (or localized) solution due to Friesecke and Pego \cite{friesecke-pego1}, which we call $\varsigma_c$.
We discuss the properties of $\varsigma_c$ in greater detail in Theorem \ref{friesecke-pego}.

\subsection{Linearizing at the Friesecke-Pego solution}\label{linearize at fp section}
If we set $\varsigmab_c := (\varsigma_c,0)$, we see that $\G_c(\varsigmab_c,0) = 0$.
Then for $\mu$ small, we are interested in solutions $\rhob$ to $\G_c(\rhob,\mu)$ that are close to the Friesecke-Pego solution $\varsigmab_c$.
In order to perturb from $\varsigmab_c$, we first define Sobolev spaces of exponentially localized functions.

\begin{definition}\label{Hrq defn}
Let
\begin{equation}\label{hrq}
H_q^r := \set{f \in H^r(\R)}{\cosh^q(\cdot)f \in H^r(\R)},
\qquad
\norm{f}_{H_q^r} = \norm{f}_{r,q} := \norm{\cosh^q(\cdot)f}_{H^r(\R)}.
\end{equation}
and 
\[
E_q^r := H_q^r \cap \{\text{even functions}\}
\quadword{and}
O_q^r := H_q^r \cap \{\text{odd functions}\}.
\]
For a function $\fb = (f_1,f_2) \in H_q^r \times H_q^r$, we set
\[
\norm{\fb}_{r,q}
:= \norm{f_1}_{r,q} + \norm{f_2}_{r,q}.
\]
\end{definition}

Under this notation, we have $\varsigma_c \in \cap_{r=0}^{\infty} E_q^r$ for $q$ sufficiently small.
We set $\rhob = \varsigmab_c + \varrhob$, where $\varrhob = (\varrho_1,\varrho_2) \in E_q^2 \times O_q^2$, and compute that $\G_c(\varsigmab_c+\varrhob,\mu) = 0$ if and only if $\varrho_1$ and $\varrho_2$ satisfy the system
\begin{equation}\label{distilled tw syst}
c^2\varrhob'' + \D_0\varrhob + 2\D_0\nl(\varsigmab_c,\varrhob)
= \rhsb_c(\varrhob,\mu)
= \begin{pmatrix*}
\rhs_{c,1}(\varrhob,\mu) \\
\rhs_{c,2}(\varrhob,\mu)
\end{pmatrix*}.
\end{equation}
The right side $\rhsb_c(\varrhob,\mu)$ is ``small'' in the sense that it consists, roughly, of linear combinations of terms of the form $\mu$, $\mu\varrhob$, and $\varrhob^{.2}$.

The first component of this system has the form
\begin{equation}\label{H-c defn}
\bunderbrace{c^2\varrho_1'' + (2-A)(1+2\varsigma_c)\varrho_1}{\H_c\varrho_1}
= \rhs_{c,1}(\varrhob,\mu).
\end{equation}
The operator $\H_c$ is the linearization of the monatomic traveling wave problem at $\varsigma_c$.
Proposition 3.1 from \cite{hoffman-wright} tells us that, for $q$ sufficiently small, $\H_c$ is invertible from $E_q^{r+2}$ to $E_{q,0}^r$ for any $r \ge 0$, where
\[
E_{q,0}^r
:= \set{f \in E_q^r}{\hat{f}(0) = 0}.
\]
In some sense, this can be seen as a spectral stability result for the monatomic wave; see also Lemma 4.2 in \cite{friesecke-pego3} and Lemma 6 in \cite{herrmann-matthies-uniqueness} for Fredholm properties of $\H_c$, under different guises, in exponentially weighted Sobolev spaces.
It follows that \eqref{H-c defn} is equivalent to
\begin{equation}\label{varrho-1 fp}
\varrho_1
= \H_c^{-1}\rhs_{c,1}(\varrhob,\mu),
\end{equation}
which is a fixed point equation for $\varrho_1$ on the function space $E_q^r$.
We now attempt to construct a similar fixed point equation for $\varrho_2$; our failure in this attempt will be quite instructive.

\subsection{The operator $\L_c$}\label{Lc intro section}
The second component of \eqref{distilled tw syst} is
\begin{equation}\label{L-c defn}
\bunderbrace{c^2\varrho_2'' + (2+A)(1+2\varsigma_c)\varrho_2}{\L_c\varrho_2}
= \rhs_{c,2}(\varrhob,\mu).
\end{equation}
The operator $\L_c$ is the sum of a constant-coefficient second-order advance-delay differential operator and a nonlocal term, which we write more explicitly as 
\begin{equation}\label{B-c Sigma-c defns}
\L_cf
= \bunderbrace{c^2f'' + (2+A)f}{\B_cf}
+ \bunderbrace{2(2+A)(\varsigma_cf)}{-\Sigma_cf}.
\end{equation}
The minus sign on $\Sigma_c$ is purely for convenience.
If $\L_c$ were invertible from $O_q^{r+2}$ to $O_q^r$, then \eqref{L-c defn} would rearrange to a fixed point equation for $\varrho_2$, which we could combine with \eqref{varrho-1 fp} to get a fixed point problem for $\varrhob$.
In this attempt we fail: $\L_c$ is injective but not surjective.

\subsubsection{Injectivity}\label{injectivity section}
We first sketch how, at least in the case $|c| \gtrsim 1$, the operator $\L_c$ is injective from $O_q^{r+2}$ to $O_q^r$.
We begin with the constant-coefficient part $\B_c$.
This operator is a Fourier multiplier with symbol
\begin{equation}\label{tB-c defn}
\tB_c(k)
:= -c^2k^2 + 2 + 2\cos(k).
\end{equation}
That is, if $f \in L^2$ or $f \in L_{\per}^2$, then
\[
\hat{\B_cf}(k)
= \tB_c(k)\hat{f}(k).
\]
See Appendix \ref{fourier multipliers appendix} for further definitions and properties of Fourier multipliers.
A straightforward application of the intermediate value theorem yields a unique $\omega_c > 0$ such that $\tB_c(\omega_c) = 0$.
This Fourier condition in fact characterizes the range of $\B_c$: using results of Beale (phrased as Lemma \ref{beale fm} in this paper) and some further properties of $\tB_c$ from Proposition \ref{symbol of B}, one can show that $\B_c$ is invertible from $O_q^{r+2}$ to the subspace
\[
\Dfrak_{c,q}^r 
:= \set{f \in O_q^r}{\hat{f}(\omega_c) = 0}.
\]

Next, when $|c| \gtrsim 1$, the Friesecke-Pego solution $\varsigma_c$ is small in the sense that $\norm{\varsigma_c}_{L^{\infty}} = \O((c-1)^2)$; see part \ref{FP Linfty-2} of Proposition \ref{hoffman-wayne proposition}.
Consequently, the operator $\Sigma_c$ is also small.
However, $\Sigma_c$ does not map $O_q^{r+2}$ to $\Dfrak_{q,c}^r$; otherwise, we could use the Neumann series to invert $\B_c-\Sigma_c$.
Nonetheless, one can parley the smallness of $\Sigma_c$ and the invertibility of $\B_c$ into a coercive estimate of the form
\begin{equation}\label{L-c coercive}
\L_cf = g, \ f \in O_q^2, \ g \in O_q^0
\Longrightarrow
\norm{f}_{2,q}
\le C(c,q)\norm{g}_{0,q},
\end{equation}
which implies that $\L_c$ is injective on $O_q^2$ and, by the containment $O_q^{r+2} \subseteq O_q^2$ for $r \ge 0$, on $O_q^r$ for all $r \ge 0$.

\subsubsection{A characterization of the range of $\L_c$}\label{char of range}
We show, abstractly, that $\L_c$ is not surjective from $O_q^{r+2}$ to $O_q^r$.
Since the operator $\Sigma_c$ localizes functions, it is compact from $O_q^{r+2}$ to $O_q^r$, and so the Fredholm index of $\L_c = \B_c + \Sigma_c$ equals the index of $\B_c$
The Fourier analysis in Section \ref{injectivity section} shows that the index of $\B_c$ is $-1$, so $\L_c$ also has index $-1$.
Since $\L_c$ is injective, we conclude that $\L_c$ has a one-dimensional cokernel in $O_q^r$ and thus is not surjective.

However, we can characterize the range of $\L_c$ in $O_q^r$ more precisely.
Classical functional analysis tells us there is a nontrivial bounded linear functional $\zfrak_c$ on $O_q^0$ such that
\begin{equation}\label{equal mass solvability}
\L_cf = g, \surjmatter
\iff
\zfrak_c[g] = 0.
\end{equation}
Let
\[
\Zcal_q^{\star}
:= \set{f \in L_{\loc}^1}{\sech^q(\cdot)f \in L^2} \cap \{\text{odd functions}\}.
\]
The Riesz representation theorem then furnishes an nonzero function $\tzfrak_c \in \Zcal_q^{\star}$ such that 
\[
\zfrak_c[g] 
= \int_{-\infty}^{\infty} g(x)\tzfrak_c(x) \dx, \ g \in O_q^0.
\]
It follows that $\L_c^*\tzfrak_c = 0$, where 
\begin{equation}\label{L-c-star defn}
\L_c^*g
:= \B_cg + \bunderbrace{2\varsigma_c(x)(2+A)g}{-\Sigma_c^*g}
\end{equation}
is the $L^2$-adjoint of $\L_c$.
We conclude from \eqref{equal mass solvability} the ``solvability condition''
\begin{equation}\label{equal mass solvability revealed}
\L_cf = g, \surjmatter
\iff \int_{-\infty}^{\infty} g(x)\tzfrak_c(x) \dx = 0.
\end{equation}

It appears, then, that our attempt to solve the traveling wave problem $\G_c(\rhob,\mu) = 0$ by perturbing from the Friesecke-Pego solution $\varsigmab_c$ will fail, since $\L_c$ is not surjective.
Moreover, although we can characterize the range of $\L_c$ precisely via \eqref{equal mass solvability revealed}, all we know about $\tzfrak_c$ is that $\L_c^*\tzfrak_c = 0$ and $\tzfrak_c \in \Zcal_q^{\star}$.
The kernel of $\L_c^*$ in $\Zcal_q^{\star}$ must be one-dimensional, as otherwise, $\L_c^*$ would have Fredholm index $-2$ or lower, and so $\tzfrak_c$ is unique up to scalar multiplication.
But this function space $\Zcal_q^{\star}$ is quite large --- it contains, for example, all odd functions in $L^2$ and $L^{\infty}$ --- and so further features of $\tzfrak_c$ are not immediately apparent.

To determine our next steps, we look back to the work of our predecessors in the long wave \cite{faver-wright} and small mass \cite{hoffman-wright} problems.
In each of these problems, an operator similar to $\L_c$ appears; each of these operators on $O_q^{r+2}$ has a one-dimensional cokernel in $O_q^r$ because of a solvability condition like \eqref{equal mass solvability revealed}.
Moreover, the authors were able to construct odd solutions in $W^{2,\infty}$ to their versions of $\L_c^*g = 0$ that asymptote to a sinusoid.
We refer to such solutions as ``Jost solutions,'' due to their similarity to the classical Jost solutions for the Schrodinger equation \cite{reed-simon3}. 
Seeing how the Jost solutions in the long wave and small mass problems are determined both guides us to the features of the Jost solution that we seek for $\L_c^*$ and help us appreciate what is intrinsically different about $\L_c^*$ when compared to its analogues in the prior nanopteron problems.

\subsubsection{Jost solutions in the long wave limit \cite{faver-wright}}
The analogue of $\L_c$ in this problem is, roughly, the operator 
\[
\W_{\varepsilon}f
:= (1+\varepsilon^2)\varepsilon^2f'' + \M_{\varepsilon}f,
\]
where $\varepsilon \approx 0$ and $\M_{\varepsilon}$ is a Fourier multiplier with the real-valued symbol $\tM(\varepsilon\cdot)$, i.e., $\hat{\M_{\varepsilon}f}(k) = \tM(\varepsilon{k})\hat{f}(k)$.
The wave speed is intrinsically linked to the small long wave parameter $\varepsilon$, and so $c$ does not appear here.
Since $\W_{\varepsilon}$ is a Fourier multiplier with real-valued symbol $\tW_{\varepsilon}(k) := -(1+\varepsilon^2)\varepsilon^2k^2 + \tM(\varepsilon{k})$, it is self-adjoint in $L^2$.
An intermediate value theorem argument similar to the one referenced in Section \ref{injectivity section} gives the existence of a unique $\Omega_{\varepsilon} > 0$ such that $\tW_{\varepsilon}(K) = 0$ if and only if $K = \pm\Omega_{\varepsilon}$, and so $\W_{\varepsilon}\sin(\Omega_{\varepsilon}\cdot) = \W_{\varepsilon}\cos(\Omega_{\varepsilon}\cdot) = 0$.
That is, the Jost solutions are exactly sinusoidal.

\subsubsection{Jost solutions in the small mass limit \cite{hoffman-wright}}
Here the analogue of $\L_c$ is, roughly,  
\[
\T_c(m)f
:= \bunderbrace{c^2mf'' + (1+2\zeta_c(x))f}{\S_c(m)f} + m\M_mf + m\J_m(\zeta_c(x)f).
\]
We take $m \approx 0$, where $m$ is the mass ratio of the diatomic lattice, per \eqref{masses}.
The operators $\M_m$ and $\J_m$ are Fourier multipliers that are $\O(1)$ in $m$, so the perturbation terms $m\M_m$ and $m\J_m$ are indeed small, but Hoffman and Wright found it essential to retain these terms rather than absorb them into their analogues of our right side $\rhs_{c,2}$ from \eqref{L-c defn}.
Here the exponentially localized function $\zeta_c$ is the Friesecke-Pego solitary wave profile corresponding to the monatomic lattice formed by taking $m=0$ in \eqref{masses}.

The $L^2$-adjoint of $\T_c(m)$ is 
\[
\T_c(m)^*g
= \S_c(m)g + m\M_mg + m\varsigma_c(x)\J_mg.
\]
Hoffman and Wright construct a nontrivial solution to $\T_c(m)^*g = 0$ as follows.
First, they show that $\S_c(m)$ vanishes on certain asymptotically sinusoidal functions, i.e., there exists $\j_c^m \in W^{2,\infty}$ such that 
\begin{equation}\label{hw asymptotics}
\S_c(m)\j_c^m = 0
\quadword{and}
\lim_{x \to \infty} |\j_c^m(x) - \sin(\Omega_c^m(x+\theta_c^m))| 
= 0
\end{equation}
for some critical frequency $\Omega_c^m$ and phase shift $\theta_c^m$.
This solution $\j_c^m$ is indeed a classical Jost solution for the Schrodinger operator $\S_c(m)$.
The proof of the asymptotics in \eqref{hw asymptotics} uses a polar coordinate decomposition that closely relies on the structure of $\S_c(m)$ as a second-order linear differential operator.

Second, it is clear that $\T_c(m)^*$ is a small nonlocal perturbation of $\S_c(m)$.
Using these facts, Hoffman and Wright perform an intricate variation of parameters argument (which again relies on the differential operator structure of $\S_c(m)$) in an asymptotically sinusoidal subspace of $W^{2,\infty}$ to construct a function $\gamma_c^m \in W^{2,\infty}$ with the properties that 
\[
\T_c(m)^*\gamma_c^m = 0
\quadword{and}
\lim_{x \to \infty} |\gamma_c^m(x) - \sin(\Omega_c^m(x +\vartheta_c^m))| = 0
\]
for some (new) phase shift $\vartheta_c^{m}$, which is a small perturbation of $\theta_c^m$.
A corollary to this analysis is the frequency-phase shift ``resonance'' relation $\sin(\Omega_c^m\vartheta_c^m) \ne 0$ for almost all values of $m$ close to zero, which turns out, in a subtle and surprising way, to be critical for their subsequent analysis.

\subsubsection{Toward Jost solutions for the equal mass operator $\L_c^*$}\label{Lc surjectivity intro}
Unlike the long wave operator $\W_{\varepsilon}$, the operator $\L_c^*$ is not a ``pure'' Fourier multiplier as it has the variable-coefficient piece $\Sigma_c^*$.
And unlike the small mass operator $\T_c(m)^*$, we cannot decompose $\L_c^*$ as the sum of a classical differential operator and a perturbation term that is small in $\mu$.
So, we cannot directly import prior results to produce the Jost solutions of $\L_c^*$. 

Our approach is to take advantage of two particular aspects of the structure of $\L_c^*$.
First, the constant-coefficient part $\B_c$ is an advance-delay operator formed by a simple linear combination of shift operators.
The Fredholm properties of such operators have received significant attention from Mallet-Paret \cite{mallet-paret}.
Next, the variable-coefficient piece $\Sigma_c^*$ is both exponentially localized and small for $|c| \gtrsim 1$.
These facts are sufficient to solve the equation $(\B_c-\Sigma_c^*)f=0$ in a class of ``one-sided'' exponentially weighted Sobolev spaces, whose features we specify below in Definition \ref{one sided expn weighted spaces defn}.

In broad strokes, then, we first use an adaptation of Mallet-Paret's theory due to Hupkes and Verduyn-Lunel \cite{hvl} to invert $\B_c$ on these one-sided spaces and, moreover, obtain a precise formula for its inverse.
Next, it turns out that $\Sigma_c^*$ does map between these one-sided spaces\footnote{Unlike its failure to map between $O_q^{r+2}$ and $\Dfrak_{c,q}^r$, as we saw in Section \ref{injectivity section}).} and so, since $\Sigma_c^*$ is still ``small,'' we are able to invert $\B_c-\Sigma_c^*$ with the Neumann series.
This procedure yields a function $\gamma_c \in W^{2,\infty}$ that satisfies $\L_c^*\gamma_c = 0$ and is asymptotic to a phase-shifted sinusoid of frequency $\omega_c$ at $\infty$, where $\omega_c$ is the ``critical frequency'' of $\B_c$ that appeared in Section \ref{injectivity section}.
The exact formulas that we enjoy for the inverses of $\B_c$ and then $\B_c-\Sigma_c^*$ permit us to calculate an asymptotic expansion for the phase shift.

\subsection{Main result for the case $|c| \gtrsim 1$}
Now that we understand precisely the range of $\L_c$, as well as its lack of surjectivity, we can confront again the traveling wave problem $\G_c(\rhob,\mu) = 0$ from \eqref{cov simplified}, which we seek to solve for $\rhob \approx \varsigmab_c$ and $\mu \approx 0$.
The general structure of this problem, most especially the solvability condition \eqref{equal mass solvability revealed}, puts us in enough concert with our predecessors in \cite{faver-wright} and  \cite{hoffman-wright} that we may follow their modifications of a method due to Beale \cite{beale2} for solving problems with such a solvability condition.
Specifically, we replace the perturbation ansatz $\rhob = \varsigmab_c + \varrhob$ with the {\it{nonlocal solitary wave}} anzatz $\rhob = \varsigmab_c + a\phib_c^{\mu}[a] + \etab$, where $\etab = (\eta_1,\eta_2)$ is exponentially localized and $a\phib_c^{\mu}[a]$ is a periodic solution to the traveling wave problem with amplitude roughly $a$ and frequency (very) roughly $\omega_c$.
Of course, proving the existence of periodic traveling wave solutions is a fundamental part of our analysis.

Under Beale's ansatz, one easily finds a fixed point equation for $\eta_1$, similar to how we converted \eqref{H-c defn} into \eqref{varrho-1 fp}.
Then we can extract from the solvability condition a fixed point equation for $a$, and, in turn, an equation for $\eta_2$.
We carry out this construction in Section \ref{nanopteron problem section}, where we prove (as Theorem \ref{main theorem formal}) a more technical version of our main theorem below.

\begin{theorem}\label{main theorem - informal}
Suppose $|c| \gtrsim 1$.
For $|\mu|$ sufficiently small, there are functions $\Upsilon_{c,1}^{\mu}$, $\Upsilon_{c,2}^{\mu}$, $\varphi_{c,1}^{\mu}$, $\varphi_{c,2}^{\mu} \in \Cal^{\infty}(\R)$ such that the traveling wave profiles
\[
\rho_{c,1}^{\mu} := \varsigma_c + \Upsilon_{c,1}^{\mu} + \varphi_{c,1}^{\mu}
\quadword{and}
\rho_{c,2}^{\mu} := \Upsilon_{c,2}^{\mu} + \varphi_{c,2}^{\mu}
\]
satisfy $\G_c((\rho_{c,1}^{\mu},\rho_{c,2}^{\mu}),\mu) = 0$.
The functions $\Upsilon_{c,1}^{\mu}$ and $\Upsilon_{c,2}^{\mu}$ are exponentially localized, while $\varphi_{c,1}^{\mu}$ and $\varphi_{c,2}^{\mu}$ are periodic.
The amplitude of all four functions is $\O(\mu)$.
\end{theorem}

When $\mu = 0$ and our lattice is monatomic, the micropteron in Theorem \ref{main theorem - informal} reduces to the Friesecke-Pego solitary wave.
We remark that the stability of these micropteron traveling wave solutions for the equal mass limit, as well as the stability of the long wave and small mass nanopterons, is an intriguing open problem; see \cite{johnson-wright} for some initial forays into this arena.

\subsection{Toward the case $|c| \gg 1$}
Throughout this introduction, we have assumed that $|c|$ is close to 1.
This first allows us to summon a Friesecke-Pego solitary wave solution $\varsigma_c$ for the monatomic problem and next shows, through the coercive estimate \eqref{L-c coercive}, that $\L_c$ is injective from $O_q^{r+2}$ to $O_q^r$.
Third, taking $|c|$ close to 1 is fundamental for the Neumann series argument that allows us to invert $\B_c-\Sigma_c^*$ in the one-sided spaces.
Ostensibly, then, it appears that we are working with two small parameters, $\mu$ and $|c|-1$, in our equal mass problem.

This turns out, from a certain point of view, to be superfluous.
In Section \ref{arbitrary wave speed}, we present four simple and natural hypotheses that, if satisfied for an arbitrary wave speed $|c| > 1$ guarantee a micropteron solution in the equal mass limit.
All of these hypotheses are satisfied for $|c| \gtrsim 1$.
Among these hypotheses is the existence of an exponentially localized and spectrally stable traveling wave solution for the monatomic problem with wave speed possibly much greater 1; this is the origin of the dashed blue line in Figure \ref{fig-all the results}.
We emphasize that the existence of the Jost solutions to $\L_c^*g = 0$ is not one of these hypotheses; their construction, for an arbitrary $|c| > 1$, follows from a weaker hypothesis.

This approach has several advantages and justifications over working only in the near-sonic regime $|c| \gtrsim 1$.
First, it decouples the starring small parameter $\mu$ from the deuteragonist\footnote{In fact, a third small parameter also lurks in nonlocal solitary wave problems: the precise decay rate $q$ of the exponentially localized spaces $E_q^r$ and $O_q^r$.  Our hypotheses make explicit the decay rates that we employ and highlight the relationships between the rates at different stages of the problem.} $|c|-1$; after all, we are interested in the equal mass limit, not the near-sonic limit.
Next, it frees us from overreliance on the Friesecke-Pego traveling wave and leaves our results open to interpretation and invocation in the case of high-speed monatomic traveling waves.
For example, Herrmann and Matthies \cite{herrmann-matthies-asymptotic, herrmann-matthies-uniqueness, herrmann-matthies-stability} have developed a number of results on the asymptotics, uniqueness, and stability of solitary traveling wave solutions to the monatomic FPUT problem (albeit with different spring forces from ours) in the ``high-energy limit,'' which inherently assumes a large wave speed.
Additionally, the original monatomic solitary traveling wave of Friesecke and Wattis \cite{friesecke-wattis} does not come with a near-sonic restriction on its speed, and so, in principle, that wave could have speed much greater than 1.
Further study of the existence and properties of solutions to the monatomic traveling wave problem in the high wave speed regime remains an interesting open problem, and we are eager to see how our hypotheses may exist in concert with future solutions to that problem.

\subsection{Remarks on notation}\label{basic notation section}
We define some basic notation that will be used throughout the paper.

\begin{enumerate}[label=$\bullet$]

\item
The letter $C$ will always denote a positive, finite constant; if $C$ depends on some other quantities, say, $q$ and $r$, we will write this dependence in function notation, i.e., $C = C(q,r)$.
Frequently $C$ will depend on the wave speed $c$, in which case we write $C(c)$.

\item
We employ the usual big-$\O$ notation: if $x_{\ep} \in \C$ for $\ep \in \R$, then we write $x_{\ep} = \O(\ep^p)$ if there are constants $\ep_0$, $C(\ep_0,p) > 0$ such that 
\[
0 < \ep < \ep_0
\Longrightarrow |x_{\ep}| \le C(\ep_0,p)\ep^p.
\]

\item
Next, we need a $c$-dependent notion of big-$\O$ notation.
Suppose that $x_c^{\mu} \in \C$ for $c$, $\mu \in \R$.
Then we write $x_c^{\mu} = \O_c(\mu^p)$ if there are constants $\mu_c$, $C(c,p) > 0$ such that 
\[
|\mu| \le \mu_c
\Longrightarrow |x_c^{\mu}| \le C(c,p)|\mu|^p.
\]

\item
If $\X$ and $\Y$ are normed spaces, then $\b(\X,\Y)$ is the space of bounded linear operators from $\X$ to $\Y$ and $\b(\X) := \b(\X,\X)$.
\end{enumerate}

\subsection{Outline of the remainder of the paper}
Here we briefly discuss the structure of the rest of the paper.

\begin{enumerate}[label=$\bullet$]

\item
Section \ref{arbitrary wave speed} contains the precise statements of the four hypotheses under which we will work.

\item
Section \ref{periodic solutions section} constructs a family of small-amplitude periodic traveling wave solutions to the traveling wave problem \eqref{cov simplified} that possess a number of uniform estimates in the small parameter $\mu$.
These periodic traveling waves exist for arbitrary $|c| > 1$.

\item 
Section \ref{analysis of L-c} characterizes the range of $\L_c$ as an operator from $O_q^{r+2}$ to $O_q^r$ by constructing a Jost solution for $\L_c^*g = 0$.

\item
Section \ref{nanopteron problem section} converts the nonlocal solitary wave equations of Section \ref{beale's ansatz section} into a fixed point problem, which we then solve with a modified contraction mapping argument.

\item
Section \ref{hypotheses verification section} shows that the four main hypotheses of Section \ref{arbitrary wave speed} are valid in the case $|c| \gtrsim 1$.

\item
The appendices contain various technical proofs and ancillary background material.
\end{enumerate}
\section{The Traveling Wave Problem for Arbitrary Wave Speeds}\label{arbitrary wave speed}

In this section we discuss the four hypotheses that are sufficient to guarantee micropteron traveling wave solutions in the equal mass limit for arbitrary wave speeds.
Throughout, we fix $c \in \R$ with $|c| > 1$.
Since $|c|$ may not be close to 1, we are not guaranteed a monatomic traveling wave solution from Friesecke and Pego, and so we make its existence our first hypothesis.

\begin{hypothesis}\label{hypothesis 1}
There exist $q_{\varsigma}(c) > 0$ and a real-valued function $\varsigma_c \in E_{q_{\varsigma}(c)}^2$ such that 
\begin{equation}\label{what sigma c does}
c^2\varsigma_c'' + (2-A)(\varsigma_c+\varsigma_c^2) = 0.
\end{equation}
\end{hypothesis}

That is, $\varsigma_c$ is an exponentially localized traveling wave profile with wave speed $c$ for the monatomic FPUT equations of motion.
With $\varsigma_c$ satisfying \eqref{what sigma c does}, it follows that if $\varsigmab_c := (\varsigma_c,0)$, then $\G_c(\varsigmab_c,0) = 0$.
A straightforward bootstrapping argument, discussed in Appendix \ref{friesecke-pego bootstrap appendix}, endows arbitrary regularity to $\varsigma_c$.

\begin{proposition}\label{friesecke-pego bootstrap}
The monatomic traveling wave solution $\varsigma_c$ from Hypothesis \ref{hypothesis 1} also satisfies $\varsigma_c \in \cap_{r=1}^{\infty} E_{q_{\varsigma}(c)}^r$.
\end{proposition}

Next, we add the invertibility of the linearization of the monatomic traveling wave problem at $\varsigma_c$ as a hypothesis for the situation when $|c|$ is not necessarily close to 1.

\begin{hypothesis}\label{hypothesis 2}
There exists $q_{\H}(c) \in (0,\min\{1,q_{\varsigma}(c)\})$ such that the operator $\H_c$ from \eqref{H-c defn} is invertible from $E_{q_{\H}(c)}^2$ to $E_{q_{\H}(c),0}^0$.
\end{hypothesis}

In Section \ref{Lc surjectivity intro}, we claimed that the key to characterizing the range of $\L_c$, defined in \eqref{L-c defn}, was the invertibility of its $L^2$-adjoint, $\L_c^*$, defined in \eqref{L-c-star defn}, on a class of one-sided exponentially weighted spaces.
Now we make precise the definitions of these spaces and the weights for which $\L_c^*$ needs to be invertible.

\begin{definition}\label{one sided expn weighted spaces defn}
For $q \in \R$ and $m$, $r \in \N$, we define
\[
W_q^{r,\infty}(\R,\C^m)
:= \set{f \in L_{\loc}^1(\R,\C^m)}{e^{-q\cdot}\fb \in W^{r,\infty}(\R,\C^m)}
\]
and
\[
L_q^{\infty}(\R,\C^m)
:= W_q^{0,\infty}(\R,\C^m).
\]
When $m = 1$, we write just $W_q^{r,\infty}$ and $L_q^{\infty}$.
We set
\[
\norm{\fb}_{W_q^{r,\infty}(\R,\C^m)}
:= \norm{e^{-q\cdot}\fb}_{W^{r,\infty}(\R,\C^m)}.
\]
\end{definition}

\begin{hypothesis}\label{hypothesis 3}\label{hypotheses theorem}
There exists $q_{\L}(c) \in (0,\min\{q_{\H}(c),q_{\varsigma}(c)/2,1\})$ such that the operator $\L_c^*$ is invertible from $W_{-q_{\L}(c)}^{2,\infty}$ to $L_{-q_{\L}(c)}^{\infty}$.
\end{hypothesis}

Assuming this hypothesis we obtain a precise characterization of the range of $\L_c$.
We prove this theorem in Section \ref{analysis of L-c}.

\begin{proposition}\label{phase shift theorem - hypos section}
There is a nonzero odd function $\gamma_c \in W^{2,\infty}$ (a ``Jost solution'' to $\L_c^*g=0$) such that for $f \in O_q^{r+2}$ and $g \in O_q^r$, with $r \ge 0$ and $q \in [q_{\L}(c),1)$, we have
\[
\L_cf = g
\iff \int_{-\infty}^{\infty} g(x)\gamma_c(x) \dx = 0. 
\]
Moreover, $\gamma_c$ is asymptotically sinusoidal in the sense that, for some $\theta_c \in \R$,
\[
\lim_{x \to \infty} |\gamma_c(x)-\sin(\omega_c(x+\vartheta_c))|
= \lim_{x \to \infty} |(\gamma_c)'(x)-\omega_c\cos(\omega_c(x+\vartheta_c))|
= 0.
\]
\end{proposition}

Last, like Hoffman and Wright, we need one more condition on the interaction between the critical frequency $\omega_c$ from Section \ref{injectivity section} and the phase shift $\vartheta_c$ of the Jost solution to $\L_c^*f = 0$ from Theorem \ref{phase shift theorem - hypos section}.
This condition arises in practice much later for quite a technical reason; see Section \ref{eta2 a eqn constr}.

\begin{hypothesis}\label{hypothesis 4}
$\sin(\omega_c\vartheta_c) \ne 0.$
\end{hypothesis}

All four hypothesis hold when $|c| \gtrsim 1$; we give the proof of the next theorem in Section \ref{hypotheses verification section}.
The verification of Hypotheses \ref{hypothesis 1} and \ref{hypothesis 2} is merely a matter of quoting results from \cite{friesecke-pego1} and \cite{hoffman-wright},
On the other hand, verifying Hypotheses \ref{hypothesis 3} and \ref{hypothesis 4} relies on results from \cite{hvl}, and the proof of their validity is among the central technical results of this paper.

\begin{theorem}\label{hypotheses theorem}
There exists $c_{\star} > 1$ such that Hypotheses \ref{hypothesis 1}, \ref{hypothesis 2}, \ref{hypothesis 3}, and \ref{hypothesis 4} hold if $|c| \in (1,c_{\star})$.
\end{theorem}

Under these hypotheses we retain the existence of micropterons in the equal mass limit.
We prove the next theorem, our central result, as Theorem \ref{main theorem formal}.

\begin{theorem}
Suppose that Hypotheses \ref{hypothesis 1}, \ref{hypothesis 2}, \ref{hypothesis 3}, and \ref{hypothesis 4} hold for some $|c| > 1$.
Then the results of Theorem \ref{main theorem - informal} remain true.
\end{theorem}
\section{Periodic Solutions}\label{periodic solutions section}
In this section we state our existence result for small-amplitude periodic solutions to the traveling wave problem $\G_c(\rhob,\mu) = 0$ from \eqref{cov simplified}.
We emphasize that the results in this section hold for all $|c| > 1$; we do not need to invoke any of the $c$-dependent hypotheses here.

We work in the periodic Sobolev spaces
\begin{equation}\label{H-per-r}
H_{\per}^r(\R,\C^m)
:= \set{\fb \in L_{\per}^2(\R,\C^m)}{\norm{\fb}_{H_{\per}^r(\R,\C^m)}  < \infty},
\end{equation}
where
\[
\norm{\fb}_{H_{\per}^r(\R,\C^m)} 
:= \left(\sum_{k=-\infty}^{\infty} (1+k^2)^r|\hat{\fb}(k)|^2\right)^{1/2}.
\]
These are Hilbert spaces with the inner product
\begin{equation}\label{H-per-r ip}
\ip{\fb}{\gb}_{H_{\per}^r(\R,\C^m}
= \sum_{k=-\infty}^{\infty} (1+k^2)^r(\hat{\fb}(k)\cdot\hat{\gb}(k)).
\end{equation}
When $m = 1$, we write just $H_{\per}^r$.
We will continue to exploit the symmetries of our traveling wave problem and work on the subspaces
\[
E_{\per}^r 
:= \set{f \in H_{\per}^r}{f \text{ is even}},
\]
\[
E_{\per,0}^r 
:= \set{f \in E_{\per}^r}{\hat{f}(0) = 0},
\]
and
\[
O_{\per}^r 
:= \set{f \in H_{\per}^r}{f \text{ is odd}}.
\]

Set $\rhob(x) = \phib(\omega{x})$, where $\phib \in E_{\per,0}^2 \times O_{\per}^2 \subseteq H_{\per}^r(\R,\C^2)$ and $\omega \in \R$.
Under this scaling, the problem $\G_c(\rhob,\mu) = 0$ becomes
\begin{equation}\label{Phib-c-mu defn}
\bunderbrace{c^2\omega^2\phib'' + \D_{\mu}[\omega]\phib + \D_{\mu}[\omega]\nl(\phib,\phib)}{\Phib_c^{\mu}(\phib,\omega)}=0,
\end{equation}
where
\[
\D_{\mu}[\omega] 
:= \frac{1}{2}
\begin{bmatrix*}
(2+\mu)(1-A[\omega]) &\mu\delta[\omega] \\
-\mu\delta[\omega] &(2+\mu)(1+A[\omega])
\end{bmatrix*},
\]
with
\[
A[\omega] 
:= S^{\omega} + S^{-\omega}
\quadword{and}
\delta[\omega] 
:= S^{\omega}-S^{-\omega}.
\]

As in the proofs of periodic solutions for the long wave problems (Theorems 4.1 in \cite{faver-wright} and 3.1 in \cite{faver-spring-dimer}) and the small mass problem (Theorem 5.1 in \cite{hoffman-wright}), we first look for solutions to the linear problem, which is
\begin{equation}\label{Gamma-c-mu defn}
\bunderbrace{c^2\omega^2\phib'' + \D_{\mu}[\omega]\phib}{\Gamma_c^{\mu}[\omega]\phib}
= 0.
\end{equation}
Observe that if $\Gamma_c^{\mu}[\omega]\phib = 0$, then taking the Fourier transform yields
\begin{equation}\label{ev ew rel}
c^2\omega^2k^2\hat{\phib}(k) = \tD_{\mu}[\omega{k}]\hat{\phib}(k),
\end{equation}
where
\begin{equation}\label{tDmu defn}
\tD_{\mu}[K]:=
\begin{bmatrix*}
(2+\mu)(1-\cos(K)) &i\mu\sin(K) \\
-i\mu\sin(K) &(2+\mu)(1+\cos(K))
\end{bmatrix*}.
\end{equation}
Thus $\Gamma_c^{\mu}[\omega]\phib = 0$ with $\phib$ nonzero if and only if, for some $k \in \Z$, $\hat{\phi}(k) \ne 0$ is an eigenvector of the matrix $\tD_{\mu}[\omega{k}] \in \C^{2\times2}$ corresponding to the eigenvalue $c^2\omega^2k^2$.

We compute that the eigenvalues of $\tD_{\mu}[K]$ are $\lambda_{\mu}^{\pm}(K)$, where
\begin{equation}\label{lambda eigs defn}
\lambda_{\mu}^{\pm}(K) := 2+\mu \pm \sqrt{\mu^2 + 4(1+\mu)\cos^2(K)}.
\end{equation}
These are the same eigenvalues studied in \cite{faver-wright}, \cite{hoffman-wright}, and \cite{faver-spring-dimer}.
See Figure \ref{fig-eigencurves} for a sketch of the curves $\lambda_{\mu}^{\pm}(K)$ against $c^2K^2$. 

So, $\Gamma_c^{\mu}[\omega]\phib = 0$ with $\hat{\phib}(k) \ne 0$ if and only if
\[
c^2\omega^2k^2 = \lambda_{\mu}^-(\omega{k})
\quadword{or}
c^2\omega^2k^2 = \lambda_{\mu}^+(\omega{k})
\]
We will prove below in Proposition \ref{lambda props} that if $|c| > 1$, then the first equality above can never hold, at least over an appropriate $c$-dependent range of $\mu$, while the second holds precisely at $\omega = \omega_c^{\mu}/k$, $k \in \Z\setminus\{0\}$, for a certain ``critical frequency'' $\omega_c^{\mu} > 0$, which is an $\O_c(\mu)$ perturbation of the frequency $\omega_c$ from Section \ref{injectivity section}.
We elect to take $\omega = \omega_c^{\mu}$ (i.e., $k = \pm 1$) so that the frequency of our periodic solutions is close to $\omega_c$; this is important in our construction of the micropterons in Section \ref{nanopteron problem section}.

We now highlight several important properties of $\omega_c^{\mu}$, both contained in and proved as part of Lemma \ref{critical frequency technical lemma}.

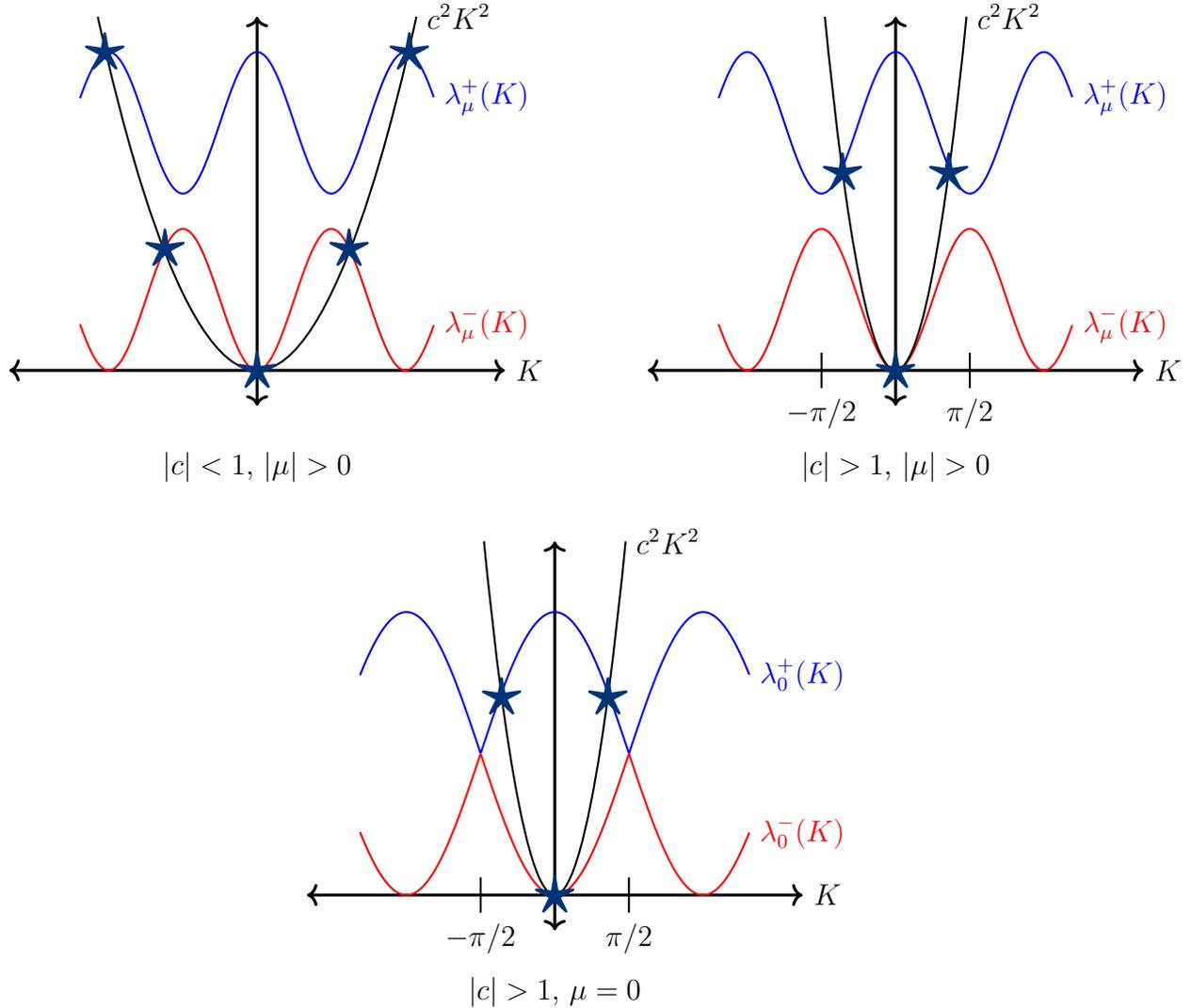
\begin{figure}
\begin{tabular}{cc}
\begin{tikzpicture}[thick]

\draw[<->, very thick] (-3.5,0)--(3.5,0)node[right]{$K$};
\draw[<->, very thick] (0,-.5)--(0,5);
\draw[blue] plot[domain=-2.5:2.5,smooth,samples=500] (\x,{3.5+cos(3*\x r)})node[right]{$\lambda_{\mu}^+(K)$};
\draw[red] plot[domain=-2.5:2.5,smooth,samples=500] (\x,{sin(deg(3*\x+3*pi/2))+1})node[right]{$\lambda_{\mu}^-(K)$};

\draw (-2.25,5) parabola bend (0,0) (2.25,5)node[right]{$c^2K^2$};

\node[below] at (0,-1){$|c| < 1$, $|\mu| > 0$};

\eigstar{-2.15}{4.5};
\eigstar{-1.3}{1.725};
\eigstar{0}{0};
\eigstar{1.3}{1.725};
\eigstar{2.15}{4.5};

\end{tikzpicture}\qquad
&\qquad
\begin{tikzpicture}[thick]

\draw[<->, very thick] (-3.5,0)--(3.5,0)node[right]{$K$};
\draw[<->, very thick] (0,-.5)--(0,5);
\draw[blue] plot[domain=-2.5:2.5,smooth,samples=500] (\x,{3.5+cos(3*\x r)})node[right]{$\lambda_{\mu}^+(K)$};
\draw[red] plot[domain=-2.5:2.5,smooth,samples=500] (\x,{sin(deg(3*\x+3*pi/2))+1})node[right]{$\lambda_{\mu}^-(K)$};

\draw (-1,5) parabola bend (0,0) (1,5)node[right]{$c^2K^2$};

\node[below] at (0,-1){$|c| > 1$, $|\mu| > 0$};

\draw (pi/3,-.25)node[below]{$\pi/2$}--(pi/3,.25);
\draw (-pi/3,-.25)node[below]{$-\pi/2$}--(-pi/3,.25);

\eigstar{-.5*1.5}{5.6*.5};
\eigstar{0}{0};
\eigstar{.5*1.5}{5.6*.5};

\end{tikzpicture}
\end{tabular}

\[
\begin{tikzpicture}[thick]

\draw[<->, very thick] (-3.5,0)--(3.5,0)node[right]{$K$};
\draw[<->, very thick] (0,-.5)--(0,5);

\draw[blue] plot[domain=-2.75:2.75,smooth,samples=500] (\x,{2+2*abs(cos(1.5*\x r))})node[right]{$\lambda_0^+(K)$};
\draw[red] plot[domain=-2.75:2.75,smooth,samples=500] (\x,{2-2*abs(cos(1.5*\x r))})node[right]{$\lambda_0^-(K)$};

\draw (-1,5) parabola bend (0,0) (1,5)node[right]{$c^2K^2$};

\node[below] at (0,-1){$|c| > 1$, $\mu = 0$};

\draw (pi/3,-.25)node[below]{$\pi/2$}--(pi/3,.25);
\draw (-pi/3,-.25)node[below]{$-\pi/2$}--(-pi/3,.25);

\eigstar{-.5*1.5}{5.6*.5};
\eigstar{0}{0};
\eigstar{.5*1.5}{5.6*.5};
\end{tikzpicture}
\]
\caption{Sketches of the graphs of the eigencurves $\lambda_{\mu}^{\pm}(K)$ and the parabola $c^2K^2$ for different cases on $|c|$ and $|\mu|$.
When $|c| < 1$, $\lambda_{\mu}^{\pm}(K)$ may have more intersections with $c^2K^2$ than just $K = \pm \omega_c^{\mu}$.  
Note that when $\mu = 0$, the curves $\lambda_0^{\pm}$ intersect at odd multiples of $\pi/2$, but this does not affect the eigenvalue analysis, since the critical frequencies $\omega_c^{\mu}$ are contained in $(-\pi/2,\pi/2)$.}
\label{fig-eigencurves}
\end{figure}

\begin{proposition}\label{critical frequency props prop}

\begin{enumerate}[label={\bf(\roman*)},ref={(\roman*)}]

\item
For all $|c| > 1$, there is $\mu_{\omega}(c) > 0$ such that if $|\mu| \le \mu_{\omega}(c)$, then there is a unique number $\omega_c^{\mu} > 0$ such that $c^2(\omega_c^{\mu})^2 = \lambda_{\mu}^+(\omega_c^{\mu})$.

\item\label{critical frequency bounds}
There are numbers $0 < A_c < B_c < \pi/2$ such that $A_c < \omega_c^{\mu} < B_c$ for all $|\mu| \le \mu_{\omega}(c)$.

\item\label{c0}
If $|c| \in (1,\sqrt{2}]$, then $A_c \ge 1$.

\item\label{omega-c-mu minus omega-c}
$\omega_c^{\mu} - \omega_c = \O_c(\mu).$
\end{enumerate}
\end{proposition}

\begin{remark}
The restriction $|c| \in (1,\sqrt{2}]$ will appear in several technical estimates throughout the rest of the paper.
This is merely to ensure the convenient lower bound $\omega_c^{\mu} > 1$, which will be useful when we verify the four hypotheses for $|c| \gtrsim 1$.
\end{remark}

Following the methods of our predecessors, we will use this critical frequency $\omega_c^{\mu}$ in a modified Crandall-Rabinowitz-Zeidler bifurcation from a simple eigenvalue argument \cite{crandall-rabinowitz, zeidler} to construct the exact periodic solutions to the full problem $\G_c(\rhob,\mu) = 0$.
Here is our result, proved in Appendix \ref{periodic solutions appendix}.

\begin{proposition}\label{periodic solutions theorem}
For each $|c| > 1$, there are $\mu_{\per}(c) \in (0,\min\{\mu_{\omega}(c),1\})$ and $a_{\per}(c) > 0$ such that for all $|\mu| \le \mu_{\per}(c)$, there are maps
\begin{align*}
\omega_c^{\mu}[\cdot] &\colon [-a_{\per}(c),a_{\per}(c)] \to \R \\
\psi_{c,1}^{\mu}[\cdot] &\colon [-a_{\per}(c),a_{\per}(c)] \to \Cal_{\per}^{\infty} \cap \{\text{even functions}\}  \\
\psi_{c,2}^{\mu}[\cdot] &\colon [-a_{\per}(c),a_{\per}(c)] \to \Cal_{\per}^{\infty} \cap \{\text{odd functions}\} 
\end{align*}
with the following properties.

\begin{enumerate}[label={\bf(\roman*)}, ref = {(\roman*)}]

\item\label{periodic expansion part}
If
\begin{equation}\label{phib-c-mu-a expansion}
\phib_c^{\mu}[a](x) 
:= \begin{pmatrix*}
\upsilon_c^{\mu}\cos(\omega_c^{\mu}[a]x) \\
\sin(\omega_{c}^{\mu}[a]x)
\end{pmatrix*}
+ \begin{pmatrix*}
\psi_{c,1}^{\mu}[a](\omega_{c}^{\mu}[a]x) \\
\psi_{c,2}^{\mu}[a](\omega_{c}^{\mu}[a]x)
\end{pmatrix*},
\end{equation}
where $\upsilon_{c}^{\mu} = \O_c(\mu)$ is defined below in \eqref{upsilonw defn}, then $\G_{c}(a\phib_{c}^{\mu}[a],\mu) = 0$.

\item\label{omega-c-mu-0=0}
$\omega_{c}^{\mu}[0] = \omega_{c}^{\mu}.$

\item\label{psib-c-mu-0=0}
$\psi_{c,1}^{\mu}[0] = \psi_{c,2}^{\mu}[0] = 0$.

\item\label{periodic lipschitz part}
For each $r \ge 0$, there is $C(c,r) > 0$ such that if $|a|$, $|\grave{a}| \le a_{\per}$ and $|\mu| \le \mu_{\per}(c)$, then
\begin{equation}\label{periodic theorem lipschitz}
|\omega_{c}^{\mu}[a] -\omega_{c}^{\mu}[\grave{a}]|
+ \norm{\psi_{c,1}^{\mu}[a] - \psi_{c,1}^{\mu}[\grave{a}]}_{W^{r,\infty}}
+ \norm{\psi_{c,2}^{\mu}[a]-\psi_{c,2}^{\mu}[\grave{a}]}_{W^{r,\infty}}
\le C(c,r)|a-\grave{a}|.
\end{equation}
\end{enumerate}
\end{proposition}

For later use, we isolate two additional estimates on the periodic solutions and their frequencies; the proof follows directly from Proposition \ref{periodic solutions theorem}.

\begin{corollary}\label{periodic corollary}
Under the notation of Proposition \ref{periodic solutions theorem}, we have
\begin{equation}\label{periodic theorem bounded}
\sup_{\substack{|\mu| \le \mu_{\per}(c) \\ |a| \le a_{\per}(c)}} |\omega_c^{\mu}[a]| + \norm{\phib_c^{\mu}[a]}_{W^{r,\infty}}
< \infty.
\end{equation}
and
\begin{equation}\label{naked psib-c-mu-a est}
\sup_{|\mu| \le \mu_{\per}(c)} \norm{\psib_c^{\mu}[a]}_{W^{r,\infty}} 
\le C(c,r)|a|.
\end{equation}
\end{corollary}

We emphasize that our periodic solutions persist for $\mu = 0$, i.e., for the monatomic lattice.
Such persistence at the zero limit of the small parameter was impossible in the long wave and small mass limits, as there the analogue of the critical frequency diverged to $+\infty$ as the small parameter approached zero.
This cannot happen in our problem, due to the bounds on $\omega_c^{\mu}$ in part \ref{critical frequency bounds} in Proposition \ref{critical frequency props prop}.

We also note that the existence of periodic solutions to the monatomic traveling wave problems has already been established by Friesecke and Mikikits-Leitner \cite{fml} in the long wave limit using a construction inspired by \cite{friesecke-pego1}.
These periodic traveling waves are close to a KdV cnoidal profile, whereas when $\mu = 0$, the expansion \eqref{phib-c-mu-a expansion} says that the periodic solutions from Proposition \ref{periodic solutions theorem} are, to leading order in the amplitude parameter $a$, close to $(0,\sin(\omega_c\cdot))$.
\section{Analysis of the Operators $\L_c$ and $\L_c^*$}\label{analysis of L-c}

Proposition \ref{phase shift theorem - hypos section} characterized the range of $\L_c$, defined in \eqref{B-c Sigma-c defns}, as an operator from $O_q^{r+2}$ to $O_q^r$.  
In this section we prove that theorem, which we restate below with some additional details.

\begin{theorem}\label{abstract phase shift theorem}
Let $|c| > 1$ satisfy Hypotheses \ref{hypothesis 1} and \ref{hypothesis 3}.

\begin{enumerate}[label={\bf(\roman*)},ref={(\roman*)}]

\item\label{injectivity of L}
The operator $\L_c$ is injective from $O_q^{r+2}$ to $O_q^r$ for each $r \ge 0$ and $q \ge q_{\L}(c)$.

\item\label{abstract solvability condition}
There is a nonzero odd function $\gamma_c \in W^{2,\infty}$ such that for $f \in O_q^{r+2}$ and $g \in O_q^r$, with $r \ge 0$ and $q \in [q_{\L}(c),1)$, we have
\begin{equation}\label{L prop}
\L_cf = g
\iff \bunderbrace{\int_{-\infty}^{\infty} g(x)\gamma_c(x) \dx}{\iota_c[g]} = 0. 
\end{equation}
Moreover, $\L_c^*\gamma_c = 0$, where $\L_c^*$ is the $L^2$-adjoint of $\L_c$ and was defined in \eqref{L-c-star defn}.

\item\label{abstract asymptotics of gamma}
The function $\gamma_c$ satisfies the limits
\begin{equation}\label{gamma asymptotics}
\lim_{x \to \infty} |\gamma_c(x)-\sin(\omega_c(x+\vartheta_c))|
= \lim_{x \to \infty} |(\gamma_c)'(x)-\omega_c\cos(\omega_c(x+\vartheta_c))|
= 0.
\end{equation}

\item\label{iota-c continuous functional}
The functional $\iota_c$ is bounded on $O_q^r$ for any $q$, $r \ge 0$, i.e., there is $C(c,q,r) > 0$ such that 
\begin{equation}\label{iota-c functional estimate}
|\iota_c[f]|
\le C(c,q,r)\norm{f}_{r,q}, \ f \in O_q^r.
\end{equation}
\end{enumerate}
\end{theorem}

Throughout the proof of this theorem, we will use a number of properties of the symbol $\tB_c$ of the operator $\B_c$, which are proved in Appendix \ref{proof of B props appendix}.

\begin{proposition}\label{symbol of B}
For $|c| > 1$, the function
\begin{equation}\label{tB-c-mu defn}
\tB_c(z)
:= -c^2z^2+2+2\cos(z)
\end{equation}
has the following properties.

\begin{enumerate}[label={\bf(\roman*)}, ref={(\roman*)}]

\item\label{B simple zeros}
Let $q_{\B} = 1$. 
Then for $z \in S_{3q_{\B}} := \set{z \in \C}{|\im(z)| \le 3q_{\B}}$, we have $\tB_c(z) = 0$ if and only if $z=\pm\omega_c$, where $\omega_c \in (\sqrt{2}/|c|,\pi/2)$ was previously studied in Section \ref{injectivity section}.
The zeros at $z=\pm\omega_c$ are simple.
Additionally, 
\begin{equation}\label{b-b-c ineq}
\inf_{|c| \in (1,\sqrt{2}]} |(\tB_c)'(\omega_c)| > 0.
\end{equation}

\item\label{B meromorphic}
The function $1/\tB_c$ is meromorphic on the strip $S_{3q_{\B}}$.
The only poles of $1/\tB_c$ in $S_{3q_{\B}}$ are $z=\pm\omega_c$, and each of these poles is simple.

\item\label{B quadratic decay}
There exists $r_{\B}(c) > 0$ such that if $z \in S_{3q_{\B}}$ with $|z| \ge r_{\B}(c)$, then
\begin{equation}\label{B quadratic estimate}
\frac{1}{|\tB_c(z)|}
\le \frac{2}{|\re(z)|^2}.
\end{equation}
Moreover,
\begin{equation}\label{r0 sup}
r_0
:= \sup_{|c| \in (1,\sqrt{2}]} r_{\B}(c) < \infty.
\end{equation}

\item\label{B L1 q estimate}
There is a constant $C_{\B} > 0$ such that if $0 < |q| \le q_{\B}$, then 
\begin{equation}\label{inv ft L1 O(q) est}
\sup_{|c| \in (1,\sqrt{2}]} \bignorm{\ft^{-1}\left[\frac{1}{\tB_c(-\cdot+iq)}\right]}_{L^1} \le \frac{C_{\B}}{|q|}.
\end{equation}
\end{enumerate}
\end{proposition}

Last, we mention two useful properties of the one-sided exponentially weighted spaces from Definition \ref{one sided expn weighted spaces defn}.
First, if $q > 0$ and $\fb \in W_q^{r,\infty}(\R,\C^m)$, then $\fb$ vanishes at $-\infty$, and if $\fb \in W_{-q}^{r,\infty}(\R,\C^m)$, then $\fb$ vanishes at $\infty$.
Next, the Sobolev embedding and some calculus tell us that the two-sided exponentially weighted spaces $H_q^r$, defined in \eqref{hrq}, are continuously embedded in $W_{\pm{q}}^{r-1,\infty}$ for $r \ge 1$.

Now we are ready to prove Theorem \ref{abstract phase shift theorem}.
We distribute the proof over the remainder of this section.

\subsection{Injectivity of $\L_c \colon O_q^{r+2} \to O_q^r$}
Recall from the definition of the $H_q^r$ spaces in \eqref{hrq} that if $0 \le r_1 \le r_2$ and $0 \le q_1 \le q_2$, then $O_{q_2}^{r_2} \subseteq O_{q_1}^{r_1}$.
Since we take $q_{\L}(c) \le q$ and $r \ge 0$, it therefore suffices to prove that $\L_c$ is injective from $O_{q_{\L}(c)}^2$ to $O_{q_{\L}(c)}^0$.

We do this by considering $\L_c$ and $\L_c^*$ as unbounded operators on $L^2$ with domain $O_{q_{\L}(c)}^2$.
For convenience, we recall that 
\begin{equation}\label{L-c adjoint defn}
\L_c^*f
= \bunderbrace{c^2f'' + (2+A)f}{\B_cf} + \bunderbrace{2\varsigma_c(x)(2+A)f}{-\Sigma_c^*f}.
\end{equation}
Suppose $\L_c^*f = 0$ for some $f \in O_{q_{\L}(c)}^2.$
Then 
\[
f'' = -\frac{(2+A)f + 2\varsigma_c(x)(2+A)f}{c^2}.
\]
Since $\varsigma_c \in \cap_{r=0}^{\infty} E_{q_{\varsigma}(c)}^r$, we can bootstrap from this equality to obtain $f \in \cap_{r=0}^{\infty} O_{q_{\L}(c)}^r$.
In particular, we have $f \in O_{q_{\L}(c)}^3 \subseteq W_{-q_{\L}(c)}^{2,\infty}$.
Hypothesis \ref{hypothesis 3} then implies that $f = 0$, and so $\L_c^*$ must be injective.
That is, 0 is not an eigenvalue of $\L_c^*$, and so 0 is also not an eigenvalue of $\L_c$.
Hence $\L_c$ is injective on $O_{q_{\L}(c)}^2$.

\subsection{A solution to $\L_c^*f = 0$}\label{mallet-paret approach section}
To obtain the forward implication in \eqref{L prop}, it suffices to find a function $\gamma_c$ with $\L_c^*\gamma_c = 0$.
We will construct a very particular $\gamma_c$ that will give us the asymptotics \eqref{gamma asymptotics}.
We do this in three steps.
First, we find a function $f_c$ satisfying $\L_c^*f_c = 0$ using methods derived from \cite{mallet-paret}.
Then the structure of the operator $\L_c^*$ guarantees that $\tilde{f}_c(x) := f_c(x)-f_c(-x)$ also satisfies $\L_c^*f_c = 0$.
Last, since the function $\varsigma_c$ and all the coefficients in $\L_c^*$ are real-valued, the functions $\re[\tilde{f}_c]$ and $\im[\tilde{f}_c]$ satisfy $\L_c^*\re[\tilde{f}_c] = \L_c^*\im[\tilde{f}_c] = 0$ as well.
We will show that a rescaled version of either $\re[\tilde{f}_c]$ or $\im[\tilde{f}_c]$ has the sinusoidal asymptotics \eqref{gamma asymptotics}.

Consider the problem $\L_c^*f = 0$ and make the ansatz
\[
f(x) = e^{i\omega_cx} + g(x),
\]
where $g \in L_{-q_{\L}(c)}^{\infty}$.
Per \ref{B simple zeros} of Proposition \ref{symbol of B}, we have $\B_ce^{i\omega_c\cdot} = 0$, so $\L_c^*f = 0$ if and only if
\begin{equation}\label{what g-c-mu does}
\L_c^*g = \Sigma_c^*e^{i\omega_c\cdot}.
\end{equation}
We use the definition of $\Sigma_c^*$ in \eqref{L-c adjoint defn} and the property $e^{q_{\varsigma}(c)|\cdot|}\varsigma_c \in L^{\infty}$ to obtain $\Sigma_c^*e^{i\omega_c\cdot} \in L_{-q_{\L}(c)}^{\infty}$.
Let $[\L_c^*]^{-1}$ be the inverse of $\L_c^*$ from $W_{-q_{\L}(c)}^{2,\infty}$ to $L_{-q_{\L}(c)}^{\infty}$, per Hypothesis \ref{hypothesis 2}.
If we set
\begin{equation}\label{f-c-mu defn}
g_c := [\L_c^*]^{-1}\Sigma_c^*e^{i\omega_c\cdot}
\quadword{and}
f_c := e^{i\omega_c\cdot} + g_c
\end{equation}
then $\L_c^*f_c = 0$.

\subsection{Asymptotics of $f_c$ at $\pm\infty$}
Since $g_c \in W_{-q_{\L}(c)}^{2,\infty}$ with $q_{\L}(c) > 0$, we know that $g_c$ vanishes at $\infty$.
Thus we have the asymptotics
\begin{equation}\label{limit of f at infty}
\lim_{x \to \infty} |f_c(x) - e^{i\omega_cx}|
= 0.
\end{equation}

Now we need the asymptotics of $f_c$ at $-\infty$.
For this, we turn to the methods of Mallet-Paret in \cite{mallet-paret}.
Specifically, we follow the proof of his Proposition 6.1.
We rewrite \eqref{what g-c-mu does} as 
\begin{equation}\label{1d version}
\B_cg_c = \bunderbrace{\Sigma_c^*(g_c+e^{i\omega_c\cdot})}{h_c}.
\end{equation}

Since $g_c \in W_{-q_{\L}(c)}^{2,\infty}$, we have $e^{q_{\L}(c)\cdot}g_c \in L^{\infty}$.
The Laplace transform $\lt_-[g_c]$ is therefore defined and analytic on $\re(z) < -q_{\L}(c)$; the definition and some essential properties of this Laplace transform are given in Appendix \ref{laplace transform appendix}.
The bounds on $q_{\L}(c)$ from Hypothesis \ref{hypothesis 3} ensure that we may find $q_{\gamma}(c) > 0$ with
\[
q_{\L}(c)
< q_{\gamma}(c)
< \min\{q_{\varsigma}(c)-q_{\L}(c),q_{\B}\}.
\]
Since $e^{q_{\varsigma}(c)|\cdot|}\varsigma_c$, $e^{q_{\L}(c)\cdot}g_c \in L^{\infty}$ with $q_{\L}(c) < q_{\varsigma}(c)-q_{\L}(c)$, it follows that $e^{(q_{\varsigma}(c)-q_{\L}(c))|\cdot|}h_c \in L^{\infty}$, and therefore $\lt_-[h_c] = \lt_-[\B_cg_c]$ is defined and analytic on $|\re(z)| < (q_{\varsigma}(c)-q_{\L}(c))$.
Elementary properties of the Laplace transform give the bounds
\begin{equation}\label{useful lt bd}
\sup_{|\re(z)| \le q_{\gamma}(c)} |\lt_-[\B_cg_c](z)| < \infty.
\end{equation}

Next, we use the formula \eqref{lt-inv} for the inverse Laplace transform to write
\begin{equation}\label{gb lt inv}
g_c(x)
= \frac{1}{2\pi{i}}\int_{\re(z) = q_{\gamma}(c)} \lt_-[g_c](-z)e^{xz} \dz.
\end{equation}
Then we apply the Laplace transform $\lt_-$ to our advance-delay problem \eqref{1d version} and find
\begin{equation}\label{eqn post lt}
\tB_c(iz)\lt_-[g_c](z)
= \Rfrak_c(z) + \lt_-[h_c](z)
= \Rfrak_c(z) + \lt_-[\B_cg_c](z),
\end{equation}
where the remainder term $\Rfrak_c$ arises from the identities \eqref{lt-deriv id} and \eqref{lt shift id-} for the Laplace transform under derivatives and shifts, respectively.
More precisely, we have
\begin{equation}\label{Rfrak-c defn}
\Rfrak_c(z)
:= 2\int_0^1 \big(g_c(x)e^{z(1-x)}+g_c(-x)e^{z(x-1)}\big)\dx - zg_c(0) - g_c'(0)
\end{equation}
and we note that $\Rfrak_c$ is entire.

Part \ref{B simple zeros} of Lemma \ref{symbol of B} tells us that $\tB_c(iz) \ne 0$ for $0 < |\re(z)| < q_{\B}$ and so $\tB_c(iz) \ne 0$ for $0 \le |\re(z)| < q_{\gamma}(c)$ and $z \ne \pm{i\omega}_c$.
We can therefore solve for $\lt_-[g_c](z)$ in \eqref{eqn post lt} and find
\[
\lt_-[g_c](z)
= \frac{\Rfrak_c(z)+ \lt_-[\B_cg_c](z)}{\tB_c(iz)}, \ 0 \le |\re(z)| < q_{\B}, \ z \ne \pm{i\omega_c}.
\]
For $x < 0$, the formula \eqref{gb lt inv} for $g_c$ then becomes
\begin{equation}\label{pre switcheroo lt}
g_c(x)
= \frac{1}{2\pi{i}}\int_{\re(z) = q_{\gamma}(c)} e^{-xz}\frac{\Ifrak_c(z)}{\tB_c(iz)} \dz,
\qquad
\Ifrak_c(x,z) := \big(\Rfrak_c(-z)+ \lt_-[\B_cg_c](-z)\big).
\end{equation}

This integrand is meromorphic on an open set containing $|\re(z)| < q_{\gamma}(c)$; it has simple poles at $z=\pm{i\omega_c}$ and is analytic elsewhere.
Moreover, the quadratic decay of $\tB_c(z)$ for $|z|$ large from part \ref{B quadratic decay} of Lemma \ref{symbol of B}, the estimate \eqref{useful lt bd}, and the definition of $\Rfrak_c$ in \eqref{Rfrak-c defn} imply
\[
\sup_{|k| \le q_{\gamma}(c)}
\lim_{|q| \to \infty} \left|\frac{e^{-x(k+iq)}\Ifrak_c(k+iq)}{\tB_c(ik-q)}\right|
= 0, \ x < 0.
\]
We can therefore shift the contour of the integral in \eqref{gb lt inv} from $\re(z) = q_{\gamma}(c)$ to $\re(z) = -q_{\gamma}(c)$ and obtain from the residue theorem that
\begin{equation}\label{residue switcheroo}
g_c(x)
= \bunderbrace{\frac{1}{2\pi{i}}\int_{\re(z) = -q_{\gamma}(c)} \frac{e^{-xz}\Ifrak_c(x,z)}{\tB_c(-iz)} \dz}{r_c(x)}
+ \alpha_ce^{i\omega_cx}
+ \beta_ce^{-i\omega_cx},
\end{equation}
where
\begin{equation}\label{res-c+}
\alpha_ce^{i\omega_cx}
= 2\pi{i}\res\left(\frac{e^{-xz}\Ifrak_c(x,z)}{\tB_c(-iz)};z=i\omega_c\right)
\end{equation}
and
\begin{equation}\label{res-c-}
\beta_ce^{-i\omega_cx}
= 2\pi{i}\res\left(\frac{e^{-xz}\Ifrak_c(x,z)}{\tB_c(-iz)};z=-i\omega_c\right).
\end{equation}
Since $\tB_c$ has simple zeros at $z = \pm{i}\omega_c$ by part \ref{B simple zeros} of Lemma \ref{symbol of B}, and since $\Ifrak_c(x,\cdot)$ and $e^{x\cdot}$ are analytic on $|\re(z)| \le q_{\gamma}(c)$, these residues are 
\[
\res\left(\frac{\Ifrak_c(x,z)}{\tB_c(-iz)};z=\pm{i}\omega_c\right) \\
= \frac{e^{\pm{i}\omega_cx}\Ifrak_c(\mp{i}\omega)}{i(\tB_c)'(\pm\omega_c)}.
\]

The strategy of the remainder of Mallet-Paret's proof of his Proposition 6.1 in \cite{mallet-paret} shows that both $r_c(x)$ and $r_c'(x)$ from \eqref{residue switcheroo} vanish as $x \to -\infty$.
Thus
\begin{equation}\label{limit of f at -infty}
\lim_{x \to -\infty} |f_c(x) - e^{i\omega_cx} - \alpha_ce^{i\omega_cx} - \beta_ce^{-i\omega_cx}|
=0.
\end{equation}

\begin{remark}
The limitation of this approach is that we do not have an explicit formula for $g_c$, and so we cannot calculate further the residues in \eqref{res-c+} and \eqref{res-c-}.
When $|c| \gtrsim 1$, we can use the Neumann series and results from \cite{hvl} to produce an explicit formula for $(\B_c-\Sigma_c^*)^{-1}$; we do this in Sections \ref{B-cep inverse} and \ref{B-cep Sigma-cep star inverse Neumann}.
In turn, this does give a formula for $g_c$ and, ultimately, $\alpha_c$ and $\beta_c$.

On the other hand, we point out that our proof here does not use the fact that $\varsigma_c$ solves the monatomic traveling wave problem; instead, we need only the decay property $e^{q_{\varsigma}(c)|\cdot|}\varsigma_c \in L^{\infty}$.
The methods of this section could therefore be applied to much more general advance-delay operators than $\L_c^*$; indeed, Mallet-Paret's results in \cite{mallet-paret} are phrased for a very broad class of such operators.
\end{remark}

\subsection{The phase shift revealed}\label{phase shift revealed}
With $f_c$ from \eqref{f-c-mu defn}, the structure of $\L_c^*$ implies that $\tilde{f}_c(x) :=f_c(x) - f_c(-x)$ also satisfies $\L_c^*\tilde{f}_c = 0$.
Setting
\begin{equation}\label{E-c defn}
\E_c(x)
:= e^{i\omega_cx}-\big(e^{-i\omega_cx}+\alpha_ce^{-i\omega_cx} + \beta_ce^{i\omega_cx}\big),
\end{equation}
we have
\begin{equation}\label{E-c limit}
\lim_{x \to \infty} \left|\tilde{f}_c(x) - \E_c(x)\right|
= 0.
\end{equation}
This follows by estimating
\begin{multline*}
\left|\tilde{f}_c(x) - \big[e^{i\omega_cx}-\big(e^{-i\omega_cx}+\alpha_ce^{-i\omega_cx} + \beta_ce^{i\omega_cx}\big)\big]\right| \\
\\
= \left|f_c(x)-f_c(-x)- \big[e^{i\omega_cx}-\big(e^{-i\omega_cx}+\alpha_ce^{-i\omega_cx} + \beta_ce^{i\omega_cx}\big)\big]\right| \\
\\
\le |f_c(x) - e^{i\omega_cx}|
+ \left|f_c(-x)-\big(e^{-i\omega_cx}+\alpha_ce^{-i\omega_cx} + \beta_ce^{i\omega_cx}\big)\right|
\end{multline*}
and using the limits \eqref{limit of f at infty} and \eqref{limit of f at -infty}.

Now we claim
\begin{equation}\label{E-c claim}
\E_c(x) \ne 0, \ x \in \R.
\end{equation}
We prove this claim in Section \ref{proof of E-c claim}.
Then either $\re[\E_c]$ or $\im[\E_c]$ does not vanish identically on $\R$.
We assume that $\im[\E_c]$ does not vanish.
If it is $\re[\E_c]$ that does not vanish, then the proof still proceeds along the lines of what follows below.

Set $\breve{f}_c(x) := \im[\tilde{f}_c(x)] = \im[f_c(x) - f_c(-x)]$.
Then $\L_c^*\breve{f}_c = 0$ and, from \eqref{E-c limit}, we have
\begin{equation}\label{fcup limit}
\lim_{x \to \infty}
\left| \breve{f}_c(x) - \im[\E_c(x)]\right|
= 0.
\end{equation}
Write
\[
\alpha_c
= \alpha_{c,\r} + i\alpha_{c,\i}
\quadword{and}
\beta_c
= \beta_{c,\r} +i\beta_{c,\i},
\]
where $\alpha_{c,\r}$, $\alpha_{c,\i}$, $\beta_{c,\r}$, $\beta_{c,\i} \in \R$.
Then
\begin{equation}\label{im pre-trig}
\im[\E_c(x)]
= (2 + \alpha_{c,\r}-\beta_{c,\r})\sin(\omega_cx) 
-(\alpha_{c,\i} + \beta_{c,\i})\cos(\omega_cx).
\end{equation}

Now we need the identity
\begin{equation}\label{big trig}
A\sin(\omega{x}) + B\cos(\omega{x})
= \sqrt{A^2+B^2}\sin\left(\omega{x} + \arctan\left(\frac{B}{A}\right)\right),
\end{equation}
valid for $A$, $B$, $\omega \in \R$ with $A \ge 0$ (in the case $A = 0$, we interpret $\arctan(B/0) = \arctan(\pm\infty) = \pm\pi/2$).
Since $\im[\E_c]$ does not vanish identically, at least one of the coefficients $(2 + \alpha_{c,\r}-\beta_{c,\r})$, $(\alpha_{c,\i} + \beta_{c,\i})$ is nonzero.
We then apply the identity \eqref{big trig} directly to \eqref{im pre-trig} and conclude that 
\[
\gamma_c
:= \frac{\sgn(2+\alpha_{c,\r}-\beta_{c,\r})(2+\alpha_{c,\r}-\beta_{c,\r})}{\sqrt{(2+\alpha_{c,\r}-\beta_{c,\r})^2 + (\alpha_{c,\i}+\beta_{c,\i})^2}}\breve{f}_c
\]
satisfies $\L_c^*\gamma_c = 0$ and 
\begin{equation}\label{abstract phase shift limit 1}
\lim_{x \to \infty}
|\gamma_c(x) - \sin(\omega_c(x + \vartheta_c))|
=0,
\end{equation}
where 
\[
\vartheta_c
:= -\frac{1}{\omega_c}\arctan\left(\frac{\alpha_{c,\i} + \beta_{c,\i}}{2 + \alpha_{c,\r}-\beta_{c,\r}}\right).
\]

Since $\L_c^*\gamma_c = 0$, we have the forward implication of \eqref{L prop} in part \ref{abstract solvability condition} of Theorem \ref{abstract phase shift theorem}, and also \eqref{abstract phase shift limit 1} implies the first limit in the asymptotics \eqref{gamma asymptotics} from part \ref{abstract asymptotics of gamma}.
Proving the second limit in \eqref{gamma asymptotics} is essentially a matter of establishing the limit \eqref{fcup limit} with $\breve{f}_c(x)$ replaced by $\partial_x[\breve{f}_c](x)$ and $\Im[\E_c(x)]$ replaced by $\Im[\E_c'(x)]$.  
The validity of this limit with derivatives, in turn, is a consequence of the two representations for $f_c$: first, per \eqref{f-c-mu defn}, as $f_c(x) = e^{i\omega_cx} + g_c(x)$, where $g_c$ and $g_c'$ vanish at $\infty$, and, next, with $g_c$ replaced by its expression in \eqref{residue switcheroo}, in which $r_c'$ decays at $-\infty$.

Last, since $f_c$ and $f_c'$ are asymptotic to bounded functions at $\pm\infty$ by \eqref{limit of f at infty} and \eqref{limit of f at -infty}, it follows that $\gamma_c$, $\gamma_c' \in L^{\infty}$, which implies \eqref{iota-c functional estimate}.
This proves part \ref{iota-c continuous functional}.

\subsection{Characterization of the range of $\L_c$ as an operator from $O_q^{r+2}$ to $O_q^r$}
Here we prove the reverse implication in \eqref{L prop}.
In Section \ref{char of range}, we argued\footnote{This argument used the injectivity of $\L_c$ that was established in Section \ref{injectivity section} for the restricted case $|c| \gtrsim 1$. 
Now we may rely on part \ref{injectivity of L} of Theorem \ref{abstract phase shift theorem} to obtain injectivity for arbitrary $c$ assuming Hypothesis \ref{hypothesis 3}, and the results of Section \ref{char of range} remain true.} that the cokernel of $\L_c$ in $O_q^r$ was one-dimensional, and we characterized the range of $\L_c$ via
\begin{equation}\label{char of range - zfrak}
\L_cf = g, \surjmatter
\iff \bunderbrace{\int_{-\infty}^{\infty} g(x)\tzfrak_c(x) \dx}{\zfrak_c[g]} = 0
\end{equation}
for some odd function $\tzfrak_c \in L_{\loc}^1$ with $\sech^q(\cdot)\tzfrak_c \in L^2$.
On the other hand, we constructed in Section \ref{phase shift revealed} an odd nontrivial function $\gamma_c \in W^{2,\infty}$ such that 
\[
\bunderbrace{\int_{-\infty}^{\infty} (\L_cf)(x)\gamma_c(x) \dx}{\iota_c[\L_cf]}
= 0
\]
for all $f \in O_q^{r+2}$.
The functions $\gamma_c$ and $\tzfrak_c$ must be linearly dependent, as otherwise the functionals $\iota_c$ and $\zfrak_c$ would be linearly independent and $\L_c$ would have a cokernel of dimension at least 2.
The characterization \eqref{char of range - zfrak} therefore implies \eqref{L prop}.

\subsection{Proof of the claim \eqref{E-c claim}}\label{proof of E-c claim}
Fix $q \in [q_{\L}(c),q_{\B})$ and suppose the claim is false, so that $\E_c(x) = 0$ for all $x$.
Then \eqref{E-c limit} and the exponential decay of $g_c$ imply that $\tilde{f}_c(x) = f_c(x)-f_c(-x)$ vanishes exponentially fast at $\pm\infty$, so $\tilde{f}_c \in W_{-q_{\L}(c)}^{2,\infty}$.
But $\L_c^*\tilde{f}_c = 0$, so Hypothesis \ref{hypothesis 3} forces $\tilde{f}_c = 0$.
That is, $f_c$ must be even.

Now observe that $\L_c^*f_c = \L_c^*\overline{f}_c = 0$, where $f_c$ and its complex conjugate $\overline{f}_c$ are linearly independent since they asymptote, respectively, to $e^{i\omega_cx}$ and $e^{-i\omega_cx}$ at $\infty$.
The functionals
\[
\zfrak_{c,1}[g]
:= \int_{-\infty}^{\infty} g(x)f_c(x) \dx
\quadword{and}
\zfrak_{c,2}[g]
:= \int_{-\infty}^{\infty} g(x)\overline{f_c(x)} \dx,
\]
defined for $g \in H_q^0$, are therefore also linearly independent.
Moreover, $\zfrak_{c,1}[\L_cf] = \zfrak_{c,2}[\L_cf] = 0$ for all $f \in H_q^2$.

The methods of Section \ref{char of range} can be adapted to show that $\L_c$ has a two-dimensional cokernel in $H_q^0$ when considered as an operator from $H_q^2$ to $H_q^0$.
Consequently, if $\zfrak$ is any functional on $H_q^0$ such that $\zfrak[\L_cf] = 0$ for all $f \in H_q^2$, then $\zfrak$ must be a linear combination of $\zfrak_{c,1}$ and $\zfrak_{c,2}$.
On the other hand, with $\zfrak_c$ and $\tzfrak_c$ from \eqref{char of range - zfrak}, we already have $\zfrak_c[\L_cf] = 0$ for all $f \in H_q^2$, and so $\zfrak_c$ must be a linear combination of $\zfrak_{c,1}$ and $\zfrak_{c,2}$.
But then the odd function $\tzfrak_c$ must be a linear combination of $f_c$ and $\overline{f}_c$, and so $\tzfrak_c$ is even, a contradiction.
\section{The Micropteron Fixed Point Problem}\label{nanopteron problem section}

\subsection{Beale's ansatz}\label{beale's ansatz section}
We study our problem $\G_c(\rhob,\mu) = 0$ from \eqref{cov simplified} under Beale's ansatz
\[
\rhob = \varsigmab_c + a\phib_c^{\mu}[a] + \etab,
\]
where

\begin{enumerate}[label=$\bullet$]

\item
$\varsigmab_c = (\varsigma_c,0)$ solves $\G_c(\varsigmab_c,0) = 0$, per Hypothesis \ref{hypothesis 1}.

\item
$\phib_c^{\mu}[a]$ is periodic, $a \in \R$ with $|a| \le a_{\per}(c)$, $|\mu| \le \mu_{\per}(c)$, and $\G_c(a\phib_c^{\mu}[a],\mu) = 0$ by Proposition \ref{periodic solutions theorem}.

\item
$\etab \in E_q^r\times O_q^r$ with $r \ge 2$ and $q >0$ to be specified later.
\end{enumerate}

We expand $\G_c(\varsigmab_c + a\varphib_c^{\mu}[a]+\etab,\mu)$ using the bilinearity \eqref{nl symm bil} of $\nl$ and the following decomposition of  $\D_{\mu}$ from \eqref{D-mu defn} as the sum of a diagonal operator and a small perturbation term:
\[
\D_{\mu}
= \D_0 + \bunderbrace{\mu\begin{bmatrix*} (2-A)/2 &\delta \\ -\delta &(2+A)/2\end{bmatrix*}}{\mu\Dring}.
\]
Next, we cancel a number of terms with the existing solutions
\[
\G_c(\varsigmab_c,0) = 0
\quadword{and}
\G_c(a\varphib_c^{\mu}[a],\mu) = 0.
\]
After some further rearrangements, we find that $\G_c(\varsigmab_c + a\varphib_c^{\mu}[a]+\etab) = 0$ is equivalent to
\begin{multline}\label{nanopteron equations 1}
c^2\etab'' 
+ \D_0\etab
+ 2\D_0\nl(\varsigmab_c,\etab)
= -\mu\Dring(\varsigmab_c + \nl(\varsigmab_c,\varsigmab_c))
-\mu\Dring(\etab -2\nl(\varsigmab_c,\etab))
-2a\D_{\mu}\nl(\varsigmab_c,\phib_c^{\mu}[a]) \\
\\
-2a\D_{\mu}\nl(\phib_c^{\mu}[a],\etab)
-\D_{\mu}\nl(\etab,\etab).
\end{multline}
Recalling the definitions of $\H_c$ in \eqref{H-c defn} and $\L_c$ in \eqref{L-c defn}, we see that the left side of this equation is just the diagonal operator $\diag(\H_c,\L_c)$ applied to $\etab$.
Then we can rewrite \eqref{nanopteron equations 1} in the componentwise form
\begin{equation}\label{nanopteron equations 2}
\begin{cases}
\H_c\eta_1
= \sum_{k=1}^5 h_{c,k}^{\mu}(\etab,a) \\
\\
\L_c\eta_2
= \sum_{k=1}^5 \ell_{c,k}^{\mu}(\etab,a),
\end{cases}
\end{equation}
where the multitude of terms on the right side are
\begin{equation}\label{h ell defns}
\begin{aligned}
h_{c,1}^{\mu}(\etab,a)&:= -\mu\Dring(\varsigmab_c + \nl(\varsigmab_c,\varsigmab_c))\cdot\e_1
\qquad&\ell_{c,1}^{\mu}(\etab,a)&:= -\mu\Dring(\varsigmab_c + \nl(\varsigmab_c,\varsigmab_c))\cdot\e_2 \\[5pt]
h_{c,2}^{\mu}(\etab,a) &:= -\mu\Dring(\etab -2\nl(\varsigmab_c,\etab))\cdot\e_1 
\qquad&\ell_{c,2}^{\mu}(\etab,a) &:= -\mu\Dring(\etab -2\nl(\varsigmab_c,\etab))\cdot\e_2 \\[5pt]
h_{c,3}^{\mu}(\etab,a) &:= -2a\D_{\mu}\nl(\varsigmab_c,\phib_c^{\mu}[a])\cdot\e_1 
\qquad&\ell_{c,3}^{\mu}(\etab,a) &:= -2a\D_{\mu}\nl(\varsigmab_c,\phib_c^{\mu}[a])\cdot\e_2 \\[5pt]
h_{c,4}^{\mu}(\etab,a) &:= -2a\D_{\mu}\nl(\phib_c^{\mu}[a],\etab)\cdot\e_1 
\qquad&\ell_{c,4}^{\mu}(\etab,a) &:= -2a\D_{\mu}\nl(\phib_c^{\mu}[a],\etab)\cdot\e_2 \\[5pt]
h_{c,5}^{\mu}(\etab,a) &:= -\D_{\mu}\nl(\etab,\etab)\cdot\e_1 
\qquad&\ell_{c,5}^{\mu}(\etab,a) &:= -\D_{\mu}\nl(\etab,\etab)\cdot\e_2.
\end{aligned}
\end{equation}
There are, indeed, many terms here, but the salient features are that $h_{c,k}^{\mu}(\etab,a) \in E_{q,0}^r$ and $\ell_{c,k}^{\mu}(\etab,a) \in O_q^r$ for $\etab \in E_q^r \times O_q^r$ and $a \in \R$ and that most of these terms are ``small.''
For example, $h_{c,1}^{\mu}$ is $\O_c(\mu)$, $h_{c,2}^{\mu}$ is, very roughly, of the form $\mu\etab$, and $h_{c,4}^{\mu}$ and $h_{c,5}^{\mu}$ are quadratic in $\etab$ and $a$.
The terms $h_{c,3}^{\mu}$ and $\ell_{c,3}^{\mu}$ are more complicated and merit more precise analysis later.
We will use the explicit algebraic structure of the terms and their smallness frequently in our subsequent proofs.

\subsection{Construction of the equation for $\eta_1$}\label{eta1 eqn section}
We first extend Hypothesis \ref{hypothesis 2} to allow $\H_c$ to be invertible over a broader range of exponentially localized spaces.
The proof of the next proposition is in Appendix \ref{H-c invert general appendix}.

\begin{proposition}\label{H-c invert general}
Assume Hypotheses \ref{hypothesis 1} and \ref{hypothesis 2}.
There exist $q_{\H}^{\star}(c)$, $q_{\H}^{\star\star}(c)$ with $q_{\H}(c) \le q_{\H}^{\star} (c)< q_{\H}^{\star\star} (c)< \min\{1,q_{\varsigma}(c)\}$ such that for any $q \in [q_{\H}^{\star}(c),q_{\H}^{\star\star}(c)]$ and $r \ge 0$, the operator $\H_c$ is invertible from $E_q^{r+2}$ to $E_{q,0}^r$.
\end{proposition}

From now on we fix $q \in (\max\{q_{\L}(c),q_{\H}^{\star}(c)\},\min\{q_{\varsigma}(c),q_{\H}^{\star\star}(c),1\})$.
Since $q_{\H}^{\star}(c) < \min\{q_{\varsigma}(c),q_{\H}^{\star\star}(c),1\}$ and $q_{\L}(c) < \min\{q_{\varsigma}(c),q_{\H}(c),1\} < q_{\H}^{\star\star}(c)$ by Hypothesis \ref{hypothesis 3}, this interval is nonempty.
Moreover, since $q \in (q_{\H}^{\star}(c),q_{\H}^{\star\star}(c))$, Proposition \ref{H-c invert general} tells us that $\H_c$ is invertible from $E_q^{r+2}$ to $E_{q,0}^r$ for any $r \ge 0$.
We then invert $\H_c$ in the first equation in \eqref{nanopteron equations 2} to obtain a fixed point equation for $\eta_1$:
\begin{equation}\label{N1 defn}
\eta_1 
= -\H_c^{-1}\sum_{j=1}^6 h_{c,k}^{\mu}(\etab,a) 
=: \nano_{c,1}^{\mu}(\etab,a).
\end{equation}

\subsection{Construction of the fixed point equations for $a$}\label{eta2 a eqn constr} 
From the system \eqref{nanopteron equations 2}, the unknowns $\eta_2$ and $a$ must satisfy
\begin{equation}\label{eta2 eqn redux}
\L_c\eta_2 
= \sum_{k=1}^5 \ell_{c,k}^{\mu}(\etab,a)
\end{equation}
We formally differentiate the right side of \eqref{eta2 eqn redux} to isolate a term containing a factor of $a$:
\begin{equation}\label{chi-c-mu defn}
\partialx{a}\left[\sum_{k=1}^5 \ell_{c,k}^{\mu}(\etab,a)\right]\bigg|_{\etab=0,a=0}
= \partialx{a}[\ell_{c,3}^{\mu}(\etab,a)]\big|_{\etab=0,a=0}
= -2\D_{\mu}\nl(\varsigmab_c,\phib_c^{\mu}[0])\cdot\e_2
=: -\chi_c^{\mu}.
\end{equation}
We calculate this formal derivative by recalling the definitions of $\ell_{c,1}^{\mu},\ldots,\ell_{c,5}^{\mu}$ in \eqref{h ell defns} and observing that all the terms except $\ell_{c,3}^{\mu}$ are either constant in $a$ or quadratic in $\etab$ and $a$, in which case their derivatives with respect to $a$ at $\etab = 0$ and $a = 0$ are zero.
We expect, therefore, that the term 
\[
\ell_{c,3}^{\mu}(\etab,a) + a\chi_c^{\mu}
\]
will be roughly quadratic in $a$.
We could then write
\[
\sum_{k=1}^5 \ell_{c,k}^{\mu}(\etab,a)
= -a\chi_c^{\mu} + \bunderbrace{\sum_{\substack{k=1 \\ k \ne 3}}^5 \ell_{c,k}^{\mu}(\etab) + \big(\ell_{c,3}^{\mu}(\etab,a)+a\chi_c^{\mu}\big)}{\small\text{these terms are ``small''}}.
\]
However, it turns out to be more convenient not to work with $\chi_c^{\mu}$ but instead with
\begin{equation}\label{chi-c defn}
\chi_c
:= \chi_c^0
= 2\D_{0}\nl(\varsigmab_c,\phib_c^{0}[0])\cdot\e_2
= (2+A)(\varsigma_c\sin(\omega_c\cdot)).
\end{equation}

We now set
\begin{equation}\label{tells}
\tell_{c,k}^{\mu}(\etab,a)
:= \begin{cases}
\ell_{c,k}^{\mu}(\etab,a), k \ne 3 \\
\ell_{c,3}^{\mu}(\etab,a) + a\chi_c, \ k = 3,
\end{cases}
\end{equation}
so that \eqref{eta2 eqn redux} is equivalent to 
\begin{equation}\label{ell to tell}
\L_c\eta_2 + a\chi_c
= \sum_{k=1}^5 \tell_{c,k}^{\mu}(\etab,a).
\end{equation}
Then we apply the functional $\iota_c$, defined in \eqref{L prop}, to both sides of \eqref{ell to tell} to find
\begin{equation}\label{pre-a eqn}
a\iota_c[\chi_c]
= \sum_{k=1}^5 \iota_c[\tell_{c,k}^{\mu}(\etab,a)].
\end{equation}
The asymptotics \eqref{gamma asymptotics} on $\gamma_c$ give us an explicit formula for $\iota_c[\chi_c]$, which we prove in Appendix \ref{iota-chi formula proof appendix}:
\begin{equation}\label{iota-chi formula}
\iota_c[\chi_c] 
= (2c^2\omega_c-\sin(\omega_c))\sin(\omega_c\vartheta_c)
\end{equation}

Since $|c| > 1$, we have
\begin{equation}\label{first iota-chi estimate}
2c^2\omega_c-\sin(\omega_c)
> \omega_c-\sin(\omega_c)
> 0.
\end{equation}
We assumed in Hypothesis \ref{hypothesis 4} that $\sin(\omega_c\vartheta_c) \ne 0$, and so we have $\iota_c[\chi_c] \ne 0$.
Then \eqref{pre-a eqn} is equivalent to
\begin{equation}\label{A defn}
a
= \frac{1}{\iota_c[\chi_c]}\sum_{k=1}^5 \iota_c[\tell_{c,k}^{\mu}(\etab,a)]
=: \nano_{c,3}^{\mu}(\etab,a).
\end{equation}

\subsection{Construction of the fixed point equation for $\eta_2$}
We substitute the fixed point equation \eqref{A defn} for $a$ into \eqref{ell to tell} to find
\begin{equation}\label{eta2 before P}
\L_c\eta_2
= \sum_{k=1}^5 \tell_{c,k}^{\mu}(\etab,a) -a\chi_c
= \sum_{k=1}^5 \tell_{c,k}^{\mu}(\etab,a) - \frac{1}{\iota_c[\chi_c]}\iota_c\left[\sum_{k=1}^6 \tell_{c,k}^{\mu}(\etab,a)\right]\chi_c.
\end{equation}
If we set
\begin{equation}\label{P-c-mu defn}
\P_cf
:= f - \frac{\iota_c[f]}{\iota_c[\chi_c]}\chi_c,
\end{equation}
then \eqref{eta2 before P} can be written more succinctly as
\begin{equation}\label{eta2 with P}
\L_c\eta_2 = \P_c\sum_{k=1}^5 \tell_{c,k}^{\mu}(\etab,a).
\end{equation}
It is obvious from the definition of $\chi_c$ in \eqref{chi-c defn} that $\chi_c \in \cap_{r=0}^{\infty} O_{q_{\varsigma}(c)}^r \subseteq \cap_{r=0}^{\infty} O_q^r$, and so $\P_cf \in O_q^r$ for any $f \in O_q^r$.
Also, a straightforward calculation shows that $\iota_c[\P_cf] = 0$.

Now, recall that in Section \ref{eta1 eqn section} we specified $q$ so that $q \in (q_{\L}(c),1)$.
Then part \ref{abstract solvability condition} of Theorem \ref{abstract phase shift theorem} implies that $\P_cf$ is in the range of $\L_c$.
Conversely, if $\L_cf = g$, then $\iota_c[g] = 0$ and so $\P_cg = g$.
That is, $\P_c[O_q^r] = \L_c[O_q^{r+2}]$.
The injectivity of $\L_c$ on $O_q^r$ for $r \ge 1$, established as part \ref{injectivity of L} of Theorem \ref{abstract phase shift theorem}, then implies that $\L_c$ is bijective from $O_q^{r+2}$ to $\P_c[O_q^r]$.
The functional $\iota_c$ is continuous on $O_q^r$ by part \ref{iota-c continuous functional} of Theorem \ref{abstract phase shift theorem}, and the subspace
\[
\P_c[O_q^r]
= \L_c[O_q^{r+2}]
= \iota_c^{-1}(\{0\}) \cap O_q^r
\]
is closed in $O_q^r$.
Hence the inverse of $\L_c$ from $\P_c[O_q^r]$ to $O_q^{r+2}$ is a bounded operator, which we denote by $\L_c^{-1}$.

We are now able to rewrite \eqref{eta2 with P} as a fixed point equation for $\eta_2$:
\begin{equation}\label{N2 defn}
\eta_2 
= \L_c^{-1}\P_c\sum_{k=1}^5 \tell_{c,k}^{\mu}(\etab,a)
=: \nano_{c,2}^{\mu}(\etab,a).
\end{equation}
With our previously constructed equations \eqref{N1 defn} for $\eta_1$ and \eqref{A defn} for $a$, this gives us a fixed point problem for our three unknowns from Beale's ansatz:
\begin{equation}\label{nanob defn}
(\etab,a)
= (\nano_{c,1}^{\mu}(\etab,a), \nano_{c,2}^{\mu}(\etab,a), \nano_{c,3}^{\mu}(\etab,a))
=: \nanob_c^{\mu}(\etab,a).
\end{equation}

\subsection{Solution of the full fixed point problem}
We will solve the fixed point problem \eqref{nanob defn} using the following lemma, which was stated and proved as Lemma 4.10 in \cite{johnson-wright}.

\begin{lemma}\label{nano fp lemma}
Let $\X_0$ and $\X_1$ be reflexive Banach spaces with $\X_1 \subseteq \X_0$.
For $r > 0$, let $\Bfrak(r) := \set{x \in \X_1}{\norm{x}_{\X_1} \le r}$.
Suppose that for some $r_0 > 0$, there is a map $\nano \colon \Bfrak(r_0) \to \X_1$ with the following properties.

\begin{enumerate}[label={\bf(\roman*)}]

\item
$
\norm{x}_{\X_1} \le r_0 
\Longrightarrow \norm{\nano(x)}_{\X_1} \le r_0.
$

\item
There exists $\alpha \in (0,1)$ such that 
\[
\norm{x}_{\X_0}, \norm{\grave{x}}_{\X_1} \le r_0
\Longrightarrow
\norm{\nano(x)-\nano(\grave{x})}_{\X_0} 
\le \alpha\norm{x-\grave{x}}_{\X_0}.
\]
\end{enumerate}
Then there exists a unique $x_{\star} \in \X_1$ such that $\norm{x_{\star}}_{\X_1} \le r_0$ and $x_{\star} = \nano(x_{\star})$.
\end{lemma}

To invoke this lemma, we first need to specify the underlying Banach spaces.
With $q_{\H}^{\star}(c)$ and $q_{\H}^{\star\star}(c)$ from Proposition \ref{H-c invert general}, we fix $q_{\star}(c) \in (\max\{q_{\L}(c),q_{\H}^{\star}(c)\},\min\{q_{\varsigma}(c),q_{\H}^{\star\star}(c),1\})$ and $\qbar_{\star}(c) \in (q_{\H}(c),q_{\star}(c))$.
Let
\[
\X^r
:= \begin{cases}
E_{\qbar_{\star}(c)}^1 \times O_{\qbar_{\star}(c)}^1 \times \R, \ r = 0 \\
\\
E_{q_{\star}(c)}^r \times O_{q_{\star}(c)}^r \times \R, \ r \ge 1
\end{cases}
\]
and, for $r \ge 1$ and $\tau > 0$, set
\[
\U_{\tau,\mu}^r
:= \set{(\etab,a) \in \X^r}{\norm{\etab}_{r,q_{\star}(c)} + |a| \le \tau|\mu|}.
\]
The spaces $\X^r$ are Hilbert spaces and therefore they are reflexive, and the sets $\U_{\tau,\mu}^r$ are the balls of radius $\tau|\mu|$ centered at the origin in $\X^r$.

The next proposition shows that $\nanob_c^{\mu}$ satisfies the estimates from Lemma \ref{nano fp lemma}.
Its proof is in Appendix \ref{proof of main workhorse proposition appendix}.

\begin{proposition}\label{main workhorse proposition}
Assume that $|c| > 1$ satisfies Hypotheses \ref{hypothesis 1}, \ref{hypothesis 2}, \ref{hypothesis 3}, and \ref{hypothesis 4}.
There exist $\mu_{\star}(c)$, $\tau_c > 0$ such that if $|\mu| \le \mu_{\star}(c)$, then the following hold.

\begin{enumerate}[label={\bf(\roman*)},ref={(\roman*)}]

\item\label{main workhorse map part}
$(\etab,a) \in \U_{\tau_c,\mu}^1 
\Longrightarrow \nanob_c^{\mu}(\etab,a) \in \U_{\tau_c,\mu}^1$. 

\item\label{main workhorse lip part}
$(\etab,a), (\grave{\etab},\grave{a}) \in \U_{\tau_c,\mu}^1 
\Longrightarrow \norm{\nanob_c^{\mu}(\etab,a)-\nanob_c^{\mu}(\grave{\etab},\grave{a})}_{\X^0} 
\le \frac{1}{2}\norm{(\etab,a)-(\grave{\etab},\grave{a})}_{\X^0}$.

\item\label{main workhorse boot part}
For any $\tau > 0$, there is $\tilde{\tau} > 0$ such that 
\[
(\etab,a) \in \U_{\tau_c,\mu}^1 \cap \U_{\tau,\mu}^r
\Longrightarrow \nanob_c^{\mu}(\etab,a) \in \U_{\tilde{\tau},\mu}^{r+1}.
\]
\end{enumerate}
\end{proposition}

Lemma \ref{nano fp lemma} therefore applies to produce a unique $(\etab_c^{\mu},a_c^{\mu}) \in \U_{\tau_c,\mu}^1$ such that $(\etab_c^{\mu},a_c^{\mu}) = \nanob_c^{\mu}(\etab_c^{\mu},a_c^{\mu})$.
We then bootstrap with part \ref{main workhorse boot part} of Proposition \ref{main workhorse proposition} to conclude that $\etab_c^{\mu}$ is smooth, and so we obtain our main result.

\begin{theorem}\label{main theorem formal}
Assume that $|c| > 1$ satisfies Hypotheses \ref{hypothesis 1}, \ref{hypothesis 2}, \ref{hypothesis 3}, and \ref{hypothesis 4} and take $\mu_{\star}(c)$ and $\tau_c$ from Proposition \ref{main workhorse proposition}.
Then for each $|\mu| \le \mu_{\star}(c)$, there exists a unique $(\etab_c^{\mu},a_c^{\mu}) \in \U_{\tau_c,\mu}^1$ such that $(\etab_c^{\mu},a_c^{\mu}) = \nanob_c^{\mu}(\etab_c^{\mu},a_c^{\mu})$.
Moreover, $\etab_c^{\mu} \in \cap_{r=0}^{\infty} E_{q_{\star}(c)}^r \times O_{q_{\star}(c)}^r$ and, for each $r \ge 0$, there is $C(c,r) > 0$ such that 
\begin{equation}\label{final estimate}
\norm{\etab_c^{\mu}}_{r,q_{\star}(c)} + |a_c^{\mu}| 
\le C(c,r)|\mu|.
\end{equation}
\end{theorem}

\begin{remark}\label{why no small beyond all orders}
The estimate \eqref{final estimate} shows that the amplitude of the ripple is $a_c^{\mu} = \O_c(\mu)$,
which is not necessarily small beyond all orders of $\mu$.
Recall from \eqref{A defn} that, roughly, $a_c^{\mu} = \iota_c[\V_c^{\mu}(\etab_c^{\mu},a_c^{\mu})]$, where $\V_c^{\mu}$ maps $O_q^r \times \R$ to $O_q^r$.
Per \eqref{iota-c functional estimate}, we have an estimate of the form
\[
|a_c^{\mu}|
= |\iota_c[\V_c^{\mu}(\etab_c^{\mu},a_c^{\mu})]|
\le C(c,r)\norm{\V_c^{\mu}(\etab_c^{\mu},a_c^{\mu})}_{r,q_{\star}(c)}.
\]
The analogues of $\iota_c$ in the lattice nanopteron problems \cite{faver-wright}, \cite{hoffman-wright}, and \cite{faver-spring-dimer} all depended on $\mu$ and, if we denote one of these functionals by $\tilde{\iota}_{\mu}$, roughly had an estimate of the form
\[
|\tilde{\iota}_{\mu}[f]|
\le C(q,r)|\mu|^r\norm{f}_{r,q}
\]
for $f \in H_q^r$.
The proof of this estimate parallels the Riemann-Lebesgue lemma and closely relied on the fact that the analogue in those problems of the critical frequency $\omega_c^{\mu}$ was roughly $\O(\mu^{-1})$.
In turn, this ``high frequency'' estimate hinged on the existence of a singular perturbation in the problem.
Our equal mass problem is not singularly perturbed, our critical frequency $\omega_c^{\mu}$ remains bounded as $\mu \to 0$, and our ripple is not necessarily small beyond all orders of $\mu$.
\end{remark}
\section{The Proof of Theorem \ref{hypotheses theorem}: Verification of the Hypotheses for $|c| \gtrsim 1$}\label{hypotheses verification section}

\subsection{Verification of Hypothesis \ref{hypothesis 1}}
We extract the following result from Theorem 1.1 in \cite{friesecke-pego1}.

\begin{theorem}[Friesecke \& Pego]\label{friesecke-pego}
There exist $\ep_{\FP} > 0$, $q_{\FP} \in (0,1)$ such that if $\ep \in (0,\ep_{\FP})$, then there exists a unique positive function $\varsigma_{\cep} \in \cap_{r=1}^{\infty} E_{\ep{q}_{\FP}}^r$ with the following properties.

\begin{enumerate}[label={\bf(\roman*)}]

\item
If
\begin{equation}\label{cep defn}
\cep 
:= \left(1+\frac{\ep^2}{24}\right)^{1/2},
\end{equation}
then
\[
\cep^2\varsigma_{\cep}'' + (2-A)(\varsigma_{\cep}+\varsigma_{\cep}^2) = 0.
\]

\item
If
\begin{equation}\label{sigma defn}
\sigma(x) := \frac{1}{4}\sech^2\left(\frac{x}{2}\right),
\end{equation}
then for each $r \ge 0$, there is a constant $C(r) > 0$ such that 
\begin{equation}\label{original friesecke-pego estimate}
\bignorm{\frac{1}{\ep^2}\varsigma_{\cep}\left(\frac{\cdot}{\ep}\right) - \sigma}_{H^r} \le C(r)\ep^2
\end{equation}
for all $0 < \ep < \ep_{\FP}$.
\end{enumerate}
\end{theorem}

We set 
\begin{equation}\label{q-varsigma-cep}
q_{\varsigma}(\cep)
:= \ep{q}_{\FP}
\end{equation}
in Hypothesis \ref{hypothesis 1}.
We remark that Friesecke and Pego label the decay rate for their profile $\varsigma_c$ as $b_0(c)$; part (c) of their Theorem 1.1 and their detailed proof in Lemma 3.1 reveal that, with $\cep$ from \eqref{cep defn}, they have $b_0(\cep) = \O(\ep)$. It is convenient for us to make explicit this order-$\ep$ dependence and rewrite this decay rate as $b_0(\cep) = \ep{q}_{\FP}$.

\subsection{Verification of Hypothesis \ref{hypothesis 2}}
As we mentioned in Section \ref{linearize at fp section} when we met the operator $\H_c$, the invertibility of $\H_c$ for $|c| \approx 1$ arises from Proposition 3.1 in \cite{hoffman-wright}.
Here is that proposition.

\begin{proposition}[Hoffman \& Wright]\label{hoffman-wright H inverse}
There exists $\ep_{\HWr} \in (0,\ep_{\FP}]$ such that for $0 < \ep < \ep_{\HWr}$, $r \ge 0$, and $0 < q < \ep{q}_{\FP}$, the operator $\H_{\cep}$, defined in \eqref{H-c defn}, is invertible from $E_q^{r+2}$ to $E_{q,0}^r$.
\end{proposition}

For the precise value of the decay rate in Hypothesis \ref{hypothesis 2}, we will set 
\begin{equation}\label{q-H-cep}
q_{\H}(\cep) 
:= \frac{\ep{q}_{\FP}}{2}.
\end{equation} 
Recalling the definition of $q_{\varsigma}(\cep)$ in \eqref{q-varsigma-cep} and that we took $q_{\FP} < 1$ in Theorem \ref{friesecke-pego}, we have $q_{\H}(\cep) < \min\{1,q_{\varsigma}(\cep)\}$.

We wish to point out that the crux of the proof of Proposition \ref{hoffman-wright H inverse} by Hoffman and Wright is a clever factorization of $\H_{\cep}$ as a product of two operators, one of which is invertible from $E_q^{r+2}$ to $E_{q,0}^r$ due to Fourier multiplier theory proved by Beale in \cite{beale1} (which later appears as Lemma \ref{beale fm} in this paper), and the other of which is ultimately a small perturbation of an operator that Friesecke and Pego \cite{friesecke-pego1} prove is invertible from $E^r$ to $E^r$.  Hoffman and Wright apply operator conjugation (cf. Appendix D of \cite{faver-dissertation}) to extend the invertibility of this second operator from $E_q^r$ to $E_q^r$.  
It is interesting to note that this perturbation argument is the only time that Hoffman and Wright need to assume that their wave speed is particularly close to 1; for the rest of their paper, the small parameter $|c|-1$ does not play an explicit role, as it does for us in the concrete verification of our four main hypotheses.

We also note that we are writing $\ep{q}_{\FP}$ for the decay rate of our Friesecke-Pego solitary wave profile $\varsigma_{\cep}$; Hoffman and Wright use the notation $b_c$ for the decay rate of their Friesecke-Pego profile, which they denote by $\sigma_c$. 
In turn, this $b_c$ is equal to what Friesecke and Pego call $b_0(c)$.

\subsection{Verification of Hypothesis \ref{hypothesis 3} in the case $|c| \gtrsim 1$}\label{verification of hypothesis 3}

\subsubsection{Preliminary remarks}\label{preliminary remarks}
We will chose our decay rates $q$ to depend in a very precise way on $\ep$.
First, let
\begin{equation}\label{m}
\bpzc
:= \min\left\{\frac{1}{8},\frac{\apzc}{4},\frac{q_{\FP}}{4}\right\},
\end{equation}
where the constant $\apzc$ is defined below in Lemma \ref{hoffman-wayne lemma}.
Next, let
\begin{equation}\label{ep-B}
\ep_{\B}
:= \min\left\{\frac{\ep_{\HWa}}{2\bpzc},\frac{1}{2\bpzc},1\right\},
\end{equation}
where the threshold $\ep_{\HWa}$ comes from Lemma \ref{hoffman-wayne lemma}.
Last, for $0 < \ep < \ep_{\B}$, define
\begin{equation}\label{q-params}
q_{\ep} 
:=\bpzc\ep.
\end{equation}
These definitions ensure the useful bound
\begin{equation}\label{very hungry inequality}
0 
< q_{\ep}
< \min\left\{\frac{\ep}{4},\frac{\apzc\ep}{2},\frac{\ep{q}_{\FP}}{4},1\right\}
\end{equation}
and also, per \eqref{cep defn}, $1 < |\cep| \le \sqrt{2}$.

\subsubsection{Inversion of the operator $\B_{\cep}$}\label{B-cep inverse}
Here is our precise statement about the invertibility of $\B_{\cep}$.
Its proof relies on Theorem \ref{hvl theorem}, due to Hupkes and Verduyn Lunel \cite{hvl}.

\begin{proposition}\label{B inv prop}
For $0 < \ep < \ep_{\B}$, the operator $\B_{\cep}$ is invertible from $W_{-q_{\ep}}^{2,\infty}$ to $L_{-q_{\ep}}^{\infty}$ and from $W_{q_{\ep}}^{2,\infty}$ to $L_{q_{\ep}}^{\infty}$.
These spaces were defined in Definition \ref{one sided expn weighted spaces defn}.
We denote the inverse from $L_{-q_{\ep}}^{\infty}$ to $W_{-q_{\ep}}^{2,\infty}$ by $[\B_{\cep}^-]^{-1}$ and from $L_{q_{\ep}}^{\infty}$ to $W_{q_{\ep}}^{2,\infty}$ by $[\B_{\cep}^+]^{-1}$.
These operators have the estimates
\begin{equation}\label{B-cep norm inv est}
\norm{[\B_{\cep}^{\pm}]^{-1}}_{\b(L_{\pm{q}_{\ep}}^{\infty}, W_{\pm{q}_{\ep}}^{2,\infty})} 
= \O(\ep^{-1}).
\end{equation}
\end{proposition}

\begin{proof}
We convert the problem $\B_{\cep}f = g$ with $f \in W_{-q_{\ep}}^{2,\infty}$, $g \in L_{-q_{\ep}}^{\infty}$ to an equivalent first-order system.
Let
\begin{equation}\label{A-c-mu defn}
\A_c
:= 
\bunderbrace{\begin{bmatrix*}[r]
0 &1 \\
-2 &0
\end{bmatrix*}}{A_0(c)}
+ \bunderbrace{\begin{bmatrix*}
0 &0 \\
-2/c^2 &0 
\end{bmatrix*}}{A_1(c)}S^1
+ \begin{bmatrix*}
0 &0 \\
-2/c^2 &0
\end{bmatrix*}S^{-1},
\end{equation}
Then $\B_{\cep}f = g$ is equivalent to
\begin{equation}\label{actual first order syst}
\bunderbrace{\partial_x[\fb] - \A_{\cep}\fb}{\Lambda_{\ep}\fb}
=\gb, 
\qquad \fb := \begin{pmatrix*} f \\ f' \end{pmatrix*},
\qquad \gb := \begin{pmatrix*} 0 \\ g \end{pmatrix*}.
\end{equation}

Now set
\begin{equation}\label{mp version ce}
\Delta_c(z)
:= z\ind - (A_0(c) + A_1(c)e^z + A_1(c)e^{-z})
= \begin{bmatrix*}
z &-1 \\
\frac{1}{c^2}\big(2+2\cosh(z)\big) &z
\end{bmatrix*}
\end{equation}
and observe that 
\begin{equation}\label{mp version det}
\det(\Delta_c(z))
= z^2 + \frac{1}{c^2}\big(2+2\cosh(z)\big)
= \frac{1}{c^2}\tB_c(iz),
\end{equation}
where $\tB_c$ is defined in \eqref{tB-c-mu defn}.

Using the terminology of Theorem \ref{hvl theorem}, the determinant of the characteristic equation of \eqref{actual first order syst} is $\tB_{\cep}(iz)/\cep^2$.
By \eqref{very hungry inequality}, we have $0 < q_{\ep} < 1$, and so Proposition \ref{symbol of B} implies that $\tB_{\cep}(iz) \ne 0$ for $0 < |\re(z)| < q_{\ep}$.
We have therefore satisfied the hypotheses of Theorem \ref{hvl theorem} for the system \eqref{actual first order syst}, and so we conclude that $\Lambda_{\ep}$ is invertible from $W_{\pm{q}_{\ep}}^{1,\infty}$ to $L_{\pm{q}_{\ep}}^{\infty}$.

We will refer to this inverse as $(\Lambda_{\ep}^{\pm})^{-1}$.
More precisely, given $\gb \in L_{q_{\ep}}^{\infty}$, let $(\Lambda_{\ep}^+)^{-1}\gb$ be the unique element of $W_{q_{\ep}}^{1,\infty}$ such that $\Lambda_{\ep}[(\Lambda_{\ep}^+)^{-1}\gb] = \gb$.
Likewise, given $\gb \in L_{-q_{\ep}}^{\infty}$, let $(\Lambda_{\ep}^-)^{-1}\gb$ be the unique element of $W_{-q_{\ep}}^{1,\infty}$ such that $\Lambda_{\ep}[(\Lambda_{\ep}^-)^{-1}\gb] = \gb$.
If we set 
\[
\delta_{\ep}
:= \frac{\bpzc\ep}{2},
\]
where $\bpzc$ was defined in \eqref{m}, then \eqref{Lambda inv formula} gives a formula for $(\Lambda_{\ep}^{\pm})^{-1}\gb$:
\begin{multline}\label{1st order syst inverse}
(\Lambda_{\ep}^{\pm})^{-1}\gb)(x)
= \frac{1}{2\pi{i}}\int_{\pm{q}_{\ep} + \delta_{\ep} -i\infty}^{\pm{q}_{\ep} + \delta_{\ep} +i\infty} e^{xz}\Delta_{\cep}(z)^{-1}\lt_+[\gb](z)\dz \\
\\
+ \frac{1}{2\pi{i}}\int_{\pm{q}_{\ep}- \delta_{\ep} - i\infty}^{\pm{q}_{\ep} - \delta_{\ep} +i\infty}e^{xz}\Delta_{\cep}(z)^{-1}\lt_-[\gb](z) \dz.
\end{multline}
The function $\Delta_{\cep}$ was defined in \eqref{mp version ce}.

Now, recall that we have converted the one-dimensional problem $\B_{\cep}f = g$ to the vectorized equation $\Lambda_{\ep}\fb = \gb$ with $\fb = (f,f')$ and $\gb = (0,g)$, and so we are really interested in solving $\Lambda_{\ep}\fb = \gb$ for the first component of $\fb$.
We therefore take the dot product of \eqref{1st order syst inverse} with the vector $\e_1 = (1,0)$ and use the definition of $\Delta_{\cep}(z)$ in \eqref{mp version ce} to find
\begin{multline}\label{B inv formula}
[\B_{\cep}^{\pm}]^{-1}g
:= ((\Lambda_{\ep}^{\pm})^{-1}\gb)\cdot\e_1
= \frac{1}{2\pi{i}}\int_{\pm{q}_{\ep} + \delta_{\ep} -i\infty}^{\pm{q}_{\ep} + \delta_{\ep} +i\infty} \K_{\ep}(x,z)\lt_+[g](z)\dz \\
\\
+ \frac{1}{2\pi{i}}\int_{\pm{q}_{\ep}- \delta_{\ep} - i\infty}^{\pm{q}_{\ep} - \delta_{\ep} +i\infty}\K_{\ep}(x,z)\lt_-[g](z) \dz,
\end{multline}
where
\begin{equation}\label{K defn}
\K_{\ep}(x,z)
:= \frac{e^{xz}}{\tB_{\cep}(iz)}.
\end{equation}
As above, given $g \in L_{\pm{q}_{\ep}}^{\infty}$, the function $[\B_{\cep}^{\pm}]^{-1}g$ is the unique element of $W_{\pm{q}_{\ep}}^{1,\infty}$ to satisfy $\B_{\cep}[[\B_{\cep}^{\pm}]^{-1}g] = g$.
Note that, in fact, $[\B_{\cep}^{\pm}]^{-1}g \in W_{\pm{q}_{\ep}}^{2,\infty}$, since $\partial_x^2[[\B_{\cep}^{\pm}]^{-1}g] = \partial_x[(\Lambda_{\ep}^{\pm})^{-1}\gb \cdot \e_2] \in L_{\pm{q}_{\ep}}^{\infty}$.

Last, we prove the operator norm estimate \eqref{B-cep norm inv est}.
Part \ref{hvl-i} of Theorem \ref{hvl theorem} and some matrix-vector arithmetic imply the existence of a function $\Gscr_{\ep} \in \cap_{p=1}^{\infty} L^p$ such that 
\begin{equation}\label{actual convolution}
([\B_{\cep}^{\pm}]^{-1}g)(x)
= e^{q_{\ep}x}\int_{-\infty}^{\infty} e^{-q_{\ep}s}\Gscr_{\ep}(x-s)g(s) \ds,
\end{equation}
where 
\[
\hat{\Gscr_{\ep}}(k) 
= \frac{1}{\tB_{\cep}(-k+iq_{\ep})}.
\]
We have
\[
\norm{[\B_{\cep}^{\pm}]^{-1}g}_{W_{\pm{q}_{\ep}}^{2,\infty}}
= \norm{[\B_{\cep}^{\pm}]^{-1}g}_{L_{\pm{q}_{\ep}}^{\infty}}
+ \norm{\partial_x^2[[\B_{\cep}^{\pm}]^{-1}g]}_{L_{\pm{q}_{\ep}}^{\infty}}.
\]
We calculate
\[
\partial_x^2[[\B_{\cep}^{\pm}]^{-1}g]
= g-\frac{(2+A)[\B_{\cep}^{\pm}]^{-1}g}{c^2},
\]
and so
\[
\norm{\partial_x^2[[\B_{\cep}^{\pm}]^{-1}g]}_{L_{\pm{q}_{\ep}}^{\infty}}
\le \norm{g}_{L_{\pm{q}_{\ep}}^{\infty}} + C\norm{[\B_{\cep}^{\pm}]^{-1}g}_{L_{\pm{q}_{\ep}}^{\infty}}.
\]
Next, we estimate from \eqref{actual convolution} that
\[
\norm{[\B_{\cep}^{\pm}]^{-1}}_{L_{\pm{q}_{\ep}}^{\infty}}
\le \norm{\Gscr_{\ep}}_{L^1}
= \bignorm{\ft^{-1}\left[\frac{1}{\tB_{\cep}(-\cdot+iq_{\ep})}\right]}_{L^1}
\le \frac{C_{\B}}{q_{\ep}}
= \O(\ep^{-1}).
\]
The last inequality comes from \eqref{inv ft L1 O(q) est}.
\end{proof}

\subsubsection{Inversion of $\B_{\cep}-\Sigma_{\cep}^*$ with the Neumann series}\label{B-cep Sigma-cep star inverse Neumann}
Now consider the operator $\B_{\cep}-\Sigma_{\cep}^*$ from $W_{-{q}_{\ep}}^{2,\infty}$ to $L_{-{q}_{\ep}}^{\infty}$.
For simplicity, let
\begin{equation}\label{M defn}
\M
:= 2(2+A)
\quadword{and}
\tM(k) = 2(2+2\cos(k)),
\end{equation}
so
\[
\Sigma_{\cep}^*f 
= \varsigma_{\cep}(x)\M{f}.
\]
By part \ref{FP Linfty-2} of Proposition \ref{hoffman-wayne proposition} we estimate, for $f \in W_{-q_{\ep}}^{2,\infty}$,
\[
\norm{\Sigma_{\cep}^*f}_{L_{-q_{\ep}}^{\infty}}
\le \norm{e^{q_{\ep}\cdot}\varsigma_{\cep}\M{f}}_{L^{\infty}}
\le C\ep^2\norm{e^{q_{\ep}\cdot}\M{f}}_{L^{\infty}}
= C\ep^2\norm{\M{f}}_{L_{-q_{\ep}}^{\infty}} 
\le C\ep^2\norm{f}_{W_{-q_{\ep}}^{2,\infty}} .
\]
That is,
\[
\norm{\Sigma_{\cep}^*}_{\b(W_{-q_{\ep}}^{2,\infty},L_{-q_{\ep}}^{\infty})} 
= \O(\ep^2).
\]
We may therefore invert $\B_{\cep}-\Sigma_{\cep}^*$ from $W_{-q_{\ep}}^{2,\infty}$ to $L_{-q_{\ep}}^{\infty}$ using the Neumann series.
This verifies Hypothesis \ref{hypothesis 3} for $|c| \gtrsim 1$.
Specifically, we set $q_{\L}(\cep) := q_{\ep}$ from \eqref{q-params} and note from \eqref{very hungry inequality} that $q_{\ep} < \min\{q_{\varsigma}(\cep)/2,q_{\H}(\cep),1\}$.

\subsection{Verification of Hypothesis \ref{hypothesis 4}}\label{verification of hypothesis 4}
To compute the phase shift $\vartheta_{\cep}$ from Theorem \ref{abstract phase shift theorem}, we at first follow our methods from Section \ref{mallet-paret approach section} of a particular solution to $(\B_{\cep}-\Sigma_{\cep}^*)f = 0$.
However, now we can use the Neumann series to get an explicit formula for the inverse of this operator from $W_{-q_{\ep}}^{2,\infty}$ to $L_{-q_{\ep}}^{\infty}$, namely,
\begin{multline*}
\big[(\B_{\cep}-\Sigma_{\cep}^*)\big]^{-1}
= \big[\B_{\cep}\big(\ind-[\B_{\cep}^-]^{-1}\Sigma_{\cep}^*\big)\big]^{-1}
=\big(\ind-[\B_{\cep}^-]^{-1}\Sigma_{\cep}^*\big)^{-1}[\B_{\cep}^-]^{-1} 
= \sum_{k=0}^{\infty} \big([\B_{\cep}^-]^{-1}\Sigma_{\cep}^*\big)^{k}[\B_{\cep}^-]^{-1}.
\end{multline*}

As we did in Section \ref{mallet-paret approach section}, make the ansatz $f(x) = e^{i\omega_{\cep}x} + g(x)$, where $g \in W_{-q_{\ep}}^{2,\infty}$.
Then $(\B_{\cep}-\Sigma_{\cep}^*)f = 0$ if and only if
\begin{equation}\label{what h must satisfy}
(\B_{\cep}-\Sigma_{\cep}^*)g = \Sigma_{\cep}^*e^{i\omega_{\cep}\cdot}.
\end{equation}
The Neumann series implies
\begin{equation}\label{g-ep-mu appendix defn}
g_{\ep}
:= \big[(\B_{\cep}-\Sigma_{\cep}^*)\big]^{-1}\Sigma_{\cep}^*e^{i\omega_{\cep}\cdot}
= \sum_{k=0}^{\infty} \big([\B_{\cep}^-]^{-1}\Sigma_{\cep}^*\big)^{k+1}e^{i\omega_{\cep}\cdot}
\end{equation}
and if we set, as before,
\[
f_{\ep}(x) := e^{i\omega_{\cep}x} + g_{\ep}(x),
\]
then
\begin{equation}\label{f-ep-mu decay at +infty}
(\B_{\cep}-\Sigma_{\cep}^*)f_{\ep} = 0
\quadword{and}
\lim_{x \to \infty} |f_{\ep}(x) - e^{i\omega_{\cep}x}|
= 0.
\end{equation}

Now we need asymptotics on $f_{\ep}$ as $x \to -\infty$.
To develop these limits, we exploit a formula that relates $\B_{\cep}^-$ to $\B_{\cep}^+$.
The proof is in Appendix \ref{proof of prop b-residue appendix}.
Roughly speaking, this is motivated by our change of contours that converted the inverse Laplace transform \eqref{pre switcheroo lt} to \eqref{residue switcheroo}.

\begin{proposition}\label{b-residue}
There exists $\ep_{\RES} \in (0,\ep_{\B}]$ such that for $0 < \ep < \ep_{\RES}$ and $h \in L_{-q_{\ep}}^{\infty}$, we have $\varsigma_{\cep}h \in L_{q_{\ep}}^{\infty}$ and 
\begin{equation}\label{residue switcheroo prop}
[\B_{\cep}^-]^{-1}[\varsigma_{\cep}h](x)
= [\B_{\cep}^+]^{-1}[\varsigma_{\cep}h](x) 
+ \alpha_{\ep}[h]e^{i\omega_{\cep}x}
+ \beta_{\ep}[h]e^{-i\omega_{\cep}x},
\end{equation}
where the linear functionals $\alpha_{\ep}$ and $\beta_{\ep}$ are defined as
\begin{equation}\label{alpha defn}
\alpha_{\ep}[h] 
:= -i\left(\frac{\lt_+[\varsigma_{\cep}h](i\omega_{\cep})
+ \lt_-[\varsigma_{\cep}h](i\omega_{\cep})}{(\tB_{\ep})'(\omega_{\cep})}\right)
\end{equation}
and
\begin{equation}\label{beta defn}
\beta_{\ep}[h]
:= i\left(\frac{\lt_+[\varsigma_{\cep}h](-i\omega_{\cep})
+ \lt_-[\varsigma_{\cep}h](-i\omega_{\cep})}{(\tB_{\ep})'(-\omega_{\cep})}\right) .
\end{equation}
The functionals $\alpha_{\ep}$ and $\beta_{\ep}$ satisfy the estimate
\begin{equation}\label{alpha beta functional estimates}
\max\left\{|\alpha_{\ep}[h]|,
|\beta_{\ep}[h]|\right\}
\le C\ep\norm{h}_{L_{-q_{\ep}}^{\infty}}.
\end{equation}
\end{proposition}

Recall that $f_{\ep}$ has the form $f_{\ep}(x) = e^{i\omega_{\cep}{x}} + g_{\ep}(x)$, where, from the definition of $g_{\ep}$ above in \eqref{g-ep-mu appendix defn}, 
\begin{equation}\label{new g-ep}
g_{\ep} 
= \sum_{k=0}^{\infty} \big([\B_{\cep}^-]^{-1}\Sigma_{\cep}^*\big)^{k+1}e^{i\omega_{\cep}\cdot}
= [\B_{\cep}^-]^{-1}[\varsigma_{\cep}\M{h}_{\ep}],
\qquad
h_{\ep}
:= \sum_{k=0}^{\infty} \big([\B_{\cep}^-]^{-1}\Sigma_{\cep}^*\big)^ke^{i\omega_{\cep}\cdot}.
\end{equation}
We use Proposition \ref{b-residue} to rewrite $g_{\ep}$ as
\begin{equation}\label{h after dance of residues}
g_{\ep}(x) 
= [\B_{\cep}^+]^{-1}[\varsigma_{\cep}\M{h}_{\ep}](x)
+ \alpha_{\ep}[\M{h}_{\ep}]e^{i\omega_{\cep}x}
+ \beta_{\ep}[\M{h}_{\ep}]e^{-i\omega_{\cep}x}.
\end{equation}
This is the analogue of \eqref{residue switcheroo}.

Since the image of $[\B_{\cep}^+]^{-1}$ consists of functions that decay to zero at $-\infty$, we have, as in \eqref{limit of f at -infty},
\begin{equation}\label{f-ep-mu decay at -infty}
\lim_{x \to -\infty} \left|f_{\ep}(x) - \left(e^{i\omega_{\cep}x}
+
\alpha_{\ep}\left[\M{h}_{\ep}\right]e^{i\omega_{\cep}x} 
+ \beta_{\ep}\left[\M{h}_{\ep}\right]e^{-i\omega_{\cep}x}\right)\right|
=0. 
\end{equation}
It is possible to obtain very precise estimates on the coefficients on $e^{\pm{i\omega_{\cep}x}}$ in \eqref{f-ep-mu decay at -infty}; the proof is in Appendix \ref{alpha beta prop proof appendix}.

\begin{proposition}\label{alpha beta prop}
There exists $\ep_{\theta} \in (0,\ep_{\B}]$ such that for all $0 < \ep < \ep_{\theta}$, there are numbers $\theta_{\ep,0}^+ = \O(1)$, $\theta_{\ep}^+ = \O(1)$, and $\theta_{\ep}^- = \O(1)$ with
\[
\alpha_{\ep}[\M{g}_{\ep}]
= i\ep\theta_{\ep,0}^+ + \ep^2\theta_{\ep}^+
\quadword{and}
\beta_{\ep}[\M{g}_{\ep}]
= \ep^2\theta_{\ep}^-.
\]
Moreover, $\theta_{\ep,0}^+ \in \R$, and there are constants $C_1$, $C_2 > 0$ such that 
\[
0 
< C_1 
< \theta_{\ep,0}^+ 
< C_2 
< \infty
\]
for all such $\ep$.
\end{proposition}

Set $\breve{f}_{\ep}(x) := \im[f_{\ep}(x)-f_{\ep}(-x)]$ and write
\[
\theta_{\ep}^+
= \theta_{\ep,\r}^++ i\theta_{\ep,\i}^+
\quadword{and}
\theta_{\ep}^-
= \theta_{\ep,\r}^- +i\theta_{\ep,\i}^-
\]
with $\theta_{\ep,\r}^{\pm}$, $\theta_{\ep,\i}^{\pm} \in \R$.
The limits \eqref{f-ep-mu decay at +infty} and \eqref{f-ep-mu decay at -infty} imply
\[
\lim_{x \to \infty}
\left| \breve{f}_{\ep}(x)- \left[(2 + \ep^2(\theta_{\ep,\r}^+-\theta_{\ep,\r}^-))\sin(\omega_{\cep}x) 
- \ep(\theta_{\ep,0}^+ +\ep\theta_{\ep,\i}^++ \ep\theta_{\ep,\i}^-)\cos(\omega_{\cep}x)\right] \right|
= 0.
\]

As in Section \ref{phase shift revealed}, then, the identity \eqref{big trig} implies that a rescaled version of $\breve{f}_{\ep}$, which we call $\gamma_{\ep}$, satisfies both $(\B_{\cep}-\Sigma_{\cep}^*)\gamma_{\ep} = 0$ and 
\[
\lim_{x \to \infty}
|\gamma_{\ep}(x) - \sin(\omega_{\cep}(x + \vartheta_{\cep}))|
=0,
\]
where 
\begin{equation}\label{ultimate phase shift}
\vartheta_{\cep}
:= -\frac{1}{\omega_{\cep}}\arctan\left(\ep\left(\frac{\theta_{\ep,0}^+ +\ep\theta_{\ep,\i}^++ \ep\theta_{\ep,\i}^-)}{2 + \ep^2(\theta_{\ep,\r}^+ - \theta_{\ep,\r}^-)}\right)\right).
\end{equation}
Since $\theta_{\ep,0}^+$ is nonzero and $\O(1)$, it follows from properties of the arctangent that we can write
\[
\vartheta_{\cep}
= \ep\vartheta_{\ep,0},
\]
where, for some constant $C_{\vartheta} > 0$, we have
\[
0 < \frac{1}{C_{\vartheta}} < |\vartheta_{\ep,0}| < C_{\vartheta} < \infty
\]
for all $0 < \ep < \ep_{\theta}$.

At last, we may verify Hypothesis \ref{hypothesis 4} for $|c| \gtrsim 1$.
By part \ref{critical frequency bounds} of Proposition \ref{critical frequency props prop}, we have $A_{\cep} < \omega_{\cep} < B_{\cep} < \pi/2$ for all $|\mu| \le \mu_{\per}(\cep)$.
From part \ref{c0} of that proposition, we deduce the additional bound $1 < A_{\cep}$ for $0 < \ep < \ep_{\B}$.
Then
\[
0 < \frac{\ep\omega_{\cep}}{C_{\vartheta}}
< \ep\omega_{\cep}|\vartheta_{cep,0}|
< C_{\vartheta}\ep\omega_{\cep}
< \pi
\]
for $\ep$ small enough, and so 
\[
\sin(\omega_{\cep}\vartheta_{\cep})
= \sin(\ep\omega_{\cep}\vartheta_{\ep,0})
> 0.
\]

\appendix

\section{Transform Analysis}

\subsection{Fourier analysis}\label{fourier analysis appendix}
If $\fb \in L^1(\R,\C^m)$, we set
\[
\ft[\fb](k)
= \hat{\fb}(k)
:= \frac{1}{\sqrt{2\pi}}\int_{-\infty}^{\infty} \fb(x)e^{-ikx} \dx
\]
and
\[
\ft^{-1}[\fb](x)
= \check{\fb}(x)
:= \frac{1}{\sqrt{2\pi}}\int_{-\infty}^{\infty} \fb(k)e^{ikx} \dx.
\]
If $\fb \in L_{\per}^2(\R,\C^m)$, we define 
\[
\ft[\fb](k)
= \hat{\fb}(k)
= \frac{1}{2\pi}\int_{-\pi}^{\pi} \fb(x)e^{-ik{x}} \dx,
\]
and we have
\[
\fb(x)
= \sum_{k=-\infty}^{\infty} \hat{\fb}(k)e^{ik{x}}.
\]
In either case, with $(S^d\fb)(x) := f(x+d)$, we have the identity
\[
\hat{S^d\fb}(k) = e^{ikd}\hat{\fb}(k).
\]

\subsubsection{Fourier multipliers}\label{fourier multipliers appendix}
We first work on the periodic Sobolev spaces $H_{\per}^r(\R,\C^m)$ from \eqref{H-per-r}.
Take $r$, $s \ge 0$ and suppose that $\tM \colon \R \to \C^{m \times m}$ is measurable with
\[
\sup_{k \in \R} \frac{\norm{\tM(k)}}{(1+k^2)^{(r-s)/2}} 
< \infty.
\]
Then Fourier multiplier $\M$ with symbol $\tM$, defined by
\[
(\M\fb)(x)
:= \sum_{k =-\infty}^{\infty} e^{ikx}\tM(k)\hat{\fb}(k),
\]
i.e., by
\begin{equation}\label{how fm works}
\hat{\M\fb}(k) = \tM(k)\hat{\fb}(k),
\end{equation}
is a bounded operator from $H_{\per}^r(\R,\C^m)$ to $H_{\per}^s(\R,\C^m)$.

We will need some calculus on ``scaled'' Fourier multipliers.
Let $\M$ be the Fourier multiplier with symbol $\tM$.
For $\omega \in \R$, define $\M[\omega]$ to be the Fourier multiplier with symbol $\tM(\omega{k})$, i.e.,
\[
\hat{\M[\omega]f}(k) 
:= \tM(\omega{k})\hat{f}(k).
\]
Last, for a function $\gb \colon \R \to \C^m$, let
\begin{equation}\label{Lip defn}
\Lip(\gb)
:= \sup_{\substack{x, \grave{x} \in \R \\ x \ne \grave{x}}} \left|\frac{\gb(x)-\gb(\grave{x})}{x-\grave{x}}\right|.
\end{equation}
The methods of Appendix D.3 of \cite{faver-dissertation} prove the next lemma.

\begin{lemma}\label{calculus on Fourier multipliers lemma}
Let $\tM \colon \R \to \C^{m \times m}$ be measurable.

\begin{enumerate}[label={\bf(\roman*)},ref={(\roman*)}]

\item\label{Fourier multiplier lipschitz}
Suppose $\Lip(\tM) < \infty$.
Then
\[
\norm{(\M[\omega]-\M[\grave{\omega}])\fb}_{H_{\per}^r(\R,\C^m)}
\le \Lip(\tM)\norm{\fb}_{H_{\per}^{r+1}(\R,\C^m)}|\omega-\grave{\omega}|.
\]

\item\label{Fourier multiplier derivative lipschitz}
Suppose $\tM$ is differentiable with $\Lip(\tM') < \infty$.
Fix $\omega_0 \in \R$ and let $\partial_{\omega}\M[\omega_0]$ be the Fourier multiplier with symbol $k\tM'(\omega_0{k})$, i.e.,
\[
\ft[\partial_{\omega}\M[\omega_0]f](k)
= k\tM'(\omega_0k)\hat{f}(k).
\]
Then
\[
\norm{(\M[\omega_0+\omega]-\M[\omega_0]-\omega\partial_{\omega}\M[\omega_0])\fb}_{H_{\per}^r(\R,\C^m)}
\le \Lip(\tM')\omega^2\norm{\fb}_{H_{\per}^{r+2}(\R,\C^m)}.
\]
\end{enumerate}
\end{lemma}

Next, we discuss some properties of the adjoint of a Fourier multiplier on periodic Sobolev spaces.

\begin{lemma}\label{periodic fm adjoint}
Suppose that $\M \in \b(H_{\per}^r(\R,\C^m),H_{\per}^s(\R,\C^m))$ is a Fourier multiplier with symbol $\tM \colon \R \to \C^{m \times m}$.

\begin{enumerate}[label={\bf(\roman*)},ref={(\roman*)}]

\item\label{adjoint of periodic fourier multiplier}
The adjoint operator $\M^* \in \b(H_{\per}^s(\R,\C^m),H_{\per}^r(\R,\C^m))$, i.e., the operator $\M^*$ satisfying
\[
\ip{\M\fb}{\gb}_{H_{\per}^s(\R,\C^m)} = \ip{\fb}{\M^*\gb}_{H_{\per}^r(\R,\C^m)}, \ \fb \in H_{\per}^r(\R,\C^m), \ \gb \in H_{\per}^s(\R,\C^m),
\]
is given by 
\begin{equation}\label{fm adjoint defn}
\hat{\M^*\gb}(k) = \frac{1}{(1+k^2)^{r-s}}\tM(k)^*\hat{\gb}(k),
\end{equation}
where $\tM(k)^* \in \C^{m \times m}$ is the conjugate transpose of $\tM(k) \in \C^{m \times m}$.

\item
Suppose that $\M \in \b(H_{\per}^r(\R,\C^m),H_{\per}^s(\R,\C^m))$ is a Fourier multiplier with symbol $\tM \colon \R \to \C^{m \times m}$.
Suppose as well that $\M$ has a one-dimensional kernel spanned by $\nub$ with $\hat{\nub}(k) = 0$ for all but finitely many $k$ and that $\tM(k)$ is self-adjoint for all $k \in \R$.
Then the adjoint $\M^*$ from part \ref{adjoint of periodic fourier multiplier} has a one-dimensional kernel in $H_{\per}^r(\R,\C^m)$ spanned by $\nub^*$, whose Fourier coefficients are
\begin{equation}\label{nub star for periodic fm sa kernel}
\hat{\nub^*}(k) := (1+k^2)^{(r-s)/2}\hat{\nub}(k).
\end{equation}
\end{enumerate}
\end{lemma}

Last, we state a slight generalization of a result for Fourier multipliers on the exponentially weighted spaces $H_q^r$ from \eqref{hrq}.
In this case, a Fourier multiplier on $H_q^r$ (or $H^r$) is, of course, defined as before by \eqref{how fm works}.
For $0 < q < \grave{q}$, write
\[
S_{q,\grave{q}}
:= \set{z \in \C}{q \le |\im(z)| \le \grave{q}}.
\]

\begin{lemma}[Beale]\label{beale fm}
Let $0 < q_0 \le q_1 < q_2$ and suppose that $\tM \colon \R \to \C$ is a measurable function with the following properties.

\begin{enumerate}[label={\bf($\M$\arabic*)}]

\item
The function $\tM$ is analytic on the strips $S_{0,q_1}$ and $S_{q_1,q_2}$.

\item\label{beale simple zeros}
The function $\tM$ has finitely many zeros in $\R$, all of which are simple.  
Denote the collection of these zeros by $\Pfrak_{\M}$.

\item
There exist $C$, $z_0 > 0$ and $s \ge 0$ such that if $z \in S_{0,q_1} \cup S_{q_1,q_2}$ with $|z| \ge z_0$, then
\begin{equation}\label{beale est}
C|\re(z)|^s
\le |\tM(z)|.
\end{equation}
\end{enumerate}

Now let $\M$ be the Fourier multiplier with symbol $\tM$.
There exist $q_{\star}$, $q_{\star\star} > 0$ with $q_1 \le q_{\star} < q_{\star\star} \le q_2$ such that if $q \in [q_{\star},q_{\star\star}]$, then, for any $r \ge 0$, $\M$ is invertible from $H_q^{r+s}$ to the subspace
\[
\Dfrak_{\M,q}^r
:= \set{f \in H_q^r}{z \in \Pfrak_{\M} \Longrightarrow \hat{f}(z) = 0}
\]
and, for $f \in \Dfrak_{\M,q}^r$,
\[
\norm{\M^{-1}f}_{r+s,q}
\le \left(\sup_{k \in \R} \frac{(1+k^2)^{s/2}}{|\tM(k\pm{iq})|}\right)\norm{f}_{r,s}.
\]

There are two additional special cases.

\begin{enumerate}[label={\bf(\roman*)}, ref={(\roman*)}]

\item\label{beale special case}
If $q_1 = q_2$, then the result above is true for all $0 < q \le q_2$.

\item
If $\Pfrak_{\M} = \{0\}$ and 0 is a double zero of $\tM$, then the result above is still true if we replace $H_q^r$ and $H_q^{r+s}$ by $E_q^r$ and $E_q^{r+s}$ throughout.
\end{enumerate}
\end{lemma}

This lemma was proved by Beale as Lemma 3 in \cite{beale1} in the particular case \ref{beale special case}.

\subsection{The Laplace transform}\label{laplace transform appendix}
Suppose $\fb \in L_{\loc}^1(\R,\C^m)$ with $e^{-a\cdot}\fb \in L^{\infty}(\R,\C^m)$.
We set
\[
\lt_+[\fb](z) 
:= \int_0^{\infty} \fb(s)e^{-sz} \ds, \ \re(z) > a
\]
and
\[
\lt_-[\fb](z) := \int_0^{\infty} \fb(-s)e^{sz} \ds, \ \re(z) < a.
\]
These Laplace transforms have the following useful properties.
First, $\lt_+[\fb]$ is analytic on $\re(z) > a$ and $\lt_-[\fb]$ is analytic on $\re(z) < a$.
In particular, if $e^{-a|\cdot|}\fb \in L^{\infty}(\R,\C^m)$ with $a > 0$, then $\lt_+[\fb]$ and $\lt_-[\gb]$ are defined and analytic on the common strip $|\re(z)| < a$.
Next, we have the inverse formulas
\begin{equation}\label{lt+inv}
\frac{\fb(x^+) + \fb(x^-)}{2} = \frac{1}{2\pi{i}}\int_{\re(z) = b} \lt_+[\fb](z)e^{xz} \dz, \ x > 0, \ b > a
\end{equation}
and
\begin{equation}\label{lt-inv}
\frac{\fb(x^+) + \fb(x^-)}{2} = \frac{1}{2\pi{i}}\int_{\re(z) = b} \lt_-[\fb](-z)e^{-xz} \dz, \ x < 0, \ b > -a,
\end{equation}
where 
\[
\fb(x^{\pm})
:= \lim_{t \to x^{\pm}} \fb(t).
\]
If $\fb$ is differentiable, then
\begin{equation}\label{lt-deriv id}
\lt_{\pm}[\fb'](z)
= \mp\fb(0) + z\lt_{\pm}[\fb](z).
\end{equation}
Last, the Laplace transform interacts with shift operators via the identities
\begin{equation}\label{lt shift id+}
\lt_+[S^d\fb](z) 
= e^{zd}\lt_+[\fb](z) -e^{zd}\int_0^d f(s)e^{-sz} \ds
\end{equation}
and
\begin{equation}\label{lt shift id-}
\lt_-[S^d\fb](z) 
= e^{zd}\lt_-[\fb](z) + e^{zd}\int_{-d}^0 \fb(-s)e^{sz} \ds.
\end{equation}
\section{Derivation of the Traveling Wave Problem \eqref{cov simplified}}\label{tw derivation}
Under the traveling wave ansatz \eqref{tw ansatz}, the original equations of motion \eqref{original equations of motion} for the diatomic FPUT lattice convert to the system
\begin{equation}\label{tw problem orig}
\begin{cases}
c^2p_1'' = -(1+w)(p_1+p_1^2) + (wS^1+S^{-1})(p_2+p_2)^2 \\
\\
c^2p_2'' = (wS^{-1}+S^1)(p_1+p_1^2) - (1+w)(p_2+p_2^2).
\end{cases}
\end{equation}
Here we define $w := 1/m$.
With $\pb = (p_1,p_2)$, \eqref{tw problem orig} compresses into
\begin{equation}\label{tw prob orig compr}
c^2\pb'' + L_w\pb
+
L_w\pb^{.2}
=0,
\qquad
L_w := \begin{bmatrix*}
(1+w) &-(wS^1+S^{-1}) \\
-(wS^{-1}+S^1) &(1+w)
\end{bmatrix*}.
\end{equation}
For $\fb = (f_1,f_2)$ and $\gb = (g_1,g_2)$, we use the notation 
\[
\fb^{.2} := \begin{pmatrix*} 
f_1^2 \\
f_2^2
\end{pmatrix*}
\quadword{and}
\fb.\gb := \begin{pmatrix*}
f_1g_1 \\
g_2g_2
\end{pmatrix*}.
\]
We note that \eqref{tw prob orig compr} is the same system that was derived in equation (2.4) in \cite{faver-wright}.

Set
\[
J 
:= \begin{bmatrix*}[r]
1 &1 \\
1 &-1
\end{bmatrix*},
\]
so $J^{-1} = J/2$.  
If we substitute $\pb = J\rhob$ with $\rhob = (\rho_1,\rho_2)$, then \eqref{tw prob orig compr} becomes
\[
c^2J\rhob'' + L_wJ\rhob + L_w(J\rhob)^{.2}
= 0.
\]
Multiplying through by $J^{-1}$, this is equivalent to
\begin{equation}\label{post cov}
c^2\rhob'' + \frac{1}{2}JL_wJ\rhob + \frac{1}{2}JL_w(J\rhob)^{.2} 
= 0.
\end{equation}
Setting $w=1+\mu$ for $|\mu| < 1$ and 
\[
\D_{\mu}
= \frac{1}{2}JL_{1+\mu}J
\quadword{and}
\nl(\rhob,\grave{\rhob})
:= (J\rhob).(J\grave{\rhob})
\]
we see that \eqref{post cov} is equivalent to \eqref{cov simplified}.
One easily checks that for functions $\rhob$, $\grave{\rhob}$, and $\breve{\rhob}$ and scalars $a$, we have
\begin{equation}\label{nl symm bil}
\nl(a\rhob,\grave{\rhob})
= a\nl(\rhob,\grave{\rhob}),
\qquad
\nl(\rhob,\grave{\rhob})
= \nl(\grave{\rhob},\rhob)
\quadword{and}
\nl(\rhob+\grave{\rhob},\breve{\rhob})
= \nl(\rhob,\breve{\rhob}) + \nl(\grave{\rhob},\breve{\rhob}),
\end{equation}
and so $\nl$ is indeed symmetric and bilinear.

Last, we discuss the even-odd symmetries of $\G_c$.
First, observe that if $f$ is even, then $Af$ is even and $\delta{f}$ is odd, while if $f$ is odd, then $Af$ is odd and $\delta{f}$ is even.
So, if $\rho_1$ is even and $\rho_2$ is odd and $\rhob = (\rho_1,\rho_2)$, then $(\D_{\mu}\rhob) \cdot \e_1$ is even and $(\D_{\mu}\rhob)\cdot\e_2$ is odd, where $\e_1 = (1,0)$ and $\e_2 = (0,1)$.
Likewise, if $\rho_1$ and $\grave{\rho}_1$ are even and $\rho_2$ and $\grave{\rho}_2$ are odd, then $\nl(\rhob,\grave{\rhob})\cdot\e_1$ is even and $\nl(\rhob,\grave{\rhob})\cdot\e_2$ is odd.
We conclude that $\G_c(\rhob,\mu)\cdot\e_1$ is even and $\G_c(\rhob,\mu)\cdot\e_2$ is odd when $\rho_1$ is even and $\rho_2$ is odd.

Next, if $f$ is integrable on $\R$ or $2P$-periodic and integrable on $[-P,P]$, then $(2-A)f$ and $\delta{f}$ are ``mean-zero'' in the sense that
\[
\int_{-\infty}^{\infty} \big((2-A)f\big)(x) \dx
= \int_{-P}^P \big((2-A)f\big)(x) \dx
= \int_{-\infty}^{\infty} (\delta{f})(x) \dx
= \int_{-P}^P (\delta{f})(x) \dx
= 0.
\]
Thanks to the structure of $\D_{\mu}$, all terms in $\G_c(\rhob,\mu) \cdot \e_1$ contain either a factor of $2-A$ or $\delta$, and so $\G_c(\rhob,\mu)\cdot\e_1$ is always mean-zero.
We conclude that the symmetries in \eqref{all the pretty symmetries} hold.
\section{Existence of Periodic Solutions}\label{periodic solutions appendix}

\subsection{Linear analysis}

We begin with two propositions that study the linearization $\Gamma_c^{\mu}[\omega]$ defined in \eqref{Gamma-c-mu defn}.
The first of these contains all the technical details needed to prove Proposition \ref{critical frequency props prop}.

\begin{lemma}\label{critical frequency technical lemma}
For each $|c| > 1$, there is $\Mu(c) > 0$ with the following properties.

\begin{enumerate}[label={\bf(\roman*)},ref={(\roman*)}]

\item\label{basic lambda props}
The functions $\lambda_{\mu}^{\pm}$ defined in \eqref{lambda eigs defn} are even, real-valued, and bounded on $\R$ with
\begin{equation}\label{lambda-mu1 est}
0 
\le \lambda_{\mu}^-(K) 
\le \begin{cases}
2(1+\mu), &\mu \in (-1,0] \\
2, &\mu \in [0,1)
\end{cases}
\end{equation}
and 
\begin{equation}\label{lambda-mu2 est}
\begin{cases}
2, &\mu \in (-1,0] \\
2(1+\mu), &\mu \in [0,1)
\end{cases}
\le \lambda_{\mu}^+(K)
\le 2+2(1+\mu)
\end{equation}
for all $K \in \R$.
For $\mu \ne 0$, $\lambda_{\mu}^{\pm}$ are differentiable on $\R$, while $\lambda_0^{\pm}$ are continuous on $\R$ and differentiable except at the points $K = (2j+1)\pi/2$, $j \in \Z$.

\item\label{Lambda str incr}
For $|\mu| < \Mu(c)$, the functions 
\[
K \mapsto \Lambda_c^{\pm}(K,\mu) 
:= c^2K^2-\lambda_{\mu}^{\pm}(K)
\]
are strictly increasing on $(0,\infty)$.

\item\label{Lambda- soln}
$\Lambda_c^-(K,\mu) = 0$ if and only if $K = 0$.

\item\label{existence of omega-c-mu!}
For all $|\mu| \le \Mu(c)$, there is a unique $\omega_c^{\mu} > 0$ such that $\Lambda_c^+(\omega_c^{\mu},\mu) = 0$.
Moreover, there are numbers $0 < A_c < B_c < \pi/2$ such that $A_c < \omega_c^{\mu} < B_c$ for each such $\mu$.
This root $\omega_c^{\mu}$ also satisfies the cruder bounds
\begin{equation}\label{omega-ep-mu cruder}
\sqrt{\frac{2}{c^2(1+\mu)}} 
\le \omega_c^{\mu} 
\le \sqrt{\frac{2(2+\mu)}{c^2(1+\mu)}}.
\end{equation}

\item
$\omega_c^{\mu} - \omega_c = \O_c(\mu)$.
\end{enumerate}
\end{lemma}

\begin{proof}

\begin{enumerate}[label={\bf(\roman*)}]

\item
We recall that the eigenvalue curves $\lambda_{\mu}^{\pm}$ were defined in \eqref{lambda eigs defn}.
Restricting $|\mu| \le 1$, we have
\begin{equation}\label{est on sqr root}
|\mu|
\le \sqrt{\mu^2+4(1+\mu)\cos^2(K)}
\le \sqrt{\mu^2+4(1+\mu)}
= \sqrt{(2+\mu)^2}
= 2+\mu.
\end{equation}
Then
\[
0 \le \lambda_{\mu}^-(K) \le 2+\mu-|\mu| = \begin{cases}
2(1+\mu), &\mu \in (-1,0) \\
2, &\mu \in [0,1).
\end{cases}
\]
This proves \eqref{lambda-mu1 est}.

Next, we use \eqref{est on sqr root} to bound
\[
2 + \mu + |\mu| \le \lambda_{\mu}^+(K) \le 2(2+\mu),
\]
where
\[
2+\mu+|\mu| = \begin{cases}
2, &\mu \in (-1,0) \\
2(1+\mu), &\mu \in [0,1).
\end{cases}
\]
This proves \eqref{lambda-mu2 est}.

\item
We prove this separately for the cases $\mu \in (-1,0)$, $\mu \in (0,1)$, and $\mu = 0$.  First, when $\mu \ne 0$, the derivatives $(\lambda_{\mu}^{\pm})'$ satisfy
\begin{equation}\label{basic lambda deriv id}
|(\lambda_{\mu}^{\pm})'(K)|
= 4(1+\mu)\left|\frac{\sin(K)\cos(K)}{\sqrt{\mu^2+4(1+\mu)\cos^2(K)}}\right|.
\end{equation}

\begin{enumerate}[label=,labelsep=0pt]

\item
{\it{The case $\mu \in (-1,0]$.}}
We rewrite \eqref{basic lambda deriv id} as
\[
|(\lambda_{\mu}^{\pm})'(K)| 
= 2\left(\frac{2(1+\mu)}{2+\mu}\right)|\sin(K)|\sqrt{\frac{1-\sin^2(K)}{1-\dfrac{4(1+\mu)}{(2+\mu)^2}\sin^2(K)}}.
\]
We can check that 
\[
0 < \frac{2(1+\mu)}{2+\mu} \le 1
\quadword{and}
0 < \frac{4(1+\mu)}{(2+\mu)^2} \le 1
\]
for $\mu \in (-1,0]$ and we also have the estimate
\[
\sup_{0 \le s \le 1} \frac{1-s}{1-rs} \le 1
\]
when $0 < r < 1$.
Combining these estimates, we find
\begin{equation}\label{lambda deriv est for neg mu}
|(\lambda_{\mu}^{\pm})'(K)| \le 2|\sin(K)| \le 2|K|.
\end{equation}
Taking $K > 0$, we have
\[
\partial_K[\Lambda_c^{\pm}(K,\mu)]
= 2c^2K-(\lambda_{\mu}^{\pm})'(K)
\ge 2c^2K - 2K
= 2(c^2-1)K
> 0.
\]

\item
{\it{The case $\mu \in (0,1)$.}}
We claim
\begin{equation}\label{aux claim}
\sup_{0 \le s \le 1} \frac{s}{\sqrt{\mu^2+4(1+\mu)s^2}} \le \frac{1}{2}
\end{equation}
when $\mu \in (0,1)$.
This inequality is equivalent to 
\[
0 \le \mu^2 + 4\mu{s}^2
\]
for all $s \in [0,1]$, and this clearly holds for $\mu \in (0,1)$.
Hence \eqref{aux claim} is true, and we use that estimate on \eqref{basic lambda deriv id} to find
\begin{equation}\label{lambda deriv est for pos mu}
|(\lambda_{\mu}^{\pm})'(K)| 
= 4(1+\mu)|\sin(K)|\frac{|\cos(K)|}{\sqrt{\mu^2+4(1+\mu)\cos^2(K)}}
\le 4(1+\mu)|K|\frac{1}{2}
= 2(1+\mu)|K|.
\end{equation}
Now we estimate
\[
\partial_K[\Lambda_c^{\pm}(K,\mu)]
= 2c^2K-(\lambda_{\mu}^{\pm})'(K)
> 2c^2K-2(1+\mu)K
= 2(c^2-(1+\mu))K
> 0.
\]
This last inequality holds if $c^2-(1+\mu) > 0$, and we can ensure this by taking $|\mu| < (c^2-1)/4$.
So, we take our threshold for $\mu$ to be
\begin{equation}\label{Mu-1-c defn}
\Mu(c) 
:= \min\left\{1,\frac{c^2-1}{4}\right\}
\end{equation}
and assume in the following that $|\mu| \le \Mu(c)$.
Note that $\Mu(c) \to 0^+$ as $|c| \to 1^+$, and so the $\mu$-interval over which we work necessarily shrinks as $|c| \to 1^+$.

\item
{\it{The case $\mu=0$.}}
We have
\[
\Lambda_c^{\pm}(K,0)
= c^2K^2 - 2 \mp 2|\cos(K)|,
\]
and so the maps $\Lambda_c^{\pm}(\cdot,0)$ are differentiable on all intervals of the form
\begin{equation}\label{special Lambdac interval}
\left(\frac{(2j+1)\pi}{2}, \frac{(2j+3)\pi}{2}\right),
\end{equation}
with derivative equal to
\[
2c^2K\pm2\sin(K),
\]
with the $\pm$ sign depending on the particular interval.
Since
\[
2c^2K\pm2\sin(K) > 0
\]
for all $K > 0$, and so we conclude that the maps $\Lambda_c^{\pm}(\cdot,0)$ are strictly increasing on all intervals of the form \eqref{special Lambdac interval}.
Since the maps $\Lambda_c^{\pm}(\cdot,0)$ are continuous, it follows that they are strictly increasing on all of $(0,\infty)$ and therefore have at most one real root.
\end{enumerate}

\item
This follows from part \ref{Lambda str incr}, the calculation $\Lambda_c^-(0,\mu) = 0$, and the evenness of $\Lambda_c^-(\cdot,\mu)$.

\item
Using the bounds \eqref{lambda-mu2 est} on $\lambda_{\mu}^+$, it is straightforward to show that $\Lambda_c^+(K,\mu) < 0$ for $K < \sqrt{2/c^2(1+\mu)}$ and $\Lambda_c^+(K,\mu) > 0$ for $K > \sqrt{2(2+\mu)/c^2(1+\mu)}$.
Consequently, there exists $\omega_c^{\mu} \in [\sqrt{2/c^2(1+\mu)},\sqrt{2(2+\mu)/c^2(1+\mu)}]$ such that $\Lambda_{\mu}^+(\omega_c^{\mu},\mu) = 0$.
This $\omega_c^{\mu}$ is necessarily unique because $\Lambda_c^+(\cdot,\mu)$ is strictly increasing on $(0,\infty)$.

We now want sharper, uniform bounds of the form $0 < A_c < \omega_c^{\mu} < B_c < \pi/2$, which the estimates on $\omega_c^{\mu}$ in the preceding paragraph do not necessarily give.
Let
\begin{equation}\label{A-c defn}
A_c 
:= \frac{\sqrt{2}}{|c|}.
\end{equation}
Since $|c| > 1$, we have $A_c < \sqrt{2} < \pi/2$.
The bound \eqref{lambda-mu2 est} implies $-\lambda_{\mu}^+(K) \le -2$ for all $K$, and so if $0 < K \le A_c$, then
\[
\Lambda_c^+(K,\mu) < 2-\lambda_{\mu}^+(K) \le 0.
\]
In particular, $\Lambda_c^+(A_c,\mu) < 0$.

We use \eqref{lambda-mu2 est} again to estimate
\begin{equation}\label{tricky biz-ness}
\Lambda_c^+\left(\frac{\pi}{2},\mu\right)
= \frac{c^2\pi^2}{4}-(2+\mu+|\mu|)
\ge 2(c^2-(1+|\mu|))
\ge 2(c^2-(1+\Mu(c))
> 0
\end{equation}
by the definition of $\Mu(c)$ in \eqref{Mu-1-c defn}.

All together, we have the chain of inequalities
\[
\Lambda_c^+(A_c,\mu) 
\le 0
< c^2-(1+\Mu(c))
< \Lambda_c^+\left(\frac{\pi}{2},\mu\right),
\]
and so the intermediate value theorem produces $B_c^{\mu} \in (A_c,\pi/2)$ such that 
\begin{equation}\label{defining prop of B-c-mu}
\Lambda_c^+(B_c^{\mu},\mu) = c^2-(1+\Mu(c)) > 0.
\end{equation}
A second application of the intermediate value theorem then yields $\omega_c^{\mu} \in (A_c,B_c^{\mu})$ such that $\Lambda_c^+(\omega_c^{\mu},\mu) = 0$.
This is, of course, the same $\omega_c^{\mu}$ that we initially found using the cruder bounds above.

Now let
\[
B_c 
:= \sup_{|\mu| \le \Mu(c)} B_c^{\mu}.
\]
We want to show $A_c < B_c < \pi/2$.
The first inequality is obvious since $A_c < B_c^{\mu}$. 
If $B_c = \pi/2$, then we may take a sequence $(\mu_n(c))$ in $[-\Mu(c),\Mu(c)]$ such that 
\[
\lim_{n \to \infty} B_c^{\mu_n(c)} = \frac{\pi}{2}
\quadword{and} 
\lim_{n \to \infty} \mu_n(c) = \overline{\mu}(c)
\]
for some $\overline{\mu}(c) \in [-\Mu(c),\Mu(c)]$.
In that case, \eqref{defining prop of B-c-mu} and the continuity of $\Lambda_c^+$ on $\R \times [-\Mu(c),\Mu(c)]$ imply
\[
 c^2-(1+\Mu(c))
= \lim_{n \to \infty} \Lambda_c^+(B_c^{\mu_n(c)},\mu_n(c))
= \Lambda_c^+\left(\frac{\pi}{2},\overline{\mu}(c)\right)
\ge  2(c^2-(1+\Mu(c)))
\]
by \eqref{tricky biz-ness}.
This is, of course, a contradiction.

\item
Now we show that $\omega_c^{\mu}-\omega_c = \O_c(\mu)$.
If $\omega_c = \omega_c^{\mu}$, then there is nothing to prove, so assume $\omega_c\ne\omega_c^{\mu}$.
Recall that $\omega_c$ satisfies
\[
c^2\omega_c^2-(2+2\cos(\omega_c)) 
= 0
\]
and $\omega_c^{\mu}$ satisfies
\[
c^2(\omega_c^{\mu})^2-\lambda_{\mu}^+(\omega_c^{\mu})
= 0.
\]
Subtracting these two equalities and using the definition of $\lambda_{\mu}^+$ in \eqref{lambda eigs defn}, we obtain
\[
c^2(\omega_c-\omega_c^{\mu}) + \mu - 2\cos(\omega_c)+\sqrt{\mu^2+4(1+\mu)\cos^2(\omega_c^{\mu})}
= 0.
\]
Taylor-expanding the square root and using the uniform bound $0 < A_c < \omega_c^{\mu} < B_c$ from part \ref{existence of omega-c-mu!}, this rearranges to
\begin{equation}\label{omega-c diff of squares}
(\omega_c-\omega_c^{\mu})\left(c^2(\omega_c+\omega_c^{\mu}) -2\left(\frac{\cos(\omega_c)-\cos(\omega_c^{\mu})}{\omega_c-\omega_c^{\mu}}\right)\right)
= \O_c(\mu).
\end{equation}
Since the cosine has Lipschitz constant equal to 1, we have
\[
\sup_{|\mu| \le \Mu(c)} 2\left|\frac{\cos(\omega_c)-\cos(\omega_c^{\mu})}{\omega_c-\omega_c^{\mu}}\right|
\le 2.
\]
From part \ref{existence of omega-c-mu!}, specifically \eqref{A-c defn}, we have $\sqrt{2}/|c| \le \omega_c^{\mu}$.
Likewise, since $\tB_c(\omega_c) = 0$ with $\tB_c$ defined in \eqref{tB-c defn}, one can extract the inequaliy $\sqrt{2}/|c| < \omega_c$.
Specifically, the proof is in Appendix \ref{proof of B props appendix} in the context of the proof of part \ref{B simple zeros} of Proposition \ref{symbol of B}.
Thus
\[
c^2(\omega_c+\omega_c^{\mu})
\ge 2\sqrt{2}|c|
> 2
\]
whenever $|c| > 1$, and so we have
\[
\inf_{|\mu| \le \Mu(c)} \left|c^2(\omega_c+\omega_c^{\mu}) -2\left(\frac{\cos(\omega_c)-\cos(\omega_c^{\mu})}{\omega_c-\omega_c^{\mu}}\right)\right|
> 0.
\]
We conclude from \eqref{omega-c diff of squares} that $\omega_c-\omega_c^{\mu} = \O_c(\mu)$.
\qedhere
\end{enumerate}
\end{proof}

The next lemma contains the remaining details that we need for the quantitative bifurcation argument that we will carry out in the following sections.
To phrase this lemma, we need the definitions of the periodic Sobolev spaces $H_{\per}^r$, $E_{\per}^r$, and $O_{\per}^r$ from the start of Section \ref{periodic solutions section}.
In this appendix only, we abbreviate
\[
\norm{\fb}_r
:= \norm{\fb}_{H_{\per}^r \times H_{\per}^r}.
\]
We also need to recall the definition of the operator $\Gamma_c^{\mu}[\gamma_c^{\mu}]$ from \eqref{Gamma-c-mu defn}.

\begin{lemma}\label{lambda props}
For each $|c| > 1$, there is $\mu_{\per}(c) \in (0,\Mu(c)]$ such that the following hold.

\begin{enumerate}[label={\bf(\roman*)},ref={(\roman*)}]

\item\label{upsilon-c-mu part}
There exists $\upsilon_c^{\mu} \in \R$ such that if 
\begin{equation}\label{nub defn}
\nub_c^{\mu}
:= \begin{pmatrix*}
\upsilon_c^{\mu}\cos(\cdot) \\
\sin(\cdot)
\end{pmatrix*},
\end{equation}
then the kernel of $\Gamma_c^{\mu}[\omega_c^{\mu}]$ in $E_{\per,0}^r \times O_{\per}^r$ is spanned by $\nub_c^{\mu}$ for all $r \ge 2$.
Moreover, there is a constant $C_{\upsilon}(c) > 0$ such that $|\upsilon_c^{\mu}| \le C_{\upsilon}(c)\mu$.

\item\label{range=kernel perp}
Given $\fb \in E_{\per,0}^{r+2} \times O_{\per}^{r+2}$ and $\gb \in E_{\per,0}^r \times O_{\per}^r$, we have $\Gamma_c^{\mu}[\omega_c^{\mu}]\fb = \gb$ if and only if $\ip{\gb}{\nub_c^{\mu}}_0 = 0$.

\item\label{sort of inverse bound}
For each $r \ge 0$, there is a constant $C(c,r) > 0$ such that if $\fb \in E_{\per,0}^{r+2} \times \O_{\per}^{r+2}$ and $\gb \in E_{\per,0}^r \times O_{\per}^r$ with 
\[
\Gamma_c^{\mu}[\omega_c^{\mu}]\fb = \gb
\quadword{and}
\ip{\fb}{\nub_c^{\mu}}_0 = \ip{\gb}{\nub_c^{\mu}}_0 = 0,
\]
then $\norm{\fb}_{r+2} \le C(c,r)\norm{\gb}_r$.
In particular, given $\gb \in E_{\per,0}^r \times O_{\per}^r$ satisfying $\ip{\gb}{\nub_c^{\mu}}_0 = 0$, there is a unique $\fb \in E_{\per,0}^{r+2}\times O_{\per}^{r+2}$ with $\Gamma_c^{\mu}[\omega_c^{\mu}]\fb = \gb$ and $\ip{\fb}{\nub_c^{\mu}}_0 = 0$, and, for this $\fb$, we write $\fb = \Gamma_c^{\mu}[\omega_c^{\mu}]^{-1}\gb$.

\item\label{ip bound}
There is a constant $C(c) > 0$ such that $|\ip{\partial_{\omega}\Gamma_c^{\mu}[\omega_c^{\mu}]\nub_c^{\mu}}{\nub_c^{\mu}}_0| \ge C(c)$.
\end{enumerate}
\end{lemma}

\begin{proof}

\begin{enumerate}[label={\bf(\roman*)}]

\item
Suppose $\Gamma_c^{\mu}[\omega_c^{\mu}]\phib = 0$, for some nonzero $\phib = (\phi_1,\phi_2) \in E_{\per,0}^r \times O_{\per}^r$.
Let $k \in \Z$ such that $\hat{\phib}(k) \ne 0$.
Then $\hat{\phib}(k)$ is an eigenvector of $\tD_{\mu}[\omega_c^{\mu}k]$, and so it follows that 
\begin{equation}\label{two eig choices}
c^2(\omega_c^{\mu}k)^2-\lambda_{\mu}^-(\omega_c^{\mu}k) = 0
\quadword{or}
c^2(\omega_c^{\mu}k)^2-\lambda_{\mu}^+(\omega_c^{\mu}k) = 0.
\end{equation}
In the first case, part \ref{Lambda- soln} of Lemma \ref{critical frequency technical lemma} tells us that $\omega_c^{\mu}k = 0$, and consequently $k = 0$.
But we know $\hat{\phi}_1(0) = 0$ since $\phi_1 \in E_{\per,0}^r$, and we have $\hat{\phi}_2(0) = 0$ since $\phi_2$ is odd.
Then $\hat{\phib}(0) = 0$, a contradiction.

It therefore must be the case that the second equality in \eqref{two eig choices} holds, and so $\omega_c^{\mu}k = \pm\omega_c^{\mu}$ by part \ref{existence of omega-c-mu!} of Lemma \ref{critical frequency technical lemma}.
Hence $k = \pm 1$ and $\hat{\phib}(k) = 0$ for $|k| \ne 1$.
So, with $e_k(x) := e^{ikx}$, we may write 
\begin{multline*}
\phib 
= \hat{\phib}(-1)e_{-1} + \hat{\phib}(1)e_1 
= \begin{pmatrix*}
\hat{\phi}_1(-1)e_{-1} + \hat{\phi}_1(1)e_1 \\
\hat{\phi}_2(-1)e_{-1}+ \hat{\phi}_2(1) e_1 
\end{pmatrix*} 
= \begin{pmatrix*}
\hat{\phi}_1(1)(e_{-1}+e_1) \\
\hat{\phi}_2(1)(e_1-e_{-1}) 
\end{pmatrix*} 
= \begin{pmatrix*}
2\hat{\phi}_1(1)\cos(\cdot) \\
2i\hat{\phi}_2(1)\sin(\cdot)
\end{pmatrix*}.     
\end{multline*}
Here we have used the assumption that $\phi_1$ is even and $\phi_2$ is odd.

We know that $\hat{\phib}(1)$ is an eigenvector of $\tD_{\mu}[\omega_c^{\mu}]$ corresponding to the eigenvalue $\lambda_{\mu}^+(\omega_c^{\mu})$.
Using the definition of this matrix in \eqref{tDmu defn}, we conclude there is $a \in \C$ such that
\begin{equation}\label{upsilonw defn}
\hat{\phib}(1) = a\begin{pmatrix*}
i\upsilon_c^{\mu} \\
1
\end{pmatrix*},
\qquad
\upsilon_c^{\mu} := \frac{\mu\sin(\omega_c^{\mu})}{\lambda_{\mu}^+(\omega_c^{\mu})-(2+\mu)(1-\cos(\omega_c^{\mu}))},
\end{equation}
provided that the term in the denominator of $\upsilon_c^{\mu}$ is nonzero.
It is, in fact, bounded below away from zero.
If we assume 
\begin{equation}\label{Mu-2-c}
|\mu| 
\le \min\left\{\frac{1}{2},\Mu(c)\right\}
=:\Mu_1(c),
\end{equation}
then
\[
\lambda_{\mu}^+(\omega_c^{\mu})-(2+\mu)(1-\cos(\omega_c^{\mu}))
= \sqrt{\mu^2+4(1+\mu)\cos^2(\omega_c^{\mu})}+(2+\mu)\cos(\omega_c^{\mu}) 
\ge \sqrt{2}\cos(B_c)
\]
This gives the desired estimate $|\upsilon_c^{\mu}| \le C_{\upsilon}(c)|\mu|$.

We conclude that if $\Gamma_c^{\mu}[\omega_c^{\mu}]\phib = 0$, then
\[
\phib = \begin{pmatrix*}
2\hat{\phi}_1(1)\cos(\cdot) \\
2i\hat{\phi}_2(1)\sin(\cdot)
\end{pmatrix*}
\quadword{and}
\hat{\phib}(1) = a\begin{pmatrix*} i\upsilon_c^{\mu} \\ 1 \end{pmatrix*},
\]
thus
\[
\phib = 2ai\begin{pmatrix*}
\upsilon_c^{\mu}\cos(\cdot) \\
\sin(\cdot)
\end{pmatrix*},
\]
and so the kernel of $\Gamma_c^{\mu}[\omega_c^{\mu}]$ is spanned by the vector $\nub_c^{\mu}$.

\item
Recall that the symbol of $\Gamma_c^{\mu}[\omega_c^{\mu}]$ is 
\[
-c^2(\omega_c^{\mu}k)^2+\tD_{\mu}[\omega_c^{\mu}k],
\]
and this is a self-adjoint matrix in $\mathbb{C}^{2 \times 2}$ by the definition of $\tD_{\mu}$ in \eqref{tDmu defn}.
Moreover, by part \ref{upsilon-c-mu part}, the kernel of $\Gamma_c^{\mu}[\omega_c^{\mu}]$ is one-dimensional and spanned by $\nub_c^{\mu}$ from \eqref{nub defn}, where $\hat{\nub_c^{\mu}}(k) = 0$ for $k \ne \pm1$.
If we denote by $\Gamma_c^{\mu}[\omega_c^{\mu}]^*$ the adjoint of $\Gamma_c^{\mu}[\omega_c^{\mu}]$ from $E_{\per,0}^2 \times O_{\per}^2$ to $E_{\per,0}^0 \times O_{\per}^0$, Lemma \ref{periodic fm adjoint} tells us that the kernel of $\Gamma_c^{\mu}[\omega_c^{\mu}]^*$ is spanned by the function $(\nub_c^{\mu})^*$ whose Fourier coefficients are
\[
\hat{(\nub_c^{\mu})^*}(k) = (1+k^2)^{(2-0)/2}\hat{\nub_c^{\mu}}(k)
= \begin{cases}
2\hat{\nub_c^{\mu}}(k), &|k| = 1 \\
0, &|k| \ne 1.
\end{cases}
\]
That is, the kernel of $\Gamma_c^{\mu}[\omega_c^{\mu}]^*$ is just spanned by $\nub_c^{\mu}$.

Now we show
\begin{equation}\label{periodic solvability}
\Gamma_c^{\mu}[\omega_c^{\mu}]\fb = \gb, \ \fb \in E_{\per,0}^{r+2} \times O_{\per}^{r+2}, \ \gb \in E_{\per,0}^r \times O_{\per}^r
\iff \ip{\gb}{\nub_c^{\mu}}_0 = 0.
\end{equation}
The forward implication is clear from the containment $E_{\per,0}^r \times O_{\per}^r \subseteq E_{\per,0}^0 \times O_{\per}^0$.
For the reverse, take $\gb \in E_{\per,0}^r \times O_{\per}^r$ and suppose $\ip{\gb}{\nub_c^{\mu}}_0 = 0$.
Classical functional analysis tells us there is $\fb \in E_{\per,0}^2 \times O_{\per}^2$ such that $\Gamma_c^{\mu}[\omega_c^{\mu}]\fb = \gb$.
We then use the structure of $\Gamma_c^{\mu}[\omega_c^{\mu}]$ to bootstrap until we have $\fb \in E_{\per,0}^{r+2} \times O_{\per}^{r+2}$.

\item
If $\Gamma_c^{\mu}[\omega_c^{\mu}]\fb = \gb$, then for $k \in \Z\setminus\{0\}$ we have
\begin{equation}\label{matrix pre-neumann}
\left(\ind-\frac{1}{c^2(\omega_c^{\mu})^2k^2}\tD_{\mu}[\omega_c^{\mu}k]\right)\hat{\fb}(k) = -\frac{1}{c^2(\omega_c^{\mu})^2k^2}\hat{\gb}(k).
\end{equation}
Here $\ind$ is the $2\times2$ identity matrix, and we may ignore the case $k = 0$ since $\hat{\fb}(0) = \hat{\gb}(0) = 0$.
In this part of the proof, we will denote the $\infty$-norm of a matrix $A \in \mathbb{C}^{2\times2}$ by $\norm{A}$ and the 2-norm of a vector $\vb \in \mathbb{C}^2$ by $|\vb|$.

Our goal is to show that 
\begin{equation}\label{neumann goal}
\sup_{|k| \ge 2} \bignorm{\frac{1}{c^2(\omega_c^{\mu})^2k^2}\tD_{\mu}[\omega_c^{\mu}k]} \
< 1,
\end{equation}
in which case the matrix on the left in \eqref{matrix pre-neumann} is invertible by a Neumann series argument, and, moreover, this inverse is uniformly bounded in $k$.
This will show $|\hat{\fb}(k)| \le C(c)|\hat{\gb}(k)|$ for $|k| \ge 2$, and then we will handle the case $|k| = 1$ separately.

We begin with the case $|k| \ge 2$.
The estimate \eqref{omega-ep-mu cruder} yields
\begin{equation}\label{kge2 est1}
\frac{1}{c^2(\omega_c^{\mu})^2} \le \frac{1+\mu}{2},
\end{equation}
while the definition of $\tD_{\mu}[K]$ in \eqref{tDmu defn} implies
\begin{equation}\label{kge2 est2}
\bignorm{\tD_{\mu}[\omega_c^{\mu}k]} \le 2(2+\mu).
\end{equation}
We combine \eqref{kge2 est1} and \eqref{kge2 est2} and assume 
\begin{equation}\label{Mu-3-c}
|\mu| 
\le \min\left\{\frac{1}{10},\Mu_1(c)\right\}
=: \Mu_2(c)
\end{equation}
to find
\[
\bignorm{\frac{1}{c^2(\omega_c^{\mu})^2k^2}\tD_{\mu}[\omega_c^{\mu}k]} 
\le \frac{(1+\mu)(2+\mu)}{k^2}
\le \frac{3}{k^2}
\le \frac{3}{4}
\]
for $|k| \ge 2$.
This proves \eqref{neumann goal}.

Now we show $|\hat{\fb}(\pm1)| \le C_{c}|\hat{\gb}(\pm1)|$.
It is here that we need the additional hypothesis in part \ref{sort of inverse bound} that $\ip{\fb}{\nub_c^{\mu}}_0 = \ip{\gb}{\nub_c^{\mu}}_0 = 0$, which implies
\begin{equation}\label{|k|=1 stuff}
\hat{f}_2(1) = -i\upsilon_c^{\mu}\hat{f}_1(1)
\quadword{and}
\hat{g}_2(1) = -i\upsilon_c^{\mu}\hat{g}_1(1).
\end{equation}
This allows to rearrange the equality $\ft[\Gamma_c^{\mu}[\omega_c^{\mu}]\fb](1)\cdot\e_1 = \hat{\gb}(1)\cdot\e_1$ into
\[
\big(-c^2(\omega_c^{\mu})^2+(2+\mu)(1-\cos(\omega_c^{\mu}))-i\upsilon_c^{\mu}\mu\sin(\omega_c^{\mu})\big)\hat{f}_1(1) 
= \hat{g}_1(1).
\]
Some straightforward calculus proves 
\[
c{K}^2-2(1-\cos(K)) 
> K^2-2(1-\cos(K))
> 2\cos(1)-1
> 0
\]
for $K > 1$, so it follows that 
\[
|-c^2(\omega_c^{\mu})^2+(2+\mu)(1-\cos(\omega_c^{\mu}))-i\upsilon_c^{\mu}\mu\sin(\omega_c^{\mu})|
\ge \frac{2\cos(1)-1}{2},
\]
provided that 
\begin{equation}\label{Mu-4-c defn}
|\mu| 
< \min\left\{\frac{2\cos(1)-1}{8}, \frac{2\cos(1)-1}{4C_{\upsilon}(c)},\Mu_2(c)\right\}
=: \Mu_3(c).
\end{equation}
This shows
\[
|\hat{f}_1(1)| \le C(c)|\hat{g}_1(1)|.
\]

If $\upsilon_c^{\mu} = 0$, then \eqref{|k|=1 stuff} provides $\hat{f}_2(1) = 0 = \hat{g}_2(1)$, and there is nothing more to prove.
Otherwise, we have
\[
\frac{1}{|\upsilon_c^{\mu}|}|\hat{f}_2(1)| 
= |\hat{f}_1(1)|
\le C(c)|\hat{g}_1(1)|
= \frac{1}{|\upsilon_c^{\mu}|}|\hat{g}_2(1)|,
\]
and so $|\hat{f}_2(1)| \le C(c)|\hat{g}_2(1)|$.
The estimates for $k=-1$ follow by the even-odd symmetry of $\fb$ and $\gb$.

Last, if $\gb \in E_{\per,0}^r \times O_{\per}^r$ with $\ip{\gb}{\nub_c^{\mu}}_0 = 0$, then part \ref{range=kernel perp} gives $\fb \in E_{\per,0}^{r+2} \times O_{\per}^{r+2}$ such that $\Gamma_c^{\mu}[\omega_c^{\mu}]\fb = \gb$.
Requiring $\ip{\fb}{\nub_c^{\mu}}_0 = 0$ is enough to make $\fb$ unique; the proof is a straightforward exercise in Hilbert space theory.

\item
A lengthy calculation using patient matrix-vector multiplication along with the identity
\[
\ft[\partial_{\omega}\Gamma_c^{\mu}[\omega_c^{\mu}]\nub_c^{\mu}](k)
= -2c^2\omega_c^{\mu}k^2\hat{\nub_c^{\mu}}(k) + k\tD_{\mu}'(\omega_c^{\mu}k)\hat{\nub_c^{\mu}}(k),
\]
the definition of $\nub_c^{\mu}$ in \eqref{nub defn}, the fact that $\hat{\nub_c^{\mu}}(k) = 0$ for $|k| \ne 1$, and the definition of $\tD_{\mu}$ in \eqref{tDmu defn} and its corresponding componentwise derivative $\tD_{\mu}'$ yields
\begin{multline*}
\ip{\partial_{\omega}\Gamma_c^{\mu}[\omega_c^{\mu}]\nub_c^{\mu}}{\nub_c^{\mu}}_0
= 
\bunderbrace{-(c^2\omega_c^{\mu}+\sin(\omega_c^{\mu}))}{I} 
+\bunderbrace{\mu\sin(\omega_c^{\mu})\big(((\upsilon_c^{\mu})^2-1)+4\upsilon_c^{\mu}+\mu\cos(\omega_c^{\mu})\big)}{II} \\
\\
+ \bunderbrace{2\sin(\omega_c^{\mu})\big((\upsilon_c^{\mu})^2 + 4\upsilon_c^{\mu} + \mu\cos(\omega_c^{\mu})\big)}{III}.
\end{multline*}
Since $\upsilon_c^{\mu} = \O_c(\mu)$, it follows that $II = \O_c(\mu)$ and $III = \O_c(\mu)$.
And since $0 <A_c < \omega_c^{\mu} < B_c$, we have
\[
c^2\omega_c^{\mu}+\sin(\omega_c^{\mu})
> c^2A_c + \sin(A_c).
\]
It follows that for $|\mu|$ small, say, $|\mu| \le \mu_{\per}(c) \le \Mu_3(c)$, we have
\[
|\ip{\partial_{\omega}\Gamma_c^{\mu}[\omega_c^{\mu}]\nub_c^{\mu}}{\nub_c^{\mu}}_0|
\ge \frac{c^2A_c+\sin(A_c)}{2}
> 0.
\qedhere
\]
\end{enumerate}
\end{proof}

\subsection{Bifurcation from a simple eigenvalue}
We are now ready to solve the periodic traveling wave problem $\Phib_c^{\mu}(\phib,\omega) = 0$ posed in \eqref{Phib-c-mu defn}.
We follow \cite{hoffman-wright} and \cite{crandall-rabinowitz} and make the revised ansatz
\[
\phib = a\nub_c^{\mu} + a\psib
\quadword{and}
\omega = \omega_c^{\mu} + \xi.
\]
where $a$, $\xi \in \R$ and $\ip{\nub_c^{\mu}}{\psib}_0 = 0$.
Then our problem $\Phib_c^{\mu}(\phib,\omega) = 0$ becomes 
\[
\Phib_c^{\mu}(a\nub_c^{\mu}+a\psib,\omega_c^{\mu}+\xi) 
= 0.
\]

After some considerable rearranging, we find that $\psib$ and $a$ must satisfy 
\begin{equation}\label{pre fp periodic}
\Gamma_c^{\mu}[\omega_c^{\mu}]\psib
=
\rhs_{c,1}^{\mu}(\xi) + \rhs_{c,2}^{\mu}(\xi) + \rhs_{c,3}^{\mu}(\psib,\xi) + \rhs_{c,4}^{\mu}(\psib,\xi) + a\rhs_{c,5}^{\mu}(\psib,\xi),
\end{equation}
where
\begin{equation}\label{initial periodic rhs}
\begin{aligned}
\rhs_{c,1}^{\mu}(\xi) &:= -2c^2\omega_c^{\mu}\xi(\nub_c^{\mu})'' \\
\\
\rhs_{c,2}^{\mu}(\xi) &:= -(\D_{\mu}[\omega_c^{\mu}+\xi]-\D_{\mu}[\omega_c^{\mu}])\nub_c^{\mu} \\
\\
\rhs_{c,3}^{\mu}(\psib,\xi) &:= -c^2[\xi^2(\nub_c^{\mu})''+(2\omega_c^{\mu}+\xi)\xi\psib''] \\
\\
\rhs_{c,4}^{\mu}(\psib,\xi) &:= -(\D_{\mu}[\omega_c^{\mu}+\xi]-\D_{\mu}[\omega_c^{\mu}])\psib \\
\\
\rhs_{c,5}^{\mu}(\psib,\xi) &:= -\D_{\mu}[\omega_c^{\mu}+\xi]\nl(\nub_c^{\mu}+\psib,\nub_c^{\mu}+\psib).
\end{aligned}
\end{equation}
This is roughly the same system that Hoffman and Wright study when they construct periodic solutions for the small mass problem; specifically, its analogue appears in equation (B.9) in \cite{hoffman-wright}.
Our goal, like theirs, is to rewrite \eqref{pre fp periodic} as a fixed point argument in the unknowns $\psib$ and $\xi$ on the space $E_{\per,0}^2 \times O_{\per}^2 \times \R$ with $\mu$ as a small parameter and $c$ fixed.
We intend to use the following lemma, proved in \cite{faver-wright} as Lemma C.1.

\begin{lemma}\label{fixed point lemma}
Let $\X$ be a Banach space and for $r > 0$ let $\Bfrak(r) = \{x \in \X : \norm{x} \le r\}$.  

\begin{enumerate}[label={\bf(\roman*)}]

\item
For $|\mu| \le \mu_0$ let $F_{\mu} \colon \X \times \R \to \X$ be maps with the property that for some $C_1$, $a_1$, $r_1 > 0$, if $x,y \in \Bfrak(r_1)$ and $|a| \le a_1$, then
\begin{align}
\sup_{|\mu| \le \mu_0} \norm{F_{\mu}(x,a)} &\le C_1\big(|a|+|a|\norm{x}+\norm{x}^2\big) \label{fp lemma bound1} \\
\nonumber \\
\sup_{|\mu| \le \mu_0} \norm{F_{\mu}(x,a)-F_{c}(y,a)} &\le C_1\big(|a| + \norm{x}+\norm{y}\big)\norm{x-y} \label{fp lemma bound2}
\end{align}
Then there exist $a_0 \in (0,a_1], r_0 \in (0,r_1]$ such that for each $|\mu| \le \mu_0$ and $|a| \le a_0$, there is a unique $x_{\mu}^a \in \Bfrak(r_0)$ such that $F_{\mu}(x_{\mu}^a,a) = x_{\mu}^a$. 

\item
Suppose as well that the maps $F_{\mu}(\cdot,a)$ are Lipschitz on $\Bfrak(r_0)$ uniformly in $a$ and $\mu$, i.e., there is $L_1 > 0$ such that 
\begin{equation}\label{fp lemma lip bound}
\sup_{\substack{|\mu| \le \mu_0 \\ \norm{x} \le r_0}} \norm{F_{\mu}(x,a)-F_{\mu}(x,\grave{a})} \le L_1|a-\grave{a}|
\end{equation}
for all $|a|$, $|\grave{a}| \le a_0$. 
Then the mappings $[-a_0,a_0] \to \X \colon a \mapsto x_{\mu}^a$ are also uniformly Lipschitz; that is, there is $L_0 > 0$ such that 
\begin{equation}\label{fp lemma lip concl}
\sup_{|\mu| \le \mu_0} \norm{x_{\mu}^a-x_{\mu}^{\grave{a}}} 
\le L_0|a-\grave{a}|
\end{equation}
for all $|a|$, $|\grave{a}| \le a_1$. 
\end{enumerate}
\end{lemma}

Unlike Hoffman and Wright, we cannot invoke this lemma immediately as the terms $\rhs_{c,1}^{\mu}$ and $\rhs_{c,2}^{\mu}$ do not have quite the right structure.
Namely, they contain linear terms in $\xi$ that have no companion factor of $\mu$, and so these terms will not be small enough to achieve the estimates \eqref{fp lemma bound1} and \eqref{fp lemma bound2}.
Due to a different scaling inherent to the small mass limit, this was not an issue for Hoffman and Wright.

Therefore, we begin to modify the problem \eqref{pre fp periodic} as follows.
Let 
\begin{equation}\label{varpi Pi proj defns}
\varpi_c^{\mu}\psib 
:= \frac{\ip{\psib}{\nub_c^{\mu}}_0}{\norm{\nub_c^{\mu}}_0^2}\nub_c^{\mu}
\quadword{and}
\Pi_c^{\mu} 
:= \ind-\varpi_c^{\mu}.
\end{equation}
The definition of $\nub_c^{\mu}$ in \eqref{nub defn} ensures that the denominator of $\varpi_c^{\mu}$ is bounded away from zero uniformly in $\mu$.
Then \eqref{pre fp periodic} holds if and only if both
\begin{equation}\label{pi applied}
\varpi_c^{\mu}\Gamma_c^{\mu}[\omega_c^{\mu}]\psib = \varpi_c^{\mu}[\rhs_{c,1}^{\mu}(\xi) + \rhs_{c,2}^{\mu}(\xi)+\rhs_{c,3}^{\mu}(\psib,\xi)+\rhs_{c,4}^{\mu}(\psib,\xi)+a\rhs_{c,5}^{\mu}(\psib,\xi)]
\end{equation}
and 
\begin{equation}\label{Pi applied}
\Pi_c^{\mu}\Gamma_c^{\mu}[\omega_c^{\mu}]\psib = \Pi_c^{\mu}[\rhs_{c,1}^{\mu}(\xi) + \rhs_{c,2}^{\mu}(\xi)+\rhs_{c,3}^{\mu}(\psib,\xi)+\rhs_{c,4}^{\mu}(\psib,\xi)+a\rhs_{c,5}^{\mu}(\psib,\xi)].
\end{equation}

Since $\ip{\Gamma_c^{\mu}[\omega_c^{\mu}]\psib}{\nub_c^{\mu}}_0 = 0$ by \eqref{periodic solvability}, the first equation \eqref{pi applied} is equivalent to
\[
-\varpi_c^{\mu}[\rhs_{c,1}^{\mu}(\xi,,w) + \rhs_{c,2}^{\mu}(\xi)]
=
\varpi_c^{\mu}[\rhs_{c,3}^{\mu}(\psib,\xi) + \rhs_{c,4}^{\mu}(\psib,\xi) + a\rhs_{c,5}^{\mu}(\psib,\xi)]
\]
Observe that 
\begin{equation}\label{rhs-6 periodic defn}
-\rhs_{c,1}^{\mu}(\xi)-\rhs_{c,2}^{\mu}(\xi) = 
\bunderbrace{\xi(2c^2\omega_c^{\mu}(\nub_c^{\mu})'' + \partial_{\omega}\D_{\mu}[\omega_c^{\mu}]\nub_c^{\mu})}{\xi\partial_{\omega}\Gamma_c^{\mu}[\omega_c^{\mu}]\nub_c^{\mu}}
+ \bunderbrace{(\D_{\mu}[\omega_c^{\mu}+\xi]-\D_{\mu}[\omega_c^{\mu}] - \xi\partial_{\omega}\D_{\mu}[\omega_c^{\mu}])\nub_c^{\mu}}{\rhs_{c,6}^{\mu}(\xi)},
\end{equation}
where the derivative $\partial_{\omega}\Gamma_c^{\mu}[\omega_c^{\mu}]$ of a scaled Fourier multiplier is discussed in Appendix \ref{fourier multipliers appendix}.
Consequently, \eqref{pi applied} is really equivalent to
\begin{equation}\label{xi eqn}
\xi
= \bunderbrace{\frac{1}{\ip{\partial_{\omega}\Gamma_c^{\mu}[\omega_c^{\mu}]\nub_c^{\mu}}{\nub_c^{\mu}}_0}\varpi_c^{\mu}[\rhs_{c,3}^{\mu}(\psib,\xi) + \rhs_{c,4}^{\mu}(\psib,\xi) + a\rhs_{c,5}^{\mu}(\psib,\xi)+\rhs_{c,6}^{\mu}(\xi)]}{\Xi_c^{\mu}(\psib,\xi,a)},
\end{equation}
and part \ref{ip bound} of Proposition \ref{lambda props} guarantees that $|\ip{\partial_{\omega}\Gamma_c^{\mu}[\omega_c^{\mu}]\nub_c^{\mu}}{\nub_c^{\mu}}_0|$ is bounded away\footnote{One can think of the condition $|\ip{\partial_{\omega}\Gamma_c^{\mu}[\omega_c^{\mu}]\nub_c^{\mu}}{\nub_c^{\mu}}_0| > 0$ as, ultimately, a quantitative version of the ``bifurcation condition'' from the original theorem of Crandall and Rabinowitz, namely, part (d) of Theorem 1.7 in \cite{crandall-rabinowitz}, when their Banach spaces are specified to be Hilbert spaces.} from zero uniformly in $\mu$.

Now we construct the fixed point equation for $\psib$.
By definition of $\Pi_c^{\mu}$ in \eqref{varpi Pi proj defns} and part \ref{range=kernel perp} of Proposition \ref{lambda props}, we have
\[
\Pi_c^{\mu}\Gamma_{c}^{\mu}[\omega_c^{\mu}]\psib 
= \Gamma_{c}^{\mu}[\omega_c^{\mu}]\psib,
\]
so that \eqref{Pi applied} is equivalent to
\begin{equation}\label{Pi applied and simplified}
\Gamma_c^{\mu}[\omega_c^{\mu}]\psib = \Pi_c^{\mu}[\rhs_{c,1}^{\mu}(\xi) + \rhs_{c,2}^{\mu}(\xi)+\rhs_{c,3}^{\mu}(\psib,\xi)+\rhs_{c,4}^{\mu}(\psib,\xi)+a\rhs_{c,5}^{\mu}(\psib,\xi)].
\end{equation}
Next, it is obvious from the definition of $\Pi_c^{\mu}$ that $\ip{\Pi_c^{\mu}\fb}{\nub_c^{\mu}}_0 = 0$.
Consequently, there exists a solution $\psib$ to \eqref{Pi applied and simplified}, but, since $\Gamma_c^{\mu}[\omega_c^{\mu}]$ has a nontrivial kernel by part \ref{upsilon-c-mu part} of Proposition \ref{lambda props}, this solution is not unique.
Part \ref{sort of inverse bound} of that proposition, however, allows us to force uniqueness by taking $\psib$ to be the solution of \eqref{Pi applied and simplified} that also satisfies $\ip{\psib}{\nub_c^{\mu}}_0 = 0$.
Hence the $\psib$ that we seek must satisfy
\[
\psib 
= \Gamma_c^{\mu}[\omega_c^{\mu}]^{-1}\Pi_c^{\mu}[\rhs_{c,1}^{\mu}(\xi) + \rhs_{c,2}^{\mu}(\xi)+\rhs_{c,3}^{\mu}(\psib,\xi)+\rhs_{c,4}^{\mu}(\psib,\xi)+a\rhs_{c,5}^{\mu}(\psib,\xi)].
\]
As with our construction of the fixed-point equation for $\xi$, the terms $\rhs_{c,1}^{\mu}$ and $\rhs_{c,2}^{\mu}$ are too large in $\xi$.
However, by \eqref{xi eqn} we may replace $\xi$ by $\Xi_{c}^{\mu}(\psib,\xi,a)$, which will turn out to correct this overshoot.
So, our equation for $\psib$ is
\begin{equation}\label{psib eqn}
\psib 
= \bunderbrace{\Gamma_c^{\mu}[\omega_c^{\mu}]^{-1}\Pi_c^{\mu}[\rhs_{c,1}^{\mu}(\Xi_{c}^{\mu}(\psib,\xi,a)) + \rhs_{c,2}^{\mu}(\Xi_{c}^{\mu}(\psib,\xi,a))+\rhs_{c,3}^{\mu}(\psib,\xi)+\rhs_{c,4}^{\mu}(\psib,\xi)+a\rhs_{c,5}^{\mu}(\psib,\xi)]}{\Psib_c^{\mu}(\psib,\xi,a)}.
\end{equation}
It is worthwhile pointing out now that 
\begin{equation}\label{periodic smoothing est}
\norm{\Gamma_c^{\mu}[\omega_c^{\mu}]^{-1}\Pi_c^{\mu}\psib}_{r+2}
\le C(c,r)\norm{\psib}_r
\end{equation}
by the estimate in part \ref{sort of inverse bound} from Lemma \ref{lambda props}.

\subsection{Solution of the fixed-point problem}
The fixed-point problem
\[
\begin{cases}
\psib = \Psib_c^{\mu}(\psib,\xi,a) \\
\xi = \Xi_c^{\mu}(\psib,\xi,a)
\end{cases}
\]
is now in a form amenable to Lemma \ref{fixed point lemma}.
One first shows that $\Xi_c^{\mu}$ satisfies the estimates \eqref{fp lemma bound1}, \eqref{fp lemma bound2}, and \eqref{fp lemma lip bound} using, chiefly, the calculus on Fourier multipliers from Lemma \ref{calculus on Fourier multipliers lemma} and the techniques of Appendices B of \cite{hoffman-wright} and C of \cite{faver-wright}.
Then one establishes the same estimates on $\Psib_c^{\mu}$, using the existing estimates on $\Xi_c^{\mu}$ along the way.
Since the techniques so closely resemble those of \cite{hoffman-wright} and \cite{faver-wright}, we omit the details.
\section{Function Theory}

\subsection{Proof of Proposition \ref{symbol of B}}\label{proof of B props appendix}

\subsubsection{The proof of part \ref{B simple zeros}}
Observe that 
\[
\tB_c(0) 
= 4
> 0
\quadword{and}
\tB_c\left(\frac{\pi}{2}\right)
= -\frac{c^2\pi^2}{4} + 2
< -\frac{8}{4}+2
= 0.
\]
The intermediate value theorem furnishes $\omega_c \in (0,\pi/2)$ such that $\tB_c(\omega_c) = 0$.
Since $0 < \omega_c < \pi/2$, we can rewrite the relation $\tB_c(\omega_c) = 0$ as
\[
\omega_c
= \sqrt{\frac{2+2\cos(\omega_c)}{c^2}}
> \sqrt{\frac{2}{c^2}}
= \frac{\sqrt{2}}{|c|},
\]
and so we have the refined bounds $\sqrt{2}/|c| < \omega_c < \pi/2$.
Next, for $k \in \R$, we calculate
\[
|(\tB_c)'(k)|
= 2|c^2k-\sin(k)|
> 0
\]
since $c^2 > 1$.
This shows that the zeros at $z=\pm\omega_c$ are simple and, moreover, unique in $\R$.
The bound \eqref{b-b-c ineq} follows by estimating
\[
|(\tB_c)'(\omega_c)|
\ge 2(c^2A_c-\sin(A_c))
\ge 2(1-\sin(1)).
\]

Now we prove that $z=\pm\omega_c$ are the only zeros of $\tB_c$ on a suitably narrow strip in $\C$.
For $k$, $q \in \R$, we compute
\[
\tB_c(k+iq)
= \bunderbrace{-c^2(k^2-q^2)+2+(e^{-q}+e^q)\cos(k)}{\rhs_c(k,q)}
+ i\bunderbrace{-2c^2kq+(e^{-q}-e^q)\sin(k)}{\I_c(k,q)}.
\]
We claim for a suitable $q_{\B} > 0$, if $|q| \le 3q_{\B}$, then $\I_c(k,q) = 0$ if and only if $k=0$.
But 
\[
\rhs_c(0,q)
= q^2+e^q+e^{-q}
> 0
\]
for all $q \in \R$, and so, if our claim is true, then $\tB_c(z) \ne 0$ for $|\im(z)| \le 3q_{\B}$.

So, we just need to prove the claim about $\I_c$.
Observe that, for $q \ne 0$, $\I_c(k,q) = 0$ if and only if
\begin{equation}\label{Ifrak eq}
\frac{2c^2q}{e^{-q}-e^q} 
= \sinc(k)
:= \frac{\sin(k)}{k}.
\end{equation}
Elementary calculus tells us
\[
\sinc(k)
> -\frac{1}{4}, \ k \in \R
\quadword{and}
2\frac{q}{e^{-q}-e^q} < -\frac{1}{4}, \ |q| \le 3.
\]
So, we set $q_{\B} = 1$ and find that for $|q| \le 3q_{\B}$ and any $k \in \R$
\[
\frac{2c^2q}{e^{-q}-e^q} 
< \frac{2q}{e^{-q}-e^q}
< -\frac{1}{4}
< \sinc(k).
\]
Hence \eqref{Ifrak eq} cannot hold for $|q| \le 3q_{\B}$.

With part \ref{B simple zeros} established, part \ref{B meromorphic} follows at once.

\subsubsection{The proof of part \ref{B quadratic decay}}
The reverse triangle inequality implies
\[
|\tB_c(z)|
\ge c^2|z|^2 - |2+2\cos(z)|
\]
We estimate
\[
|2+2\cos(z)|
\le 2+2e^{2q_{\B}}
\le 2+2e^2
\]
for $z \in S_{2q_{\B}}$, since we assumed $q_{\B} = 1$.
Hence
\[
|\tB_c(z)|
\ge c^2|z|^2 - (2+2e^2).
\]
Now take
\[
|z| 
\ge \sqrt{\frac{2+2e^2}{c^2-1/2}}
=: r_{\B}(c)
\]
to find that 
\[
\frac{1}{2}|z|^2 
\le c^2|z|^2-(2+2e^2)
\le |\tB_c(z)|.
\]

\subsubsection{The proof of part \ref{B L1 q estimate}}
Residue theory tells us that 
\begin{equation}\label{laurent decomp}
\frac{1}{\tB_c(z)}
= \frac{1}{(\tB_c)'(\omega_c)(z-\omega_c)} 
+ \frac{1}{(\tB_c)'(-\omega_c)(z+\omega_c)}
+ \Rfrak_c(z), \ z \in S_{3q_{\B}}\setminus\{\pm\omega_c\},
\end{equation}
where $\Rfrak_c$ is analytic on $S_{3q_{\B}}$.

We need two elementary Fourier transforms.
Let $\alpha$, $\beta \in \C$ with $\re(\alpha) > 0$ and $\re(\beta) < 0$ and define
\[
\mathsf{e}_{\alpha}(x)
:= \begin{cases}
e^{\alpha{x}}, \ x < 0 \\
0, \ x \ge 0
\end{cases}
\quadword{and}
\mathsf{e}^{\beta}(x)
:= \begin{cases}
0, \ x < 0 \\
e^{\beta{x}}, \ x \ge 0.
\end{cases}
\]
Then
\[
\hat{\mathsf{e}}_{\alpha}(k)
= \frac{1}{\sqrt{2\pi}(\alpha-ik)}
\quadword{and}
\hat{\mathsf{e}^{\beta}}(k)
= -\frac{1}{\sqrt{2\pi}(\beta-ik)}.
\]
Since
\[
\frac{1}{-k+iq\pm\omega_c}
= \frac{i\sqrt{2\pi}}{\sqrt{2\pi}((-q\pm{i}\omega_c)-ik)},
\]
we have
\begin{equation}\label{special inv ft-}
\ft^{-1}\left[\frac{1}{-\cdot+iq\pm\omega_c}\right](x)
= \begin{cases}
i\sqrt{2\pi}e^{(-q\pm{i}\omega_c)x}, \ x < 0 \\
0, \ x \ge 0
\end{cases}
\end{equation}
if $q < 0$ and 
\begin{equation}\label{special inv ft+}
\ft^{-1}\left[\frac{1}{-\cdot+iq\pm\omega_c}\right](x)
= \begin{cases}
0, \ x < 0 \\
i\sqrt{2\pi}e^{(-q\pm{i}\omega_c)x}, \ x \ge 0
\end{cases}
\end{equation}
if $q > 0$.
In either case, using the decomposition \eqref{laurent decomp} and the transforms \eqref{special inv ft-} and \eqref{special inv ft+}, we bound
\begin{equation}\label{intermediate L1 inv ft bound}
\bignorm{\ft^{-1}\left[\frac{1}{\tB_c(-\cdot+iq)}\right]}_{L^1} 
\le \sqrt{2\pi}\left(\frac{1}{|q| + i\omega_c} + \frac{1}{|q|-i\omega_c}\right) + \norm{\ft^{-1}[\Rfrak_c(-\cdot+iq)}_{L^1},
\end{equation}
where
\begin{equation}\label{last L1 function to estimate}
\ft^{-1}[\Rfrak_c(-\cdot+iq)](x)
= \frac{1}{\sqrt{2\pi}}\int_{-\infty}^{\infty} e^{ikx}\Rfrak_c(-k+iq) \dk
= -\frac{e^{-qx}}{\sqrt{2\pi}}\int_{\im(z) = q} e^{-ixz}\Rfrak_c(z) \dz.
\end{equation}

We claim that 
\begin{equation}\label{O(q) aux}
\sup_{\substack{1 < |c| \le \sqrt{2} \\ 0 < |q| \le q_{\B}}} \norm{\ft^{-1}[\Rfrak_c(-\cdot+iq)]}_{L^1(\R_{\pm})}
< \infty.
\end{equation}
If this estimate holds, then \eqref{intermediate L1 inv ft bound} will establish the desired estimate \eqref{inv ft L1 O(q) est}.
We first prove \eqref{O(q) aux} for the $L^1(\R_-)$ case and then comment briefly on how to proceed with the similar $L^1(\R_+)$ estimate.
Observe that $\Rfrak_c$ vanishes as $|\re(z)| \to \infty$, since the other three functions in \eqref{laurent decomp}, where $\Rfrak_c$ is defined implicitly, all vanish as $|\re(z)| \to \infty$.
Moreover, $\Rfrak_c$ is analytic on the strip $S_{3q_{\B}}$. 
We can therefore shift the integration contour in \eqref{last L1 function to estimate} from $\im(z) = q$ to $\im(z) = 2q_{\B}$ and obtain
\[
\int_{\im(z) = q} e^{-ixz}\Rfrak_c(z) \dz
= e^{2q_{\B}x}\int_{-\infty}^{\infty} e^{ikx}\Rfrak_c(k + 2iq_{\B}) \dk.
\]
From \eqref{last L1 function to estimate}, we conclude
\begin{equation}\label{convenient expr for inv ft}
\ft^{-1}[\Rfrak_c(-\cdot+iq)](x)
= -\frac{e^{(2q_{\B}-q)x}}{\sqrt{2\pi}}\int_{-\infty}^{\infty} e^{ikx}\Rfrak_c(k+2iq_{\B}) \dk,
\end{equation}
Since $2q_{\B}-q > 0$, we will have \eqref{O(q) aux} in the $L^1(\R_-)$ case if we can show
\begin{equation}\label{hopeful est}
\sup_{\substack{1 < |c| \le \sqrt{2} \\ x \in \R_-}}
\left|\int_{-\infty}^{\infty} e^{ikx}\Rfrak_c(k+2iq_{\B}) \dk\right| 
< \infty.
\end{equation}

The integral $\medint_{-\infty}^{\infty} e^{ikx}\Rfrak_c(-k+2iq_{\B}) \dk$ in \eqref{hopeful est} is the inverse Fourier transform of $\Rfrak_c(-\cdot+2iq_{\B})$, which we may estimate using the implicit definition of $\Rfrak_c$ in \eqref{laurent decomp}:
We have
\begin{multline*}
|\ft^{-1}[\Rfrak_c(-\cdot+2iq_{\B})](x)|
\le \bunderbrace{\left|\ft^{-1}\left[\frac{1}{\tB_c(-\cdot+2iq_{\B})}\right](x)\right|}{I}
+ \bunderbrace{\frac{1}{|(\tB_c)'(\omega_c)|}\left|\ft^{-1}\left[\frac{1}{(-\cdot+2iq_{\B}-\omega_c)}\right](x)\right|}{II} \\
\\
+ \bunderbrace{\frac{1}{|(\tB_c)'(-\omega_c)|}\left|\ft^{-1}\left[\frac{1}{(-\cdot+2iq_{\B}+\omega_c)}\right](x)\right|}{III}.
\end{multline*}
The quantities $II$ and $III$ are uniformly bounded for $1 < |c| \le \sqrt{2}$ due to the calculations of the specific Fourier transforms in \eqref{special inv ft-} and \eqref{special inv ft+} and the bounds \eqref{b-b-c ineq}.
We obtain a uniform bound for $I$ over $1 < |c| \le \sqrt{2}$	 by estimating the integral $\medint_{-\infty}^{\infty} (1/|\tB_c(-k+2iq_{\B})|) \dk$ over the intervals $(-\infty,r_0]$ and $[r_0,\infty)$, where $r_0$ is defined in \eqref{r0 sup}, with the quadratic estimate \eqref{B quadratic estimate} and then over the interval $[-r_0,r_0]$ just by bounding $1/|\tB_c(-k+2iq_{\B})|$ uniformly in $c$ for $|k| \le r_0$.
This proves \eqref{hopeful est}.

To study $\norm{\ft^{-1}[\Rfrak_c(-\cdot+iq)}_{L^1(\R_+)}$, we repeat the work above, except we shift the contour in \eqref{last L1 function to estimate} to $\im(z) = -2q_{\B}$, so that 
\[
\int_{-\infty}^{\infty} e^{ikx}\Rfrak_c(-k+iq) \dk
= -e^{-(2q_{\B}+q)x}\int_{-\infty}^{\infty} e^{-ikx}\Rfrak_c(k-2iq_{\B}) \dk .
\]
Then we obtain an estimate analogous to \eqref{hopeful est}, which in turn proves \eqref{O(q) aux} for $L^1(\R_+)$.

\subsection{Additional estimates on the Friesecke-Pego solitary wave $\varsigma_{\cep}$}\label{additional friesecke-pego estimates appendix}
We deduce the following lemma from Lemmas 3.1 and 3.2 in \cite{hoffman-wayne}.

\begin{lemma}[Hoffman \& Wayne]\label{hoffman-wayne lemma}
There exist $\ep_{\HWa} \in (0,\ep_{\FP}]$ and $\apzc$, $C > 0$ such that
\[
\bignorm{e^{(\apzc/\ep)\cdot}\left(\frac{1}{\ep^2}\varsigma_{\cep}\left(\frac{\cdot}{\ep}\right)-\sigma\right)}_{H^1} \le C\ep^2
\]
for $0 < \ep < \ep_{\HWa}$.
\end{lemma}

Now we prove a bevy of estimates on $\varsigma_{\cep}$, all of which say either that $\varsigma_{\cep}$ is ``small'' in a certain norm or that $\varsigma_{\cep}$ and $\ep^2\sigma(\ep\cdot)$, where $\sigma$ was defined in \eqref{sigma defn}, are ``close'' in some norm.

\begin{proposition}\label{hoffman-wayne proposition}
There is $C > 0$ such that the following estimates hold for all $\ep \in (0,\ep_{\HWa})$.

\begin{enumerate}[label={\bf(\roman*)},ref={(\roman*)}]

\item\label{HW fundamental}
$\norm{e^{\apzc|\cdot|}(\varsigma_{\cep}-\ep^2\sigma(\ep\cdot))}_{L^2} \le C\ep^{7/2}$.

\item\label{FP L1-1}
$\norm{\varsigma_{\cep}-\ep^2\sigma(\ep\cdot)}_{L^1} \le C\ep^{7/2}$.

\item\label{FP L1-2}
$\norm{\varsigma_{\cep}}_{L^1} \le C\ep$.

\item\label{FP Linfty-1}
$\norm{\varsigma_{\cep}-\ep^2\sigma(\ep\cdot)}_{L^{\infty}} \le C\ep^2$.

\item\label{FP Linfty-2}
$\norm{\varsigma_{\cep}}_{L^{\infty}} \le C\ep^2$.

\item\label{FP weighted Linfty}
$\norm{e^{q\cdot}\varsigma_{\cep}}_{L^{\infty}} \le C\ep^2$ for $0 < q < \min\{\ep/2,\apzc\}$.
\end{enumerate}
\end{proposition}

\begin{proof}

\begin{enumerate}[label={\bf(\roman*)}]

\item
We use the evenness of the integrand and substitute $u = \ep{x}$ to find
\begin{multline*}
\norm{e^{\apzc|\cdot|}(\varsigma_{\cep}-\ep^2\sigma(\ep\cdot))}_{L^2}^2
=\int_{-\infty}^{\infty} e^{2\apzc|x|}|\varsigma_{\cep}(x) - \ep^2\sigma(\ep{x})|^2 \dx
= 2\ep^3\int_0^{\infty} e^{2\apzc{u}/\ep}\left|\frac{1}{\ep^2}\varsigma_{\cep}\left(\frac{u}{\ep}\right)-\sigma(u)\right|^2 \du \\
\\
\le 2\ep^3\int_{-\infty}^{\infty} e^{2\apzc{u}/\ep}\left|\frac{1}{\ep^2}\varsigma_{\cep}\left(\frac{u}{\ep}\right)-\sigma(u)\right|^2 \du 
= 2\ep^3\bignorm{e^{(\apzc/\ep)\cdot}\left(\frac{1}{\ep^2}\varsigma_{\cep}\left(\frac{\cdot}{\ep}\right)-\sigma\right)}_{L^2}^2 
\le C\ep^7.
\end{multline*}

\item
The Cauchy-Schwarz inequality implies
\begin{multline*}
\norm{\varsigma_{\cep}-\ep^2\sigma(\ep\cdot)}_{L^1}
= \int_{-\infty}^{\infty} |\varsigma_{\cep}(x) - \ep^2\sigma(\ep{x})| \dx
= \int_{-\infty}^{\infty} e^{-\apzc|x|}\big(e^{\apzc|x|}|\varsigma_{\cep}(x) - \ep^2\sigma(\ep{x})|\big) \dx \\
\\
\le \left(\int_{-\infty}^{\infty} e^{-2\apzc|x|} \dx\right)^{1/2}\left(\int_{-\infty}^{\infty} e^{2\apzc|x|}|\varsigma_{\cep}(x)-\ep^2\sigma(\ep{x})|^2\dx\right)^{1/2}.
\end{multline*}
The first integral on the second line is just a constant (depending on $\apzc$), and the second integral is $\O(\ep^{7/2})$ by part \ref{HW fundamental}.

\item
This follows directly from part \ref{FP L1-1} and the calculation $\norm{\sigma(\ep\cdot)}_{L^1} = \O(\ep^{-1})$.

\item
The Sobolev embedding and the Friesecke-Pego estimate \eqref{original friesecke-pego estimate} imply
\[
\bignorm{\frac{1}{\ep^2}\varsigma_{\cep}\left(\frac{\cdot}{\ep}\right)-\sigma}_{L^{\infty}} \le C\ep^2,
\]
and then rescaling gives the desired estimate.

\item
This follows directly from part \ref{FP Linfty-1} and the calculation $\norm{\sigma(\ep\cdot)}_{L^{\infty}} = \O(1)$.

\item
First, we have
\[
\norm{e^{q\cdot}\varsigma_{\cep}}_{L^{\infty}}
\le \norm{e^{q\cdot}\varsigma_{\cep}}_{L^{\infty}(\R_-)}
+ \norm{e^{q\cdot}\varsigma_{\cep}}_{L^{\infty}(\R_+)}.
\]
Since $q > 0$, part \ref{FP Linfty-2} gives
\[
\norm{e^{q\cdot}\varsigma_{\cep}}_{L^{\infty}(\R_-)}
\le \norm{\varsigma_{\cep}}_{L^{\infty}(\R_-)}
\le \norm{\varsigma_{\cep}}_{L^{\infty}}
\le C\ep^2. 
\]
Next,
\begin{equation}\label{last HW R+}
\norm{e^{q\cdot}\varsigma_{\cep}}_{L^{\infty}(\R_+)}
\le \ep^2\norm{e^{q\cdot}\sigma(\ep\cdot)}_{L^{\infty}(\R_+)}
+ \norm{e^{q\cdot}(\varsigma_{\cep}-\ep^2\sigma(\ep\cdot))}_{L^{\infty}(\R_+)}.
\end{equation}
Since $0 < q < \ep/2$ and $x \ge 0$, we have
\[
|e^{q{x}}\sigma(\ep{x})| 
\le Ce^{q{x}}e^{-\ep{x}}
\le Ce^{-\ep{x}/2}
\le C,
\]
hence $\norm{e^{q\cdot}\sigma(\ep\cdot)}_{L^{\infty}(\R_+)} = \O(1)$.
For the other term in \eqref{last HW R+}, we rewrite
\begin{multline*}
\norm{e^{q\cdot}(\varsigma_{\cep}-\ep^2\sigma(\ep\cdot))}_{L^{\infty}(\R_+)}
\le \norm{e^{\apzc|\cdot|}(\varsigma_{\cep}-\ep^2\sigma(\ep\cdot))}_{L^{\infty}(\R_+)}
\le \norm{e^{\apzc|\cdot|}(\varsigma_{\cep}-\ep^2\sigma(\ep\cdot))}_{L^{\infty}} \\
\\
= \ep^2\bignorm{e^{(\apzc/\ep)|\cdot|}\left(\frac{1}{\ep^2}\varsigma_{\cep}\left(\frac{\cdot}{\ep}\right)-\sigma\right)}_{L^{\infty}}
\le C\ep^4
\end{multline*}
by Lemma \ref{hoffman-wayne lemma}.
\qedhere
\end{enumerate}
\end{proof}
\section{Some Proofs for the Verification of Hypotheses \ref{hypothesis 3} and \ref{hypothesis 4} in the case $|c| \gtrsim 1$}\label{hypos 3-4 appendix}

\subsection{Background from the theory of modified functional differential equations}
We extract the content of the following theorem from Theorem 3.2 and Proposition 3.4 in \cite{hvl}.

\begin{theorem}[Hupkes \& Verduyn Lunel]\label{hvl theorem}
Let $A_0, A_1, \ldots, A_n \in \C^{m \times m}$ and let $d_0 < d_1 < \cdots < d_n$ be real numbers.
For $\fb \in W^{1,\infty}$, define
\begin{equation}\label{abstract hvl syst}
\Lambda\fb 
:= \fb' - \sum_{j=1}^n A_jS^{d_j}\fb.
\end{equation}
Let
\begin{equation}\label{abstract characteristic equation}
\Delta(z)
:= z\ind-\sum_{j=1}^n e^{d_jz}A_j \in \C^{m \times m},
\end{equation}
where $\ind$ is the $m\times{m}$ identity matrix.
We call $\Delta$ the characteristic matrix corresponding to the system \eqref{abstract hvl syst}.
Suppose that $q \in \R$ with $\det[\Delta(z)] \ne 0$ for $\re(z) = q$.
Then $\Lambda$ is an isomorphism between $W_q^{1,\infty}(\R,\C^m)$ and $L_q^{\infty}(\R,\C^m)$, where these spaces were defined in Definition \ref{one sided expn weighted spaces defn}.
We denote its inverse by $\Lambda_q^{-1}$.
Moreover, we have two formulas for $\Lambda_q^{-1}$.

\begin{enumerate}[label={\bf(\roman*)}, ref={(\roman*)}]

\item\label{hvl-i}
There is a function $\Gscrb_q \in \cap_{p=1}^{\infty} L^p(\R,\C^{m \times m})$ such that for $\gb \in L_q^{\infty}(\R,\C^{m \times m})$,
\begin{equation}\label{Lambda inv via convolution}
(\Lambda_q^{-1})(x)
= e^{qx}\int_{-\infty}^{\infty} e^{-qs}\Gscrb_q(x-s)\gb(s)\ds
\end{equation}	
Moreover,
\begin{equation}\label{convolution kernel}
\hat{\Gscrb}_q(k) = \Delta(ik+q)^{-1}, \ k \in \R.
\end{equation}

\item\label{hvl-ii}
Given $\delta \in (0,|q|)$ and $\gb \in L_q^{\infty}(\R,\C^m)$, we have
\begin{equation}\label{Lambda inv formula}
(\Lambda_q^{-1}\gb)(x)
= \frac{1}{2\pi{i}}\int_{q+\delta-i\infty}^{q+\delta+i\infty} e^{xz}\Delta(z)^{-1}\lt_+[\gb](z)\dz
+ \frac{1}{2\pi{i}}\int_{q-\delta-i\infty}^{q-\delta+i\infty} e^{xz}\Delta(z)^{-1}\lt_-[\gb](z) \dz,
\end{equation}
where $\lt_{\pm}$ are the Laplace transforms defined in Appendix \ref{laplace transform appendix}.
\end{enumerate}
\end{theorem}

\subsection{Proof of Proposition \ref{b-residue}}\label{proof of prop b-residue appendix}
That $\varsigma_{\cep}g \in L_{q_{\ep}}^{\infty}$ for $g \in L_{-q_{\ep}}^{\infty}$ is a straightforward consequence of part \ref{FP weighted Linfty} of Proposition \ref{hoffman-wayne proposition}.

\subsubsection{Residue theory}\label{residue strategery}
Take $R > 0$ so large that $0 < \omega_{\cep} < R$ for all $0 < \ep < \ep_{\B}$ and consider the contour $\varrho_{\ep,1}^{\pm}(R) + \varrho_{\ep,2}^{\pm}(R) + \varrho_{\ep,3}^{\pm}(R) + \varrho_{\ep,4}^{\pm}(R)$ sketched in Figure \ref{fig_contour}.
The residue theorem tells us
\begin{multline}\label{residue!}
\frac{1}{2\pi{i}}\int_{\varrho_{\ep,1}^{\pm}(R) + \varrho_{\ep,2}^{\pm}(R) + \varrho_{\ep,3}^{\pm}(R) + \varrho_{\ep,4}^{\pm}(R)}
\K_{\ep}(x,z)\lt_{\pm}[\varsigma_{\cep}g](z) \dz
= 
\res\left(\K_{\ep}(x,z)\lt_{\pm}[\varsigma_{\cep}g](z); z=i\omega\right) \\
\\
+\res\left(\K_{\ep}(x,z)\lt_{\pm}[\varsigma_{\cep}g](z); z=-i\omega\right).
\end{multline}
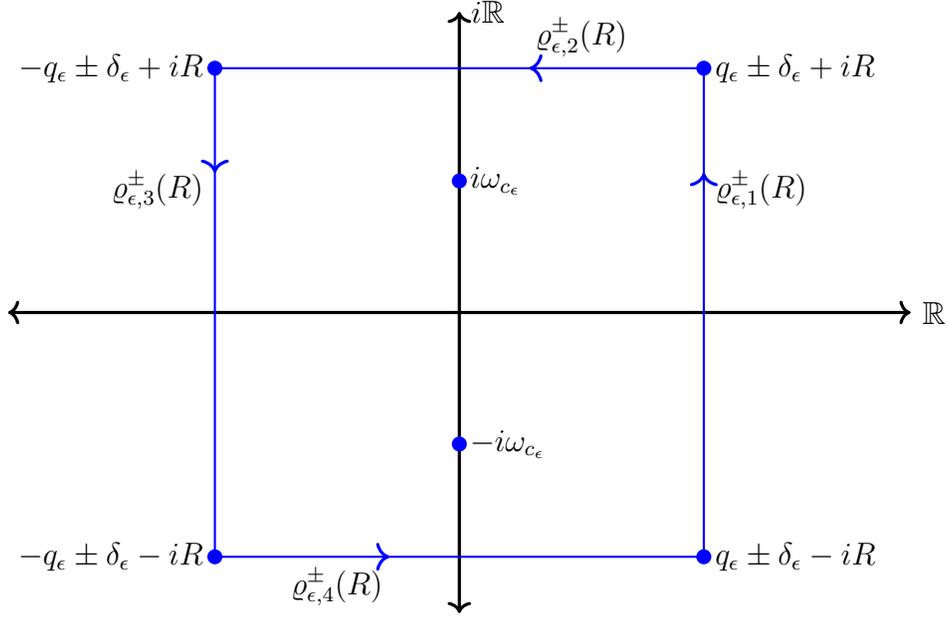
\begin{figure}
\[
\begin{tikzpicture}

\draw[<->, very thick] (-6,0)--(6,0)node[right]{$\R$};
\draw[<->,very thick] (0,-4)--(0,4)node[right]{$i\R$};

\draw[thick,blue] 
(3.25,0)
--node[midway,right,black]{$\varrho_{\ep,1}^{\pm}(R)$} 
(3.25,3.25)node[right,black]{$q_{\ep}\pm\delta_{\ep}+iR$}
-- node[midway,above,black]{$\varrho_{\ep,2}^{\pm}(R)$}
(0,3.25)
-- (-3.25,3.25)node[left,black]{$-q_{\ep}\pm\delta_{\ep}+iR$}
-- node[midway,left,black]{$\varrho_{\ep,3}^{\pm}(R)$}
(-3.25,0) 
-- (-3.25,-3.25)node[left,black]{$-q_{\ep}\pm\delta_{\ep}-iR$}
-- node[midway,below,black]{$\varrho_{\ep,4}^{\pm}(R)$}
(0,-3.25)
-- (3.25,-3.25)node[right,black]{$q_{\ep}\pm\delta_{\ep}-iR$}
-- cycle[arrow inside={opt={blue,scale=1.5}}{1/14,3/14,6/14,10/14}];

\fill[blue] (3.25,3.25) circle(.1);
\fill[blue] (3.25,-3.25) circle(.1);
\fill[blue] (-3.25,-3.25) circle(.1);
\fill[blue] (-3.25,3.25) circle(.1);

\fill[blue] (0,1.75)node[right,black]{$i\omega_{\cep}$} circle(.1);
\fill[blue] (0,-1.75)node[right,black]{$-i\omega_{\cep}$} circle(.1);

\end{tikzpicture}
\]
\caption{The contour $\varrho_{\ep,1}^{\pm}(R) + \varrho_{\ep,2}^{\pm}(R) + \varrho_{\ep,3}^{\pm}(R) + \varrho_{\ep,4}^{\pm}(R)$}
\label{fig_contour}
\end{figure}

By the inverse formulas in \eqref{B inv formula}, we can write
\begin{equation}\label{B_-eta contour}
[\B_{\cep}^-]^{-1}[\varsigma_{\cep}g](x)
= -\frac{1}{2\pi{i}}\lim_{R \to \infty} \left(\int_{\varrho_{\ep,3}^+(R)} \K_{\ep}(x,z)\lt_+[\varsigma_{\cep}g](z) \dz
+\int_{\varrho_{\ep,3}^-(R)} \K_{\ep}(x,z)\lt_+[\varsigma_{\cep}g](z) \dz\right)
\end{equation}
and
\begin{equation}\label{B_+eta contour}
[\B_{\cep}^+]^{-1}[\varsigma_{\cep}g](x)
= \frac{1}{2\pi{i}}\lim_{R \to \infty} \left(\int_{\varrho_{\ep,1}^+(R)} \K_{\ep}(x,z)\lt_+[\varsigma_{\cep}g](z) \dz
+\int_{\varrho_{\ep,1}^-(R)} \K_{\ep}(x,z)\lt_-[\varsigma_{\cep}g](z) \dz\right).
\end{equation}
If we can show that 
\begin{equation}\label{vanishing curves}
\lim_{R \to \infty} \int_{\varrho_{\ep,2}^{\pm}(R)} \K_{\ep}(x,z)\lt_{\pm}[\varsigma_{\cep}g](z) \dz 
= \lim_{R \to \infty} \int_{\varrho_{\ep,4}^{\pm}(R)} \K_{\ep}(x,z)\lt_{\pm}[\varsigma_{\cep}g](z) \dz
=0
\end{equation}
then we can combine \eqref{residue!}, \eqref{B_-eta contour}, and \eqref{B_+eta contour} to conclude 
\begin{equation}\label{formula pre-res}
\begin{aligned}
[\B_{\cep}^-]^{-1}[\varsigma_{\cep}g](x)
&= [\B_{\cep}^+]^{-1}[\varsigma_{\cep}g](x)
- \res\left(\K_{\ep}(x,z)\lt_+[\varsigma_{\cep}g](z); z=i\omega_{\cep}\right) \\
&- \res\left(\K_{\ep}(x,z)\lt_+[\varsigma_{\cep}g](z); z=-i\omega_{\cep}\right) 
-\res\left(\K_{\ep}(x,z)\lt_-[\varsigma_{\cep}g](z); z=i\omega_{\cep}\right) \\
&-\res\left(\K_{\ep}(x,z)\lt_-[\varsigma_{\cep}g](z); z=-i\omega_{\cep}\right).
\end{aligned}
\end{equation}
We will then be able to massage the right side of \eqref{formula pre-res} into our desired formula \eqref{residue switcheroo}.
So, we commence with the residue theory.

\subsubsection{Calculation of the residues \eqref{residue!}}
Recall from part \ref{B simple zeros} of Proposition \ref{symbol of B} that 
\[
|(\tB_{\cep})'(\pm\omega_{\cep})|
\ge b_{\B} > 0
\]
for $0 < \ep < \ep_{\B}$ and $|\mu| \le \mu_{\B}(\cep)$.
This hypothesis and the definition of $\K_{\ep}$ in \eqref{K defn} ensures that $\K_{\ep}(x,z)\lt_{\pm}[\varsigma_{\cep}g](z)$ has simple poles at $z=\pm{i}\omega_{\cep}$ and so
\begin{equation}\label{res+ep}
\res(\K_{\ep}(x,z)\lt_{\pm}[\varsigma_{\cep}g](z);z={i\omega_{\cep}})
= i\left(\frac{\lt_{\pm}[\varsigma_{\cep}g](i\omega_{\cep})}{(\tB_{\cep})'(\omega_{\cep})}\right)e^{i\omega_{\cep}x}
\end{equation}
and
\begin{equation}\label{res-ep}
\res(\K_{\ep}(x,z)\lt_{\pm}[\varsigma_{\cep}g](z);z={-i\omega_{\cep}})
= -i\left(\frac{\lt_{\pm}[\varsigma_{\cep}g](-i\omega_{\cep})}{(\tB_{\cep})'(-\omega_{\cep})}\right)e^{-i\omega_{\cep}x}.
\end{equation}

\subsubsection{Estimates on the Laplace transforms in \eqref{res+ep} and \eqref{res-ep}}\label{Laplace estimates}
Throughout this section, we will need the estimates on $\varsigma_{\cep}$ from Proposition \ref{hoffman-wayne proposition}.

We calculate
\[
|\lt_+[\varsigma_{\cep}g](\pm{i\omega_{\cep}})|
 = \left|\int_0^{\infty} e^{\mp{i\omega_{\cep}}s}\varsigma_{\cep}(s)g(s)\ds\right|
 \le \int_0^{\infty} \varsigma_{\cep}(s) |g(s)| \ds.
\]
Now we multiply by $1 = e^{q_{\ep}{s}}e^{-q_{\ep}{s}}$ to find
\begin{multline*}
\int_0^{\infty} \varsigma_{\cep}(s)|g(s)|\ds
= \int_0^{\infty} e^{-q_{\ep}{s}}\varsigma_{\cep}(s)e^{q_{\ep}{s}}|g(s)| \ds
\le \norm{g}_{L_{-q_{\ep}}^{\infty}}\int_0^{\infty} e^{-q_{\ep}{s}}\varsigma_{\cep}(s) \ds \\
\\
\le \norm{g}_{L_{-q_{\ep}}^{\infty}}\int_0^{\infty} \varsigma_{\cep}(s) \ds 
\le \norm{g}_{L_{-q_{\ep}}^{\infty}}\norm{\varsigma_{\cep}}_{L^1}
\le C\ep\norm{g}_{L_{-q_{\ep}}^{\infty}}.
\end{multline*}
Here we need to recall from Theorem \ref{friesecke-pego} that $\varsigma_{\cep} > 0$, so $\norm{\varsigma_{\cep}}_{L^1} = \medint_{-\infty}^{\infty} \varsigma_{\cep}(x) \dx$.

The situation for $\lt_-$ is slightly different.
With the same manipulations as above, we find
\[
|\lt_-[\varsigma_{\cep}g](\pm{i\omega_{\cep}})|
\le \norm{g}_{L_{-q_{\ep}}^{\infty}}\int_0^{\infty} e^{q_{\ep}{s}}\varsigma_{\cep}(s)\ds.
\]
Now the exponential contains the positive factor $q_{\ep}$, so we cannot blithely ignore it.
Instead, we bound this integral as 
\[
\int_0^{\infty} e^{q_{\ep}{s}}\varsigma_{\cep}(s)\ds
\le \bunderbrace{\ep^2\int_0^{\infty} e^{q_{\ep}{s}}\sigma(\ep{s})\ds}{\I_{\ep,1}}
 + \bunderbrace{\int_0^{\infty} e^{q_{\ep}{s}}|\varsigma_{\cep}(s) - \ep^2\sigma(\ep{s})| \ds}{\I_{\ep,2}}.
\]
Since $\sigma(X) = \sech^2(X/2)/4$, we use the estimates on $q_{\ep}$ from \eqref{very hungry inequality} to bound
\[
\I_{\ep,1}
\le C\ep^2\int_0^{\infty} e^{q_{\ep}{s}}e^{-\ep{s}}\ds
\le C\ep^2\int_0^{\infty} e^{-\ep{s}/2} \ds
= \O(\ep).
\]

Next, we rewrite
\begin{multline*}
\I_{\ep,2}
= \int_0^{\infty} (e^{q_{\ep}{s}}e^{-\apzc\ep{s}})\big(e^{\apzc\ep{s}}|\varsigma_{\cep}(s)-\ep^2\sigma(\ep{s})|\big) \ds \\
\\
\le \bunderbrace{\left(\int_0^{\infty} e^{2(q_{\ep}-\apzc\ep)s}\ds\right)^{1/2}}{\I_{\ep,3}}\bunderbrace{\left(\int_0^{\infty} e^{2\apzc\ep{s}}|\varsigma_{\cep}(s)-\ep^2\sigma(\ep{s})|^2\ds\right)^{1/2}}{\I_{\ep,4}}.
\end{multline*}
We appeal again to \eqref{very hungry inequality} to find
\[
\I_{\ep,3} 
\le \left(\int_0^{\infty} e^{-2\apzc\ep{s}}\ds\right)^{1/2}
= \O(\ep^{-1/2}).
\]
Last, since $0 \le e^{\apzc\ep{s}} \le e^{\apzc{s}}$ for $0 \le \ep \le 1$ and $s \ge 0$, we use part \ref{HW fundamental} of Proposition \ref{hoffman-wayne proposition} to estimate
\[
\I_{\ep,4}
\le \left(\int_0^{\infty} e^{2\apzc{s}}|\varsigma_{\cep}(s)-\ep^2\sigma(\ep{s})|^2\ds\right)^{1/2}
\le \norm{e^{\apzc|\cdot|}(\varsigma_{\cep}-\ep^2\sigma(\ep\cdot))}_{L^2}
\le C\ep^2.
\]

Combining all these estimates, we conclude
\begin{equation}\label{ultimate Laplace estimate}
|\lt_{\pm}[\varsigma_{\cep}g](\pm{i\omega_{\cep}})|
\le C\ep\norm{g}_{L_{-q_{\ep}}^{\infty}}.
\end{equation}

\subsubsection{The limits \eqref{vanishing curves}}
We prove only that 
\begin{equation}\label{vanishing limit eg}
\lim_{R \to \infty} \int_{\varrho_{\ep,2}^{+}(R)} \K_{\ep}(x,z)\lt_{+}[\varsigma_{\cep}g](z) \dz 
= 0,
\end{equation}
the other cases being similar.
We parametrize the line segment $\varrho_{\ep,2}^+(R)$ by
\[
\zpzc_{\ep}(t,R)
:= (1-t)(q_{\ep} + \delta_{\ep} + iR) + t(-q_{\ep}+\delta_{\ep}+iR)
= -2q_{\ep}{t} + \delta_{\ep}+q_{\ep} + iR,
\]
and so the line integral in \eqref{vanishing limit eg} over $\varrho_{\ep,2}^+(R)$ is 
\[
-2q_{\ep}\int_0^1 \K_{\ep}(x,\zpzc_{\ep}(t,R))\lt_+[\varsigma_{\cep}g](\zpzc_{\ep}(t,R)) \dt.
\]
To establish \eqref{vanishing limit eg}, it suffices, of course, to show that 
\begin{equation}\label{suffice contour vanish}
\lim_{R \to \infty} \max_{0 \le t \le 1} | \K_{\ep}(x,\zpzc_{\ep}(t,R))\lt_+[\varsigma_{\cep}g](\zpzc_{\ep}(t,R))|
=0.
\end{equation}
The methods of Section \ref{Laplace estimates} can be adapted to show that $|\lt_+[\varsigma_{\cep}g](\zpzc_{\ep}(t,R))|$ is bounded above by a constant independent of $t$ or $R$ (but dependent on $g$ and $\ep$, although that does not matter here).
Next, we refer to the definition of $\K_{\ep}$ in \eqref{K defn} to estimate
\[
|\K_{\ep}(x,-2q_{\ep}{t} + \delta_{\ep} + q_{\ep} + iR)| 
\le \frac{e^{x(-2q_{\ep}{t}+\delta_{\ep})}}{\tB_{\cep}(i\zpzc_{\ep}(t,R))}
\]
The numerator on the right side above is independent of $R$ and bounded in $t$.
We can infer from the quadratic estimates on $\tB_{\cep}$ in \eqref{B quadratic estimate} that the denominator is $\O(R^2)$ uniformly in $t$, and so the whole expression above vanishes as $R \to \infty$.

\subsubsection{Conclusion of the proof of Proposition \ref{b-residue}}
Since we have established the vanishing of the integrals in \eqref{vanishing curves}, we may use the residues computed in \eqref{res+ep} and \eqref{res-ep} and the strategy outlined in Section \ref{residue strategery} to conclude from \eqref{formula pre-res} that
\begin{multline*}
[\B_{\cep}^-]^{-1}[\varsigma_{\cep}g](x)
= [\B_{\cep}^+]^{-1}[\varsigma_{\cep}g](x)
+ -\left(\frac{\lt_+[\varsigma_{\cep}g](i\omega_{\cep}) + \lt_-[\varsigma_{\cep}g](i\omega_{\cep})}{(\tB_{\cep})'(\omega_{\cep})}\right)e^{i\omega_{\cep}{x}} \\
\\
+i\left(\frac{\lt_+[\varsigma_{\cep}g](-i\omega_{\cep}) + \lt_-[\varsigma_{\cep}g](-i\omega_{\cep})}{(\tB_{\cep})'(-\omega_{\cep})}\right)e^{-i\omega_{\cep}{x}}.
\end{multline*}
The estimates on the functionals $\alpha_{\ep}$ and $\beta_{\ep}$ in \eqref{alpha beta functional estimates} then follow from the estimates on the Laplace transforms in Section \ref{Laplace estimates}, particularly \eqref{ultimate Laplace estimate}.

\subsection{Proof of Proposition \ref{alpha beta prop}}\label{alpha beta prop proof appendix}

\subsubsection{Refined estimates on $\alpha_{\ep}$}\label{alpha estimates proof appendix}
The functional $\alpha_{\ep}$ was defined in \eqref{alpha defn}.
Using the expression for $g_{\ep}$ from \eqref{new g-ep}, we have
\[
\alpha_{\ep}[\M{g}_{\ep}]
= \alpha_{\ep}[\M{e}^{i\omega_{\cep}\cdot}] 
+ \alpha_{\ep}\left[\M\sum_{k=1}^{\infty} \big([\B_{\cep}^-]^{-1}\Sigma_{\cep}^*\big)^ke^{i\omega_{\cep}\cdot}\right].
\]
Since $\M{e}^{i\omega_{\cep}\cdot} = \tM(\omega_{\cep})e^{i\omega_{\cep}\cdot}$, per \eqref{M defn}, we can calculate the first term just by computing
\[
\alpha_{\ep}[e^{i\omega_{\cep}\cdot}]
= -i\tM(\omega_{\cep})\left(\frac{\lt_+[\varsigma_{\cep}e^{i\omega_{\cep}\cdot}](i\omega_{\cep})
+ \lt_-[\varsigma_{\cep}e^{i\omega_{\cep}\cdot}](i\omega_{\cep})}{(\tB_{\cep})'(\omega_{\cep})}\right).
\]
The Laplace transforms are
\[
\lt_{\pm}[\varsigma_{\cep}e^{i\omega_{\cep}\cdot}](i\omega_{\cep})
= \int_0^{\infty} e^{\mp{i\omega{s}}}\varsigma_{\cep}(s)e^{\pm{i\omega_{\cep}{s}}}\ds
= \int_0^{\infty} \varsigma_{\cep}(s) \ds
= \frac{\norm{\varsigma_{\cep}}_{L^1}}{2}.
\]
Here we needed the result from Theorem \ref{friesecke-pego} that $\varsigma_{\cep}$ is positive.
Thus
\begin{equation}\label{phi0+}
\alpha_{\ep}[\M{e}^{i\omega_{\cep}\cdot}] 
= -i\ep\left(\frac{\tM(\omega_{\cep})}{(\tB_{\cep})'(\omega_{\cep})}\right)\left(\frac{\norm{\varsigma_{\cep}}_{L^1}}{\ep}\right)
=: i\ep\theta_{\ep,0}^+,
\end{equation}
and we note that $\theta_{\ep,0}^+$ is real.
The uniform boundedness of $\omega_{\cep}$, the estimate \eqref{b-b-c ineq}, the definition of $\tM$ in \eqref{M defn}, and part \ref{FP L1-2} of Proposition \ref{hoffman-wayne proposition} tell us there are $C_1$, $C_2 > 0$ such that 
\[
0 < C_1 \le |\theta_{\ep,0}^+| \le C_2 < \infty
\]
for all $0 < \ep < \ep_{\B}$.

Now we show
\begin{equation}\label{O(ep^2) goal}
\alpha_{\ep}\left[\M\sum_{k=1}^{\infty} \big([\B_{\cep}^-]^{-1}\Sigma_{\cep}^*\big)^ke^{i\omega_{\cep}\cdot}\right]
= \O(\ep^2).
\end{equation}
This holds if there exists $C = \O(1)$ such that 
\begin{equation}\label{toward O(ep-2) est}
\big|\alpha_{\ep}[\M\big([\B_{\cep}^-]^{-1}\Sigma_{\cep}^*\big)^ke^{i\omega_{\cep}\cdot}]\big|
\le C^k\ep^{k+1}.
\end{equation}
In that case, we have
\begin{equation}\label{geometric series}
\alpha_{\ep}\left[\M\sum_{k=1}^{\infty} \big([\B_{\cep}^-]^{-1}\Sigma_{\cep}^*\big)^ke^{i\omega_{\cep}\cdot}\right]
\le \sum_{k=1}^{\infty} C^k\ep^{k+1}
= C\ep^2\sum_{k=0}^{\infty} (C\ep)^k.
\end{equation}
Then we set
\begin{equation}\label{ep-res-ep1}
\ep_{\RES,1} 
:= \min\left\{\ep_{\B},\frac{1}{2C}\right\},
\end{equation}
in which case the geometric series in \eqref{geometric series} converges for\footnote{We are, admittedly, being rather cavalier about what $C$ is. 
In Section \ref{basic notation section}, we agreed that $C$ would denote any constant that is $\O(1)$ in $\ep$; now we are restricting our range of $\ep$ based on one of these $C$.
To be more precise, we could trace the lineage of the $C$ defining $\ep_{\RES,1}$ in \eqref{ep-res-ep1} back to three sources: the estimates in Proposition \ref{B inv prop} on $[\B_{\cep}^-]^{-1}$, in \eqref{alpha beta functional estimates} on $\alpha_{\ep}$, and in Proposition \ref{hoffman-wayne proposition} on $\varsigma_{\cep}$.}
$0 < \ep < \ep_{\RES,1}$.
We conclude
\begin{equation}\label{alpha expansion}
\alpha_{\ep}\left[\M{g}_{\ep}\right]
= i\ep\theta_{\ep,0}^+ + \ep^2\theta_{\ep}^+,
\end{equation}
where $\theta_{\ep,0}^+ \in \R\setminus\{0\}$ and $\theta_{\ep}^+ = \O(1)$.

So, we only need to prove \eqref{toward O(ep-2) est}.
Fix some $h \in L^{\infty} \cup L_{-q_{\ep}}^{\infty}$.
Then
\begin{equation}\label{ind-est-1}
\norm{\varsigma_{\cep}h}_{L_{-q_{\ep}}^{\infty}}
\le C\ep^2\min\left\{\norm{h}_{L^{\infty}},\norm{h}_{L_{-q_{\ep}}^{\infty}}\right\},
\end{equation}
for if $h \in L_{-q_{\ep}}^{\infty}$, then
\[
\norm{e^{q_{\ep}\cdot}\varsigma_{\cep}h}_{L^{\infty}}
\le \norm{\varsigma_{\cep}}_{L^{\infty}}\norm{e^{q_{\ep}\cdot}h}_{L^{\infty}}
\le C\ep^2\norm{h}_{L_{-q_{\ep}}^{\infty}}
\]
by part \ref{FP Linfty-2} of Proposition \ref{hoffman-wayne proposition}.
Otherwise, if $h \in L^{\infty}$, we can use the estimate $\norm{e^{q_{\ep}\cdot}\varsigma_{\cep}}_{L^{\infty}} \le C\ep^2$ from part \ref{FP weighted Linfty} of that proposition to bound $\norm{\varsigma_{\cep}h}_{L_{-q_{\ep}}^{\infty}} \le C\ep^2\norm{h}_{L^{\infty}}$.

Next, we estimate
\begin{equation}\label{ind-est-2}
\norm{[\B_{\cep}^-]^{-1}\varsigma_{\cep}h}_{L_{-q_{\ep}}^{\infty}}
\le \norm{[\B_{\cep}^-]^{-1}\varsigma_{\cep}h}_{W_{-q_{\ep}}^{2,\infty}}
\le C\ep^{-1}\norm{\varsigma_{\cep}h}_{L_{-q_{\ep}}^{\infty}}
\le C\ep\min\left\{\norm{h}_{L^{\infty}},\norm{h}_{L_{-q_{\ep}}^{\infty}}\right\}
\end{equation}
by Proposition \ref{B inv prop} and \eqref{ind-est-1}.
Last, we use the estimate \eqref{alpha beta functional estimates} on $\alpha_{\ep}$ to bound
\begin{equation}\label{ind-est-3}
\big|\alpha_{\ep}[\M[\B_{\cep}^-]^{-1}\varsigma_{\cep}h]\big|
\le C\ep\norm{[\B_{\cep}^-]^{-1}\varsigma_{\cep}h}_{L_{-q_{\ep}}^{\infty}}
\le C\ep^2\min\left\{\norm{h}_{L^{\infty}},\norm{h}_{L_{-q_{\ep}}^{\infty}}\right\}
\end{equation}
by \eqref{ind-est-2}.

If we rewrite
\[
\M\big([\B_{\cep}^-]^{-1}\Sigma_{\cep}^*\big)^ke^{i\omega_{\cep}\cdot}
= \M[\B_{\cep}^-]^{-1}\varsigma_{\cep}\M\big([\B_{\cep}^-]^{-1}\Sigma_{\cep}^*\big)^{k-1}e^{i\omega_{\cep}\cdot}
\]
and take 
\[
h = \M\big([\B_{\cep}^-]^{-1}\Sigma_{\cep}^*\big)^{k-1}e^{i\omega_{\cep}\cdot},
\]
then we can induct on $k$ and use \eqref{ind-est-3} to obtain our desired estimate \eqref{toward O(ep-2) est}.

\subsubsection{Refined estimates on $\beta_{\ep}$}\label{beta estimates proof appendix}
The functional $\beta_{\ep}$ was defined in \eqref{beta defn}.
We have
\[
\beta_{\ep}[\M{g}_{\ep}]
= \tM_{\ep}(\omega_{\cep})\beta_{\ep}[e^{i\omega_{\cep}\cdot}] 
+ \beta_{\ep}\left[\M\sum_{k=1}^{\infty} \big([\B_{\cep}^-]^{-1}\Sigma_{\cep}^*\big)^ke^{i\omega_{\cep}\cdot}\right].
\]
The methods of Appendix \ref{alpha estimates proof appendix} show
\[
\beta_{\ep}\left[\M\sum_{k=1}^{\infty} \big([\B_{\cep}^-]^{-1}\Sigma_{\cep}^*\big)^ke^{i\omega_{\cep}\cdot}\right] = \O(\ep^2);
\]
all we have to do is replace $\alpha_{\ep}$ with $\beta_{\ep}$.
However, now we will show that $\beta_{\ep}[e^{i\omega_{\cep}\cdot}] = \O(\ep^2)$ as well.

We begin with the Laplace transforms in the definition of $\beta_{\ep}[e^{i\omega_{\cep}\cdot}]$:
\begin{multline}\label{lt+-}
\lt_+[\varsigma_{\cep}e^{i\omega\cdot}](-i\omega_{\cep})
= \int_0^{\infty} e^{-(-i\omega_{\cep})s}\varsigma_{\cep}(s)e^{i\omega_{\cep}{s}}\ds
= \int_0^{\infty} e^{2i\omega_{\cep}{s}}\varsigma_{\cep}(s) \ds \\
\\
= \bunderbrace{\int_0^{\infty} \cos(2\omega_{\cep}{s})\varsigma_{\cep}(s)\ds}{\I_{\ep,1}}
+ \bunderbrace{i\int_0^{\infty} \sin(2\omega_{\cep}{s})\varsigma_{\cep}(s) \ds}{i\I_{\ep,2}}.
\end{multline}
We decompose
\[
\I_{\ep,1}
= \bunderbrace{\frac{\ep^2}{2}\int_{-\infty}^{\infty} \cos(2\omega_{\cep}{s})\sigma(\ep{s}) \ds}{\I_{\ep,3}}
+ \bunderbrace{\int_0^{\infty} \cos(2\omega_{\cep}{s})(\varsigma_{\cep}(s)-\ep^2\sigma(\ep{s}))\ds}{\I_{\ep,4}}.
\]

Note that the evenness of the integrand allows us to integrate over all of $\R$ in $\I_{\ep,3}$.
Moreover, after changing variables with $u = \ep{s}$ in $\I_{\ep,3}$, evenness also implies
\[
\I_{\ep,3} 
= \frac{\ep}{2}\int_{-\infty}^{\infty} \cos\left(\frac{2\omega_{\cep}{u}}{\ep}\right)\sigma(u)\du
= \frac{\ep}{2}\hat{\sigma}\left(\frac{2\omega_{\cep}}{\ep}\right).
\]
Since $\sigma \in \cap_{r=1}^{\infty} H_1^r$ and $\omega_{\cep}$ is bounded uniformly in $\ep$, a variant on the Riemann-Lebesgue lemma (specifically, Lemma A.5 in \cite{faver-wright}) shows $\I_{\ep,3} = \O(\ep^2)$.
Next, we estimate directly
\begin{equation}\label{easy-peasy L1}
|\I_{\ep,4}| 
\le \norm{\varsigma_{\cep}-\ep^2\sigma(\ep\cdot)}_{L^1}
\le C\ep^2
\end{equation}
by part \ref{FP L1-1} of Proposition \ref{hoffman-wayne proposition}.

We make the same decomposition on $\I_{\ep,2}$ from \eqref{lt+-} and use the $L^1$-estimate as in \eqref{easy-peasy L1} to find
\[
|\I_{\ep,2}|
\le \left|\ep^2\int_0^{\infty} \sin(2\omega_{\cep}{s})\sigma(\ep{s}) \ds\right| + C\ep^2.
\]
An $L^1$-estimate on $\medint_0^{\infty} \sin(2\omega_{\cep}{s})\sigma(\ep{s})\ds$ will cost us a power of $\ep$ and reduce the estimate on $\I_{\ep,2}$ to only $\O(\ep)$.
We can do better by changing variables with $u = \ep{s}$ and integrating by parts to find
\begin{multline*}
\ep^2\int_0^{\infty} \sin(2\omega_{\cep}{s})\sigma(\ep{s}) \ds
= \ep\int_0^{\infty} \sin\left(\frac{2\omega_{\cep}{u}}{\ep}\right)\sigma(u) \du \\
\\
= \ep\left(-\frac{\ep}{2\omega_{\cep}}\cos\left(\frac{2\omega_{\cep}{u}}{\ep}\right)\sigma({u})\bigg|_{u=0}^{u=\infty}
+ \frac{\ep}{2\omega_{\cep}}\int_0^{\infty} \cos\left(\frac{2\omega_{\cep}{u}}{\ep}\right)\sigma(u) \du\right) \\
\\
= \frac{\ep^2\sigma(0)}{2\omega_{\cep}}+\ep^2\int_0^{\infty}\cos\left(\frac{2\omega_{\cep}{u}}{\ep}\right)\sigma(u) \du.
\end{multline*}
This is plainly $\O(\ep^2)$.

We have 
\[
\lt_-[\varsigma_{\cep}e^{i\omega_{\cep}\cdot}](-i\omega_{\cep})
\overline{\lt_+[\varsigma_{\cep}e^{i\omega_{\cep}\cdot}](i\omega_{\cep})}
 = \O(\ep^2),
\] 
and so we conclude that for some $\ep_{\RES} \le \ep_{\RES,1}$ and all $\ep \in (0,\overline{\ep})$, we have
\begin{equation}\label{beta expansion}
\beta_{\ep}[\M{g}_{\ep}]
= \ep^2\theta_{\ep}^-,
\end{equation}
where $\theta_{\ep}^- = \O(1)$.
\section{Proofs Involved in the Construction of the Nonlocal Solitary Wave Problem}\label{nanopteron proofs appendix}

\subsection{Proof of Proposition \ref{friesecke-pego bootstrap}}\label{friesecke-pego bootstrap appendix}
We simply rearrange \eqref{what sigma c does} into
\begin{equation}\label{basic boot}
\varsigma_c''
= \frac{(A-2)(\varsigma_c+\varsigma_c^2)}{c^2}.
\end{equation}
The operator $(A-2)/c^2$ maps $E_q^r$ into $E_q^r$ for any $q$ and $r$.
Since $\varsigma_c \in E_{q_{\varsigma}(c)}^2$, we also have $\varsigma_c^2 \in E_{q_{\varsigma}(c)}^2$, and so \eqref{basic boot} implies $\varsigma_c'' \in E_{q_{\varsigma}(c)}^2$, i.e., $\varsigma_c \in E_{q_{\varsigma}(c)}^4$.
We continue to bootstrap with \eqref{basic boot} to obtain $\varsigma_c \in \cap_{r=1}^{\infty} E_{q_{\varsigma}(c)}^r$.

\subsection{Proof of Proposition \ref{H-c invert general}}\label{H-c invert general appendix}
First we study an auxiliary Fourier multiplier.
Let
\[
\tM_c(z)
:= -c^2z^2+2-2\cos(z).
\]
Clearly $\tM_c$ is entire and one can check that $\tM_c$ has a double zero at $z=0$ and no other zeros on $\R$.
Next, as in the proof of part \ref{B quadratic decay} of Proposition \ref{symbol of B}, one can show the existence of $C(c)$, $z_0(c) > 0$ such that if $q_{\H}(c) \le |\im(z)| \le \min\{q_{\varsigma}(c),1\}$ and $|z| \ge z_0(c)$, then
\[
C(c)|\re(z)|^2 
\le |\tM_c(z)|.
\]
Lemma \ref{beale fm} then provides $q_{\H}^{\star}(c)$, $q_{\H}^{\star\star}(c)$ such that $q_{\H}(c) < q_{\H}^{\star}(c) < q_{\H}^{\star\star}(c) < \min\{q_{\varsigma}(c),1\}$ such that for $q \in [q_{\H}^{\star}(c),q_{\H}^{\star\star}(c)]$ and $r \ge 0$, the Fourier multiplier $\M_c$ with symbol $\tM_c$ is invertible from $E_q^{r+2}$ to $E_{q,0}^r$.

Now we are ready to prove the proposition.
Fix $q \in [q_{\H}^{\star}(c),q_{\H}^{\star\star}(c)]$ and $r \ge 0$.
Since $q \ge q_{\H}^{\star}(c) > q_{\H}(c)$, we have $E_q^{r+2} \subseteq E_{q_{\H}(c)}^2$, and so $\H_c$ is injective from $E_q^{r+2}$ to $E_{q,0}^r$.
For surjectivity, take $g \in E_{q,0}^r$.
Then $g \in E_{q_{\H}(c),0}^0$, so there exists $f \in E_{q_{\H}(c)}^2$ such that $\H_cf = g$. 

Next, we need to show that $f$ that $f \in E_q^{r+2}$.
We rearrange the equality $\H_cf = g$ into
\[
\M_cf
= 2(A-2)\varsigma_cf + g.
\]
Since $\varsigma_c \in E_{q_{\varsigma}(c)}^2$ and $f \in E_{q_{\H}(c)}^2$, and since $q \le q_{\H}^{\star\star} < q_{\varsigma}(c)$, we have $\varsigma_cf \in E_{q_{\varsigma}(c) + q_{\H}(c)}^2 \subseteq E_q^2$.
Then $2(A-2)\varsigma_cf + g \in E_{q,0}^0$, and so 
\[
f 
= \M_c^{-1}[2(A-2)\varsigma_cf + g] 
\in E_q^2.
\]
If $r = 0$, then we are done; otherwise, we rearrange the equality $\H_cf = g$ into the different form
\[
f''
= \frac{1}{c^2}(A-2)(1+2\varsigma_c)f + g,
\]
and we may bootstrap from this equality until we achieve $f \in E_q^{r+2}$.

\subsection{Proof of the formula \ref{iota-chi formula}}\label{iota-chi formula proof appendix}
We use the definition of $\chi_c$ in \eqref{chi-c defn}, the definitions of $\B_c$ and $\Sigma_c$ in \eqref{B-c Sigma-c defns}, the definition of $\Sigma_c^*$ in \eqref{L-c adjoint defn}, the fundamental property $(\B_c-\Sigma_c^*)\gamma_c = 0$, and the evenness of $\B_c\gamma_c\sin(\omega_c\cdot)$ to calculate
\[
\iota_c[\chi_c]
= 2\lim_{R \to \infty} \int_0^R \big(c^2\gamma_c''(x) + (2+A)\gamma_c(x)\big)\sin(\omega_cx) \dx.
\]

Now we integrate by parts.
First, we use the formula (as stated in Appendix D of \cite{hoffman-wright})
\[
\int_0^R f''(x)g(x) \dx = f'(R)g(R)-f(R)g'(R)+\int_0^R f(x)g''(x) \dx,
\]
valid for $f$ and $g$ odd, to rewrite
\begin{multline}\label{ibp1}
2c^2\int_0^R \gamma_c''(x)\sin(\omega_cx) \dx
= 2c^2\big(\gamma_c'(R)\sin(\omega_cR)-\omega_c\gamma_c(R)\cos(\omega_cR)\big) \\
-2c^2\omega_c^2\int_0^R \gamma_c(x)\sin(\omega_cx)\dx.
\end{multline}
Observe
\begin{equation}\label{ibp2}
-\omega_c^2\int_0^R \gamma_c(x) \sin(\omega_cx) \dx
= \int_0^R \gamma_c(x)[\partial_x^2\sin(\omega_c\cdot)](x) \dx.
\end{equation}
We have another formula, easily established through direct calculation and $u$-substitution:
\begin{multline}\label{adjoint 0 R A formula}
\int_0^R (Af)(x)g(x) \dx
=\int_0^R f(x)(Ag)(x) \dx 
+ \frac{1}{2}\int_R^{R+1} f(x)(S^{-1}g)(x) \dx
- \frac{1}{2}\int_{R-1}^R f(x)(S^1g)(x) \dx \\
\\
+ \frac{1}{2}\int_{-1}^0 f(x)(S^1g)(x) \dx 
- \frac{1}{2}\int_0^1f(x)(S^{-1})g(x) \dx.
\end{multline}
We combine \eqref{ibp1} and \eqref{ibp2}, apply \eqref{adjoint 0 R A formula}, and use $\B_c^* = \B_c$ to find
\begin{equation}\label{I-c-3 limits}
\begin{aligned}
\iota_c[\chi_c]
&=2c^2\lim_{R \to \infty} \big(\gamma_c'(R)\sin(\omega_cR)-\omega_c\gamma_c(R)\cos(\omega_cR)\big) \\
\\
&+2\lim_{R \to \infty} \int_0^R \gamma_c(x)(\B_c\sin(\omega_c\cdot))(x)\dx \\
\\
&+ \left(\int_{-1}^0 \gamma_c(x)\sin(\omega_c(x+1))\dx- \int_0^1 \gamma_c(x)\sin(\omega_c(x-1))\dx \right) \\
\\
&+ \lim_{R \to \infty} \left(\int_R^{R+1} \gamma_c(x)\sin(\omega_c(x-1))\dx - \int_{R-1}^R \gamma_c(x)\sin(\omega_c(x+1))\dx\right).
\end{aligned}
\end{equation}

We can evaluate exactly each of the four terms above.
Trigonometric identities and the asymptotics of $\gamma_c$ from \eqref{gamma asymptotics} give 
\[
\lim_{R \to \infty} \big(\gamma_c'(R)\sin(\omega_cR)-\omega_c\gamma_c(R)\cos(\omega_cR)\big)
= \omega_c\sin(\omega_c\vartheta_c).
\]
Next, since $\B_ce^{\pm{i}\omega_c\cdot} = 0$,  we have
\[
\lim_{R \to \infty} \int_0^R \gamma_c(x)(\B_c\sin(\omega_c\cdot))(x)\dx 
= 0,
\]
and so the second term is zero.
For the third term, we rewrite the integrals using the addition formula for sine and then use the evenness of $\gamma_c\sin(\omega_c\cdot)$ and the oddness of $\gamma_c\cos(\omega_c\cdot)$ to show that the resulting integrals all add up to zero.

Last, we rewrite the first integral in the fourth term as
\begin{multline*}
\int_R^{R+1} \gamma_c(x)\sin(\omega_c(x-1))\dx
= \int_R^{R+1} \big(\gamma_c(x)-\sin(\omega_c(x+\vartheta_c))\big)\sin(\omega_c(x-1)) \dx \\
\\
+ \int_R^{R+1} \sin(\omega_c(x+\vartheta_c))\sin(\omega_c(x-1))\dx,
\end{multline*}
where
\[
\left|\int_R^{R+1} \big(\gamma_c(x)-\sin(\omega_c(x+\vartheta_c))\big)\sin(\omega_c(x-1)) \dx\right|
\le \max_{R \le x \le R+1} |\gamma_c(x) - \sin(\omega_c(x+\vartheta_c))| \to 0
\]
as $R \to \infty$ by the asymptotics of $\gamma_c$ in \eqref{gamma asymptotics}.
Similar manipulations in the other integral in the fourth term and a host of trig identities then yield
\begin{multline*}
\lim_{R \to \infty} \left(\int_R^{R+1} \gamma_c(x)\sin(\omega_c(x-1))\dx - \int_{R-1}^R \gamma_c(x)\sin(\omega_c(x+1))\dx\right) \\
\\
= \lim_{R \to \infty} \left(\int_R^{R+1} \sin(\omega_c(x+\vartheta_c))\sin(\omega_c(x-1)) \dx
- \int_{R-1}^R \sin(\omega_c(x+\vartheta_c))\sin(\omega_c(x+1)) \dx\right) \\
\\
= -\sin(\omega_c\vartheta_c)\sin(\omega_c).
\end{multline*}

All together, we have
\[
\iota_c[\chi_c]
= \big(2c^2\omega_c-\sin(\omega_c)\big)\sin(\omega_c\vartheta_c),
\]
and this is exactly the formula \eqref{iota-chi formula}.
The preceding calculation of $\iota_c[\chi_c]$ is similar to the work in Appendix D of \cite{hoffman-wright}, except there the small parameter $\mu$ appeared throughout the calculations as well, which allowed Hoffman and Wright to ignore some terms analogous to those in \eqref{I-c-3 limits}.
\section{Proof of Proposition \ref{main workhorse proposition}}\label{all the nonlocal sol wave proofs}

\subsection{General estimates in $H_q^r$}

We begin with a collection of very useful estimates that we will invoke throughout the proof of Proposition \ref{main workhorse proposition}.
Most of these estimates follow from straightforward calculus and occasional recourses to the norms
\[
f \mapsto \norm{f}_{L^2} + \norm{\cosh(q\cdot)\partial_x^r[f]}_{L^2},
\quadword{and}
f \mapsto \norm{f}_{L^2} + \norm{\cosh^q(\cdot)\partial_x^r[f]}_{L^2},
\]
which are equivalent on $H_q^r$ to the norm defined in \eqref{hrq}.
Further details of the proof are given in Appendix C.3.3 of \cite{faver-dissertation}.

\begin{proposition}\label{general Hrq estimates prop}

\begin{enumerate}[label={\bf(\roman*)},ref={(\roman*)}]

\item
If $f \in H_q^r$ and $g \in W^{r,\infty}$, then 
\begin{equation}\label{Hrq-Wrinfty general est}
\norm{fg}_{r,q}
\le \norm{f}_{r,q}\norm{g}_{W^{r,\infty}}.
\end{equation}

\item
If $f$, $g \in H_q^r$, then
\begin{equation}\label{Hrq algebra est}
\norm{fg}_{r,q}
\le \norm{\sech(q\cdot)}_{W^{r,\infty}}\norm{f}_{r,q}\norm{g}_{r,q}.
\end{equation}

\item
If $f \in W^{r,\infty}$ and $\omega \in \R$, then 
\begin{equation}\label{Wrinfty scaling est}
\norm{f(\omega\cdot)}_{W^{r,\infty}}
\le \left(\max_{0 \le k \le r} |\omega|^k\right)\norm{f}_{W^{r,\infty}}.
\end{equation}

\item\label{decay borrowing}
If $f \in H_{q_2}^r$; $g \in W^{r,\infty}$; $\omega$, $\grave{\omega} \in \R$; and $0 < q_1 < q_2$, then
\begin{equation}\label{decay borrowing est}
\norm{f \cdot (g(\omega\cdot)-g(\grave{\omega}\cdot))}_{r,q_1}
\le C(r,q_2-q_1)\left(\max_{0 \le k \le r} \Lip(\partial_x^k[g])\right)\norm{f}_{r,q_2}\norm{g}_{W^{r,\infty}}|\omega-\grave{\omega}|,
\end{equation}
where the Lipschitz constant $\Lip(\cdot)$ was defined in \eqref{Lip defn}.

\item\label{decay borrowing BIG}
If $f \in H_{q_2}^r$; $g$, $\grave{g} \in W^{r,\infty}$; $\omega$, $\grave{\omega} \in \R$; and $0 < q_1 < q_2$, then
\begin{multline}\label{outer inner Lipschitz}
\norm{f \cdot (g(\omega\cdot)-\grave{g}(\grave{\omega}\cdot))}_{r,q_1} \\
\\
\le C(r,q_2-q_1)\norm{f}_{r,q_2}\left[\left(\max_{0 \le k \le r} \Lip(\partial_x^k[g])\right)\norm{g}_{W^{r,\infty}}|\omega-\grave{\omega}|
+ \left(\max_{0 \le k \le r} |\grave{\omega}|^k\right)\norm{g-\grave{g}}_{W^{r,\infty}}\right] .
\end{multline}
\end{enumerate}
\end{proposition}

We refer to parts \ref{decay borrowing} and \ref{decay borrowing BIG} as ``decay borrowing'' estimates, as they permit us to achieve a Lipschitz estimate on a product starting in a space of lower decay by ``borrowing'' from the decay rates of one of the faster-decaying functions in the product.  
A version of part \ref{decay borrowing} in particular was stated and proved as Lemma A.2 in \cite{faver-wright}.
From this proposition we immediately deduce two estimates for our quadratic nonlinearity $\nl$ from \eqref{nl}.

\begin{lemma}\label{estimates for Q}

\begin{enumerate}[label={\bf(\roman*)}]

\item
If $\rhob$, $\grave{\rhob} \in H_q^r \times H_q^r$, then
\begin{equation}\label{Hrq product estimate for Q}
\norm{\nl(\rhob,\grave{\rhob})}_{r,q}
\le C(r,q)\norm{\rhob}_{r,q}\norm{\grave{\rhob}}_{r,q}
\end{equation}

\item
If $\rhob \in H_q^r \times H_q^r$ and $\phib \in W^{r,\infty} \times W^{r,\infty}$, then
\begin{equation}\label{Hrq-Wrinfty estimate for Q}
\norm{\nl(\rhob,\phib)}_{r,q}
\le C(r,q)\norm{\rhob}_{r,q}\norm{\phib}_{W^{r,\infty} \times W^{r,\infty}}.
\end{equation}
\end{enumerate}
\end{lemma}

Proposition \ref{general Hrq estimates prop} also allows us to relate $\chi_c^{\mu}$ from \eqref{chi-c-mu defn} and $\chi_c$ from \eqref{chi-c defn}.

\begin{lemma}\label{chi-c-mu minus chi-c est}
For any $q \in (0,q_{\varsigma}(c))$ and $r \ge 0$, there is a constant $C(c,q,r) > 0$ such that 
\[
\norm{\chi_c^{\mu}-\chi_c}_{r,q}
\le C(c,q,r)|\mu|
\]
for all $|\mu| \le \mu_{\per}(c)$.
\end{lemma}

\begin{proof}
Recall that
\[
\chi_c^{\mu} 
= 2\D_{\mu}\nl(\varsigmab_c,\phib_c^{\mu}[0])\cdot\e_2
\quadword{and}
\chi_c 
= 2\D_0\nl(\varsigmab_c,\phib_c^0[0])\cdot\e_2.
\]
Then
\[
\chi_c^{\mu}-\chi_c
= \bunderbrace{(\D_{\mu}-\D_0)\nl(\varsigmab_c,\phib_c^{\mu}[0])\cdot\e_2}{I}
 + \bunderbrace{\D_0\nl(\varsigmab_c,\phib_c^{\mu}[0]-\phib_c^0[0])\cdot\e_2}{II}.
\]
From the definition of $\D_{\mu}$ in \eqref{D-mu defn}, the uniform bounds on the periodic solutions from \eqref{periodic theorem bounded}, and the product estimate \eqref{Hrq-Wrinfty estimate for Q}, we estimate that if $0 < q < q_{\varsigma}(c)$, then
\[
\norm{I}_{r,q}
\le C(c,r)|\mu|.
\]
Next, we have
\[
\phib_c^{\mu}[0]
= \begin{pmatrix*}
\upsilon_c^{\mu}\cos(\omega_c^{\mu}\cdot) \\
\sin(\omega_c^{\mu}\cdot)
\end{pmatrix*}
\quadword{and}
\phib_c^0[0]
= \begin{pmatrix*}
0 \\
\sin(\omega_c\cdot)
\end{pmatrix*},
\]
per \eqref{phib-c-mu-a expansion} and the estimate $\upsilon_c^{\mu} = \O_c(\mu)$.
We use this to calculate
\[
II
= (2+A)(\varsigma_c\cdot(\sin(\omega_c^{\mu}\cdot)-\sin(\omega_c\cdot))).
\]
Then we use the decay-borrowing estimate \eqref{decay borrowing est} with the condition $0 < q < q_{\varsigma}(c)$ as well as part \ref{omega-c-mu minus omega-c} of Proposition \ref{critical frequency props prop}  to find
\[
\norm{II}_{r,q}
\le C(c,q,r)|\omega_c^{\mu}-\omega_c|
\le C(c,q,r)|\mu|.
\qedhere
\]
\end{proof}

\subsection{Proof of Proposition \ref{main workhorse proposition}}\label{proof of main workhorse proposition appendix}
We rely on the following lemma, which we prove in the next three sections.
We inherit the general techniques from the myriad nanopteron estimates in \cite{faver-wright}, \cite{hoffman-wright}, \cite{faver-spring-dimer}, and \cite{johnson-wright}.

\begin{lemma}\label{workhorse lemma for workhorse prop}
For all $r \ge 1$, there are constants $C(c,r)$, $C(c) > 0$ such that the following estimates hold for any $|\mu| \le \mu_{\per}(c)$.

\begin{enumerate}[label={\bf(\roman*)},ref={(\roman*)}]

\item\label{workhorse map part}
If $\etab \in E_{q_{\star}(c)}^r \times O_{q_{\star}(c)}^r$ and $|a| \le a_{\per}(c)$, then
\begin{equation}\label{workhorse map est}
\norm{\nanob_c^{\mu}(\etab,a)}_{\X^r} 
\le C(c,r)\big(|\mu| + |\mu|\norm{\etab}_{r,q_{\star}(c)} + |\mu||a| + |a|\norm{\etab}_{r,q_{\star}(c)} + \norm{\etab}_{r,q_{\star}(c)}^2 + a^2\big).
\end{equation}

\item\label{workhorse lip part}
If $\etab$, $\grave{\etab} \in E_{q_{\star}(c)}^1 \times O_{q_{\star}(c)}^1$ and $|a|$, $|\grave{a}| \le a_{\per}(c)$, then 
\begin{equation}\label{workhorse lip est}
\norm{\nanob_c^{\mu}(\etab,a)-\nanob_c^{\mu}(\grave{\etab},\grave{a})}_{\X^0}
\le C(c)\big(|\mu| + |a| + |\grave{a}| + \norm{\etab}_{1,q_{\star}(c)} + \norm{\grave{\etab}}_{1,q_{\star}(c)})(\norm{\etab-\grave{\etab}}_{1,\qbar_{\star}(c)} + |a-\grave{a}|\big).
\end{equation}

\item
If $\etab \in E_{q_{\star}(c)}^r \times O_{q_{\star}(c)}^r$ and $|a| \le a_{\per}(c)$, then
\begin{equation}\label{workhorse boot est}
\norm{\nanob_c^{\mu}(\etab,a)}_{\X^{r+1}} 
\le C(c,r)\big(|\mu| + |\mu|\norm{\etab}_{r,q_{\star}(c)} + |\mu||a| + |a|\norm{\etab}_{r,q_{\star}(c)} + \norm{\etab}_{r,q_{\star}(c)}^2 + a^2\big).
\end{equation}
\end{enumerate}
\end{lemma}

Now we prove Proposition \ref{main workhorse proposition}.
Set
\[
\tau_c
= 6C(c,1)
\]
with $C(c,1)$ from part \ref{workhorse map part} of Lemma \ref{workhorse lemma for workhorse prop}, and let
\[
\mu_{\star}(c)
:= \min\left\{1,\frac{1}{\tau_c},\frac{1}{\tau_c^2}, \mu_{\per}(c),\frac{a_{\per}(c)}{\tau_c},\frac{1}{2C(c)(1+4\tau_c)}\right\},
\]
with $C(c)$ from part \ref{workhorse lip part} of the same lemma.
Note in particular that if $|\mu| \le \mu_{\star}(c)$ and $(\etab,a) \in \U_{\tau,\mu}^r$, then 
\[
|a|
\le \norm{\etab}_{r,q_{\star}(c)} + |a|
\le \tau_c|\mu|
\le a_{\per}(c), 
\]
and so $\nanob_c^{\mu}(\etab,a)$ is well-defined.

\begin{enumerate}[label=,labelsep=0pt]

\item
{\it{Proof of \ref{main workhorse map part}.}}
Assume $|\mu| \le \mu_{\star}(c)$ and $(\etab,a) \in \U_{\tau_c,\mu}^1$.
We estimate from \eqref{workhorse map est} with $r=1$ that
\[
\norm{\nanob_c^{\mu}(\etab,a)}_{1,q_{\star}(c)}
\le C(c,1)(1 + 3\tau_c|\mu| + 2\tau_c^2|\mu|)|\mu|
\le 6C(c,1)|\mu|
= \tau_c|\mu|.
\]
That is, $\nanob_c^{\mu}(\etab,a) \in \U_{\tau_c,\mu}^1$.

\item
{\it{Proof of \ref{main workhorse lip part}.}}
Next, let $(\etab,a)$, $(\grave{\etab},\grave{a}) \in \U_{\tau_c,\mu}^1$.
We estimate from \eqref{workhorse lip est} that 
\[
\norm{\nanob_c^{\mu}(\etab,a)-\nanob_c^{\mu}(\grave{\etab},\grave{a})}_{1,\qbar_{\star}(c)}
\le C(c)|\mu|(1+4\tau_c)\big(\norm{\etab-\grave{\etab}}_{1,\qbar_{\star}(c)} + |a-\grave{a}|\big)
< \frac{1}{2}\big(\norm{\etab-\grave{\etab}}_{1,\qbar_{\star}(c)} + |a-\grave{a}|\big).
\]

\item
{\it{Proof of \ref{main workhorse boot part}.}}
Finally, let $(\etab,a) \in \U_{\tau_c,\mu}^1 \cap \U_{\tau,\mu}^r$, where $\tau > 0$ is arbitrary.
Then \eqref{workhorse boot est} implies
\[
\norm{\nanob_c^{\mu}(\etab,a)}_{r,q_{\star}(c)} 
\le C(c,r)(1 + 3\tau|\mu| + 2\tau^2|\mu|)|\mu|
\le C(c,r)(1+3\tau+2\tau^2)|\mu|.
\]
With 
\[
\ttau
:= C(c,r)(1+3\tau+2\tau^2),
\]
we see that $\nanob_c^{\mu}(\etab,a) \in \U_{\ttau,\mu}^{r+1}$.
\end{enumerate}

\subsection{Mapping estimates}\label{mapping estimates appendix}
In this appendix we prove part \ref{workhorse map part} of Lemma \ref{workhorse lemma for workhorse prop}.
The definitions of $\nano_{c,1}^{\mu}$ in \eqref{N1 defn}, $\nano_{c,2}^{\mu}$ in \eqref{N2 defn}, and $\nano_{c,3}^{\mu}$ in \eqref{A defn} imply
\[
\norm{\nano_{c,1}^{\mu}(\etab,a)}_{r,q_{\star}(c)}
\le \norm{\H_c^{-1}}_{\b(E_{q_{\star}(c),0}^r,E_{q_{\star}(c)}^{r+2})}\sum_{k=1}^5 \norm{h_{c,k}^{\mu}(\etab,a)}_{r,q_{\star}(c)},
\]
\[
\norm{\nano_{c,2}^{\mu}(\etab,a)}_{r,q_{\star}(c)}
\le \norm{\L_c^{-1}\P_c}_{\b(O_{q_{\star}(c)}^r,O_{q_{\star}(c)}^{r+2})}\sum_{k=1}^5\norm{\tell_{c,k}^{\mu}(\etab,a)}_{r,q_{\star}(c)},
\]
and
\[
|\nano_{c,3}^{\mu}(\etab,a)|
\le \frac{1}{|\iota_c[\chi_c]|}\norm{\iota_c}_{(O_q^r)^*}\sum_{k=1}^5 \norm{\tell_{c,k}^{\mu}(\etab,a)}_{r,q_{\star}(c)}.
\]
Then to obtain the estimate \eqref{workhorse map est}, it suffices to find bounds of the form
\[
\norm{h_{c,k}^{\mu}(\etab,a)}_{r,q_{\star}(c)}
+ \norm{\tell_{c,k}^{\mu}(\etab,a)}_{r,q_{\star}(c)}
\le C(c,r)\rhs_{\map}(\norm{\etab}_{r,q_{\star}(c)},a,\mu),\ k = 1,\ldots,6,
\]
where for $a$, $\mu$, $\rho \in \R$, we define
\[
\rhs_{\map}(\rho,a,\mu)
:= |\mu| + |\mu\rho| + |\mu{a}| + |a\rho| + a^2 + \rho^2.
\]
We do this in the following sections; throughout, we recall that these $h_{c,k}^{\mu}$ and $\tell_{c,k}^{\mu}$ terms were defined in \eqref{h ell defns} and \eqref{tells}.

\subsubsection{Mapping estimates for $h_{c,1}^{\mu}$ and $\tell_{c,1}^{\mu}$}
Since $\varsigma_c \in \cap_{r=0}^{\infty} E_{q_{\varsigma}(c)}^r$, we have
\[
\norm{h_{c,1}^{\mu}(\etab,a)}_{r,q_{\star}(c)}
+ \norm{\tell_{c,1}^{\mu}(\etab,a)}_{r,q_{\star}(c)}
= |\mu|\norm{\Dring(\varsigmab_c+\nl(\varsigmab_c,\varsigmab_c)}_{r,q_{\star}(c)}
= \O_c(\mu).
\]

\subsubsection{Mapping estimates for $h_{c,2}^{\mu}$ and $\tell_{c,2}^{\mu}$}
We use \eqref{Hrq algebra est} to estimate 
\[
\norm{h_{c,1}^{\mu}(\etab,a)}_{r,q_{\star}(c)}
\le C(c,r)|\mu|\norm{\etab}_{r,q_{\star}(c)} + C(c,r)|\mu|\norm{\varsigma_c\eta_1}_{r,q_{\star}(c)} 
\le C(c,r)|\mu|\norm{\etab}_{r,q_{\star}(c)}.
\]

\subsubsection{Mapping estimates for $h_{c,3}^{\mu}$}\label{h-c-3-mu mapping appendix}
The estimates for this term and its counterpart $\tell_{c,3}^{\mu}$ are probably the most intricate of all the mapping estimates.
First, we compute
\begin{equation}\label{h-c-3-mu mapping decomp}
h_{c,3}^{\mu}(\etab,a)
= -2a\D_{\mu}\nl(\varsigmab_c,\phib_c^{\mu}[a])\cdot\e_1
= \bunderbrace{-a(2+\mu)(2-A)(\varsigma_c\phi_{c,1}^{\mu}[a])}{aI} - \bunderbrace{\mu{a}\delta(\varsigma_c\phi_{c,2}^{\mu}[a])}{aII}.
\end{equation}
The bound \eqref{periodic theorem bounded} on $\norm{\phib_c^{\mu}[a]}_{W^{r,\infty}}$ from Proposition \ref{periodic solutions theorem} and the product estimate \eqref{Hrq-Wrinfty general est} from Proposition \ref{general Hrq estimates prop} tell us
\begin{equation}\label{h-c-3-mu mapping II}
\norm{II}_{r,q_{\star}(c)}
\le C(c,r)|\mu|.
\end{equation}

Now we use the expansion \eqref{phib-c-mu-a expansion} in Proposition \ref{periodic solutions theorem} to write
\[
\phi_{c,1}^{\mu}[a]
= \upsilon_c^{\mu}\cos(\omega_c^{\mu}[a]\cdot) + \psi_{c,1}^{\mu}[a](\omega_c^{\mu}[a]\cdot),
\]
We then use the product estimate \eqref{Hrq-Wrinfty general est} and this expansion to bound
\begin{align*}
\norm{I}_{r,q_{\star}(c)}
&\le C(c,r)\norm{\varsigma_c\phi_{c,1}^{\mu}[a]}_{r,q_{\star}(c)} \\
\\
&\le C(c,r)\norm{\upsilon_c^{\mu}\cos(\omega_c^{\mu}[a]\cdot) + \psi_{c,1}^{\mu}[a](\omega_c^{\mu}[a]\cdot)}_{W^{r,\infty}} \\
\\
&\le C(c,r)|\upsilon_c^{\mu}|\norm{\cos(\omega_c^{\mu}[a]\cdot)}_{W^{r,\infty}}
+ C(c,r)\norm{\psi_{c,1}^{\mu}[a](\omega_c^{\mu}[a]\cdot)}_{W^{r,\infty}}.
\end{align*}
The estimate $\upsilon_c^{\mu} = \O_c(\mu)$ and the bounds on $|\omega_c^{\mu}[a]|$ from \eqref{periodic theorem bounded} give
\[
|\upsilon_c^{\mu}|\norm{\cos(\omega_c^{\mu}[a]\cdot)}_{W^{r,\infty}}
\le C(c,r)|\mu|.
\]

Last, from the uniform bound \eqref{naked psib-c-mu-a est} on $\psib_c^{\mu}$ and the scaling estimate \eqref{Wrinfty scaling est}, we have
\[
\norm{\psi_{c,1}^{\mu}[a](\omega_c^{\mu}[a]\cdot)}_{W^{r,\infty}}
\le C(c,r)|a|
\]
and thus
\begin{equation}\label{h-c-3-mu mapping I}
\norm{I}_{r,q_{\star}(c)}
\le C(c,r)(|\mu| + |a|).
\end{equation}
We conclude
\[
\norm{h_{c,3}^{\mu}(\etab,a)}_{r,q_{\star}(c)}
\le |a|\big(\norm{I}_{r,q_{\star}(c)} + \norm{II}_{r,q_{\star}(c)}\big)
\le C(c,r)(|a\mu| + a^2).
\]

\subsubsection{Mapping estimates for $\tell_{c,3}^{\mu}$}\label{tell-c-3-mu mapping appendix}
Recall the definitions of $\tell_{c,3}^{\mu}$ in \eqref{tells} and of $\chi_c$ from \eqref{chi-c defn} and $\chi_c^{\mu}$ from \eqref{chi-c-mu defn}.
Then
\begin{align*}
\tell_{c,3}^{\mu}(\etab,a)
&= -2a\D_{\mu}\nl(\varsigmab_c,\phib_c^{\mu}[a])\cdot\e_2 +a\chi_c \\
\\
&= -2a\D_{\mu}\nl(\varsigmab_c,\phib_c^{\mu}[a])\cdot\e_2 + a\chi_c^{\mu} + a(\chi_c-\chi_c^{\mu}) \\
\\
&= -2a\D_{\mu}\nl(\varsigmab_c,\phib_c^{\mu}[a]-\phib_c^{\mu}[0])\cdot\e_2 + a(\chi_c-\chi_c^{\mu}) \\
\\
&= \bunderbrace{\mu{a}\delta(\varsigma_c\cdot\big(\phi_{c,1}^{\mu}[a]-\phi_{c,2}^{\mu}[0])\big)}{\mu{a}I}
- \bunderbrace{a(2+\mu)(2+A)\big(\sigma_c\cdot(\phi_{c,2}^{\mu}[a]-\phi_{c,2}^{\mu}[0])\big)}{aII}
+ \bunderbrace{a(\chi_c-\chi_c^{\mu})}{aIII}.
\end{align*}
We know from Lemma \ref{chi-c-mu minus chi-c est} that $\norm{III}_{r,q_{\star}(c)} \le C(c,q_{\star}(c),r)|\mu|$, so we only need to estimate the terms $I$ and $II$.

We can use the triangle inequality and the periodic bounds \eqref{periodic theorem bounded} to obtain the crude estimate
\[
|\mu{a}|\norm{I}_{r,q_{\star}(c)}
\le C(c,r)|\mu{a}|,
\] 
but we need to do more work with $II$.
From the periodic expansion \eqref{phib-c-mu-a expansion} and part \ref{psib-c-mu-0=0} of Proposition \ref{periodic solutions theorem}, we calculate
\[
\phi_{c,2}^{\mu}[a]-\phi_{c,2}^{\mu}[0]
= \big(\sin(\omega_c^{\mu}[a]\cdot)-\sin(\omega_c^{\mu}[0]\cdot)\big)
+ \psi_{c,2}^{\mu}[a](\omega_c^{\mu}[a]\cdot).
\]
Then
\[
\norm{II}_{r,q_{\star}(c)}
\le C(c,r)\bunderbrace{\norm{\varsigma_c\cdot\big(\sin(\omega_c^{\mu}[a]\cdot)-\sin(\omega_c^{\mu}[0]\cdot)\big)}_{r,q_{\star}(c)}}{IV}
+ \bunderbrace{\norm{\varsigma_c\psi_{c,2}^{\mu}[a](\omega_c^{\mu}[a]\cdot)}_{r,q_{\star}(c)}}{V}.
\]
We use \eqref{periodic theorem bounded}, \eqref{naked psib-c-mu-a est}, and \eqref{Wrinfty scaling est} to estimate
\[
V
\le C(c,r)|a|.
\]

To estimate $IV$, we exploit the condition $\varsigma_c \in \cap_{r=0}^{\infty} (E_{q_{\varsigma}(c)}^r \cap E_{q_{\star}(c)}^r)$ with $q_{\star}(c) < q_{\varsigma}(c)$.
Since this inequality is strict, we may call on the decay borrowing estimate \eqref{decay borrowing est} to find
\[
\norm{IV}_{r,q_{\star}(c)}
\le C(c,r)\norm{\varsigma_c}_{r,q_{\varsigma}(c)}|\omega_c^{\mu}[a]-\omega_c^{\mu}[0]|.
\]
Then we use \eqref{periodic theorem lipschitz} to conclude
\[
\norm{IV}_{r,q_{\star}(c)}
\le C(c,r)|a|.
\]
We put all these estimates together to find
\[
\norm{\tell_{c,3}^{\mu}(\etab,a)}_{r,q_{\star}(c)}
\le C(c,r)(|\mu{a}| + a^2).
\]

\subsubsection{Mapping estimates for $h_{c,4}^{\mu}$ and $\tell_{c,4}^{\mu}$}
These terms are roughly quadratic of the form $a\etab$.
We first use the estimate \eqref{Hrq-Wrinfty estimate for Q} in Lemma \ref{estimates for Q} to find
\begin{multline*}
\norm{h_{c,4}^{\mu}(\etab,a)}_{r,q_{\star}(c)}
+ \norm{\tell_{c,4}^{\mu}(\etab,a)}_{r,q_{\star}(c)}
\le 2|a|\norm{\D_{\mu}\nl(\phib_c^{\mu}[a],\etab)}_{r,q_{\star}(c)} \\
\\
\le C(c,r)|a|\norm{\phib_c^{\mu}[a]}_{W^{r,\infty}}\norm{\etab}_{r,q_{\star}(c)}.
\end{multline*}
Next, we use the estimate \eqref{periodic theorem bounded} on $\norm{\phib_c^{\mu}[a]}_{W^{r,\infty}}$ to conclude
\[
\norm{h_{c,4}^{\mu}(\etab,a)}_{r,q_{\star}(c)}
+ \norm{\tell_{c,4}^{\mu}(\etab,a)}_{r,q_{\star}(c)}
\le C(c,r)|a|\norm{\etab}_{r,q_{\star}(c)}.
\]

\subsubsection{Mapping estimates for $h_{c,5}^{\mu}$ and $\tell_{c,5}^{\mu}$}
These terms are both quadratic in $\etab$, and so we estimate simultaneously
\[
\norm{h_{c,5}^{\mu}(\etab,a)}_{r,q_{\star}(c)}
+ \norm{\tell_{c,5}^{\mu}(\etab,a)}_{r,q_{\star}(c)}
\le \norm{\D_{\mu}\nl(\etab,\etab)}_{r,q_{\star}(c)}
\le C(c,r)\norm{\etab}_{r,q_{\star}(c)}^2
\]
by \eqref{Hrq-Wrinfty estimate for Q} in Lemma \ref{estimates for Q}.

\subsection{Lipschitz estimates}
In this appendix we prove part \ref{workhorse lip part} of Lemma \ref{workhorse lemma for workhorse prop}.
By the same reasoning from the start of Appendix \ref{mapping estimates appendix}, it suffices to find bounds of the form
\begin{multline*}
\norm{h_{c,k}^{\mu}(\etab,a)-h_{c,k}^{\mu}(\grave{\etab},\grave{a})}_{1,\qbar_{\star}(c)}
+\norm{\tell_{c,k}^{\mu}(\etab,a)-\tell_{c,k}^{\mu}(\grave{\etab},\grave{a})}_{1,\qbar_{\star}(c)} \\
\\
\le C(c)\rhs_{\lip}\big(\norm{\etab}_{1,q_{\star}(c)},\norm{\grave{\etab}}_{1,q_{\star}(c)},|a|,|\grave{a}|\big)(\norm{\etab-\grave{\etab}}_{1,\qbar_{\star}(c)}+|a-\grave{a}|), \ k = 1,\ldots, 6,
\end{multline*}
where for $\rho$, $\grave{\rho}$, $a$, $\grave{a}$, $\mu \in \R$ we set
\[
\rhs_{\lip}(\rho,\grave{\rho},a,\grave{a},\mu)
:= |\mu| + |\rho| + |\grave{\rho}| + |a| + |\grave{a}|.
\]

\subsubsection{Lipschitz estimates for $h_{c,1}^{\mu}$ and $\tell_{c,1}^{\mu}$}
This is obvious because these terms are constant in both $\etab$ and $a$.

\subsubsection{Lipschitz estimates for $h_{c,2}^{\mu}$ and $\tell_{c,2}^{\mu}$}
This is obvious because these terms are linear in $\etab$ and come with a factor of $\mu$.

\subsubsection{Lipschitz estimates for $h_{c,3}^{\mu}$}
First we write
\[
h_{c,3}^{\mu}(\etab,a) - h_{c,3}^{\mu}(\grave{\etab},\grave{a})
= \bunderbrace{(a-\grave{a})\D_{\mu}\nl(\varsigmab_c,\phib_c^{\mu}[a])\cdot\e_1}{(a-\grave{a})I}
+ \bunderbrace{\grave{a}\D_{\mu}\nl\big(\varsigmab_c,\phib_c^{\mu}[a]-\phib_c^{\mu}[\grave{a}]\big)\cdot\e_1}{\grave{a}II}
\]
The methods of Appendix \ref{h-c-3-mu mapping appendix}, specifically, the decomposition \eqref{h-c-3-mu mapping decomp} and the estimates \eqref{h-c-3-mu mapping I} and \eqref{h-c-3-mu mapping II}, carry over to yield
\begin{equation}\label{2-I}
\norm{I}_{1,\qbar_{\star}(c)}
\le C(q,c)(|\mu| + |a|).
\end{equation}

Next, we estimate $II$ using techniques similar to those in Appendix \ref{tell-c-3-mu mapping appendix}.
Rewrite
\[
II
=\bunderbrace{ -(2+\mu)(2-A)\left(\varsigma_c \cdot \big(\phi_{c,1}^{\mu}[a] - \phi_{c,1}^{\mu}[\grave{a}]\big)\right)}{-(2+\mu)(2-A)(III)}
-\bunderbrace{\mu\delta\left(\varsigma_c \cdot \big(\phi_{c,2}^{\mu}[a] - \phi_{c,2}^{\mu}[\grave{a}]\big)\right)}{\mu\delta(IV)} .
\]
We will estimate only $III$ explicitly; the estimates on $IV$ are the same.
We have
\[
III
= \bunderbrace{\varsigma_c\cdot\big(\upsilon_c^{\mu}\cos(\omega_c^{\mu}[a]\cdot)-\upsilon_c^{\mu}\cos(\omega_c^{\mu}[\grave{a}]\cdot)\big)}{III_1}
+ \bunderbrace{\varsigma_c\cdot\big(\psi_{c,1}^{\mu}[a](\omega_c^{\mu}[a]\cdot) - \psi_{c,1}^{\mu}[\grave{a}](\omega_c^{\mu}[\grave{a}]\cdot)\big)}{III_2}.
\]
Apply the decay borrowing estimates \eqref{decay borrowing est} to $III_1$ and \eqref{outer inner Lipschitz} to $III_2$ and use the periodic Lipschitz estimates \eqref{periodic theorem lipschitz} and the estimate $\upsilon_{\mu} = \O_c(\mu)$ to conclude
\begin{equation}\label{2-III}
\norm{III_1}_{1,\qbar_{\star}(c)}
\le C(c)|\mu||a-\grave{a}|.
\end{equation}
Combine this estimate and its omitted counterpart on $IV$ with \eqref{2-I} to conclude
\[
\norm{h_{c,3}^{\mu}(\etab,a) - h_{c,3}^{\mu}(\grave{\etab},\grave{a})}_{1,\qbar_{\star}(c)}
\le 
C(c)(|\mu| + |a| + |\grave{a}|)|a-\grave{a}|.
\]

\subsubsection{Lipschitz estimates for $\tell_{c,3}^{\mu}$}
From the definitions of $\tell_{c,3}^{\mu}$ in \eqref{tells}, we compute
\begin{multline*}
\tell_{c,3}^{\mu}(\etab,a)-\tell_{c,3}^{\mu}(\grave{\etab},\grave{a})
= \big(-2\D_{\mu}\nl(\varsigmab_c,\phib_c^{\mu}[a])\cdot\e_2 + a\chi_c\big) - \big(-2\D_{\mu}\nl(\varsigmab_c,\phib_c^{\mu}[\grave{a}])\cdot\e_2 + \grave{a}\chi_c\big) \\
\\
= \bunderbrace{\big(-2\D_{\mu}\nl(\varsigmab_c,\phib_c^{\mu}[a])\cdot\e_2 + a\chi_c^{\mu}\big) - \big(-2\D_{\mu}\nl(\varsigmab_c,\phib_c^{\mu}[\grave{a}])\cdot\e_2 + \grave{a}\chi_c^{\mu}\big)}{I}
+\bunderbrace{(a-\grave{a})(\chi_c-\chi_c^{\mu})}{(a-\grave{a})II}
\end{multline*}
We know from Lemma \ref{chi-c-mu minus chi-c est} that 
\[
\norm{II}_{1,\qbar_{\star}(c)}
\le C(c)|\mu|,
\]
and so we just work on $I$, which, using the definition of $\chi_c^{\mu}$ in \eqref{chi-c-mu defn}, is 
\[
I
= 2\D_{\mu}\nl\big(\varsigmab_c,a(\phib_c^{\mu}[0]-\phib_c^{\mu}[a])-\grave{a}(\phib_c^{\mu}[0]-\phib_c^{\mu}[\grave{a}])\big)\cdot\e_2.
\]
Adding zero, we have
\[
a\big(\phib_c^{\mu}[0]-\phib_c^{\mu}[a])\big)-\grave{a}\big(\phib_c^{\mu}[0]-\phib_c^{\mu}[\grave{a}]\big)
= a\big(\phib_c^{\mu}[\grave{a}]-\phib_c^{\mu}[a]\big) + (a-\grave{a})\big(\phib_c^{\mu}[0]-\phib_c^{\mu}[\grave{a}]\big),
\]
and therefore
\begin{multline*}
\norm{\tell_{c,2}^{\mu}(\etab,a)-\tell_{c,2}^{\mu}(\grave{\etab},\grave{a})}_{1,\qbar_{\star}(c)}
\le \bunderbrace{C|a|\bignorm{\varsigma_c\cdot\big(\phi_{c,2}^{\mu}[\grave{a}]-\phi_{c,2}^{\mu}[a]\big)}_{1,\qbar_{\star}(c)}}{C|a|III} \\
\\
+ \bunderbrace{C|a-\grave{a}|\bignorm{\varsigma_c\cdot\big(\phi_{c,2}^{\mu}[0]-\phi_{c,2}^{\mu}[\grave{a}]\big)}_{1,\qbar_{\star}(c)}}{C|a-\grave{a}|IV}.
\end{multline*}
For both $III$ and $IV$, we use the decay borrowing estimate \eqref{outer inner Lipschitz} and the periodic Lipschitz estimate \eqref{periodic theorem lipschitz} to bound
\[
III \le C|a-\grave{a}|
\quadword{and}
IV \le C|\grave{a}|.
\]

\subsubsection{Lipschitz estimates for $h_{c,4}^{\mu}$ and $\tell_{c,4}^{\mu}$}
We compute
\begin{multline*}
\norm{h_{c,4}^{\mu}(\etab,a)-h_{c,4}^{\mu}(\grave{\etab},\grave{a})}_{1,\qbar_{\star}(c)}
+ \norm{\tell_{c,4}^{\mu}(\etab,a) - \tell_{c,4}^{\mu}(\grave{\etab},\grave{a})}_{1,\qbar_{\star}(c)} \\
\\
= 2\norm{a\D_{\mu}\nl(\phib_c^{\mu}[a],\etab)-\grave{a}\D_{\mu}\nl(\phib_c^{\mu}[\grave{a}],\grave{\etab})}_{1,\qbar_{\star}(c)} \\
\\
\le \bunderbrace{C|a-\grave{a}|\norm{\nl(\phib_c^{\mu}[a],\etab)}_{1,\qbar_{\star}(c)}}{C|a-\grave{a}|I}
+ \bunderbrace{C|\grave{a}|\bignorm{\nl\big(\phib_c^{\mu}[a]-\phib_c^{\mu}[\grave{a}],\etab\big)}_{1,\qbar_{\star}(c)}}{C|\grave{a}|II}
+ \bunderbrace{C|\grave{a}|\norm{\nl(\phib_c^{\mu}[\grave{a}],\etab-\grave{\etab})}_{1,\qbar_{\star}(c)}}{C|\grave{a}|III}.
\end{multline*}
We estimate
\[
I
\le C(c)\norm{\etab}_{1,\qbar_{\star}(c)}
\quadword{and}
III
\le C(c)\norm{\etab-\grave{\etab}}_{1,\qbar_{\star}(c)}
\]
with the product estimate \eqref{Hrq-Wrinfty estimate for Q} on $\nl$ and the periodic bounds \eqref{periodic theorem bounded}.
For $II$, recall that $\etab$, $\grave{\etab} \in E_{\qbar_{\star}(c)}^1 \times O_{\qbar_{\star}(c)}^1 \subseteq E_{q_{\star}(c)}^1 \times O_{q_{\star}(c)}^1$.
We can therefore invoke the decay borrowing estimate \eqref{outer inner Lipschitz} and the periodic Lipschitz estimate \eqref{periodic theorem lipschitz} to conclude
\[
II
\le C|a-\grave{a}|\norm{\etab}_{1,\qbar_{\star}(c)}.
\]

\subsubsection{Lipschitz estimates for $h_{c,5}^{\mu}$ and $\tell_{c,5}^{\mu}$}
These estimates follow immediately from the quadratic estimate \eqref{Hrq product estimate for Q} for $\nl$.

\subsection{Bootstrap estimates}
Since $\H_c^{-1} \in \b(E_{q_{\star}(c),0}^r,E_{q_{\star}(c)}^{r+2})$ and $\L_c^{-1}\P_c \in \b(O_{q_{\star}(c)}^r,O_{q_{\star}(c)}^{r+2})$, we have
\begin{multline*}
\norm{\nano_{c,1}^{\mu}(\etab,a)}_{r+1,q_{\star}(c)}
+ \norm{\nano_{c,2}^{\mu}(\etab,a)}_{r+1,q_{\star}(c)}
\le C(c,r)\sum_{k=1}^5 \norm{h_{c,k}^{\mu}(\etab,a)}_{r,q_{\star}(c)} \\
\\
+ C(c,r)\sum_{k=1}^5 \norm{h_{c,k}^{\mu}(\etab,a)}_{r,q_{\star}(c)}.
\end{multline*}
The individual mapping estimates above show that each of these eleven terms on the right is bounded by $C(c,r)\rhs_{\map}^{\mu}(\norm{\etab}_{r,q_{\star}(c)},a)$, and so we conclude that \eqref{workhorse boot est} also holds.

\begingroup
\footnotesize
\setlength{\parskip}{0pt}
\bibliographystyle{alpha}
\bibliography{equal_mass_bib}
\endgroup
\end{document}